\documentclass[11pt,leqno,english]{amsart}
\usepackage[utf8]{inputenc}
\usepackage[T1]{fontenc}
\usepackage[french, main=english]{babel}
\usepackage{amssymb,latexsym}

\usepackage{mathrsfs,tgschola}

\usepackage{a4}
\usepackage{url}
\usepackage[dvipsnames]{xcolor}
\usepackage{comment}
\definecolor{violet}{rgb}{0.0,0.2,0.7}
\definecolor{rouge2}{rgb}{0.8,0.0,0.2}
\usepackage{tikz}
\usepackage{empheq}
\usepackage{tikz-cd}
\usetikzlibrary{matrix,arrows,decorations.pathmorphing}
\usepackage{hyperref}
\usepackage{mathpazo}

\hypersetup{
    bookmarks=true,         
    unicode=false,          
    pdftoolbar=true,        
    pdfmenubar=true,        
    pdffitwindow=false,     
    pdfstartview={FitH},    
    pdftitle={},    
    pdfauthor={},     
    colorlinks=true,       
   linkcolor=rouge2,          
    citecolor=violet,        
    filecolor=black,      
    urlcolor=cyan}           
\usepackage{enumitem}
\usepackage{appendix}
\usetikzlibrary{decorations.pathreplacing}

 \theoremstyle{plain}    
 \newtheorem{thm}{Theorem}[section]
\theoremstyle{plain} 
\newtheorem{bigthm}{Theorem}

 \numberwithin{equation}{section} 
 \numberwithin{figure}{section} 
 \newtheorem{cor}[thm]{Corollary} 
 \theoremstyle{plain}    
 \newtheorem{prop}[thm]{Proposition} 
 \theoremstyle{plain}    
  \theoremstyle{thm*}    
  \newtheorem*{question*}{Question}
 \newtheorem{lem}[thm]{Lemma} 
 \theoremstyle{remark}
 \theoremstyle{remark}
 \newtheorem{rem}[thm]{Remark}
 \theoremstyle{definition}

\theoremstyle{plain}  

\theoremstyle{plain}

\theoremstyle{definition}

\newcommand{\Q}{{\mathbb{Q}}}
\newcommand{\R}{{\mathbb{R}}}

\newcommand{\cC}{{\mathcal{C}}}

\newcommand{\ie}{i.e.~}

\newcommand{\quotient}[2]{{\left.\raisebox{-.2em}{$#1$}\middle\backslash\raisebox{.2em}{$#2$}\right.}}
\newcommand{\quotientd}[2]{{\left.\raisebox{.2em}{$#1$}\middle\slash\raisebox{-.2em}{$#2$}\right.}}

\def\1{\bold{1}}

\renewcommand{\k}{\kappa}

\newcommand{\gk}{\bar g_\beta}

\newcommand{\om}{\omega}

\newcommand{\omke}{\om_{\rm KE}}
\newcommand{\ome}{\om_{\ep}}

\newcommand{\omkh}{\widehat \om_{\beta}}

\newcommand{\vp}{\varphi}

\newcommand{\ep}{\varepsilon}

\DeclareMathOperator{\Ric}{Ric}

\renewcommand{\d}{\partial}
\newcommand{\dbar}{\overline{\partial}}
\renewcommand{\ge}{\geqslant}
\renewcommand{\le}{\leqslant}
\newcommand{\wt}{\widetilde}
\newcommand{\wh}{\widehat}

\newcommand{\Hom}{\operatorname{Hom}}

\newcommand{\vol}{\operatorname{vol}}
\newcommand{\Vol}{\operatorname{Vol}}

\newcommand{\tr}{\operatorname{tr}}

\newcommand{\oX}{\overline X}
\newcommand{\oXm}{\overline X_{\rm min}}

\title[Degenerating Kähler-Einstein cones]{Degenerating Kähler-Einstein cones, locally symmetric cusps, and the Tian-Yau metric}
\date{\today}

\author{Olivier Biquard}
\address{Sorbonne Université and Université de Paris, CNRS, IMJ-PRG, F-75005 Paris, France}
\email{olivier.biquard@sorbonne-universite.fr}

\author{Henri Guenancia}
\address{Institut de Mathématiques de Toulouse; UMR 5219, Université de Toulouse; CNRS, UPS, 118 route de Narbonne, F-31062 Toulouse Cedex 9, France}
\email{henri.guenancia@math.cnrs.fr}
\date{}

\begin{document}

\begin{abstract}  
Let $X$ be a complex projective manifold and let $D\subset X$ be a smooth divisor. In this article, we are interested in studying limits when $\beta\to 0$ of Kähler-Einstein metrics $\omega_\beta$ with a cone singularity of angle $2\pi \beta$ along $D$. In our first result, we assume that $X\setminus D$ is a locally symmetric space and we show that $\omega_\beta$ converges to the locally symmetric metric and further give asymptotics of $\omega_\beta$ when $X\setminus D$ is a ball quotient. Our second result deals with the case when $X$ is Fano and $D$ is anticanonical. We prove a folklore conjecture asserting that a rescaled limit of $\omega_\beta$ is the complete, Ricci flat Tian-Yau metric on $X\setminus D$. Furthermore, we prove that $(X,\omega_\beta)$ converges to an interval in the Gromov-Hausdorff sense. 
\end{abstract}

\maketitle

\tableofcontents

\section*{Introduction}

Let $X$ be a complex projective manifold and let $D\subset X$ be a smooth divisor. In many geometrically meaningful situations, one is able to construct Kähler-Einstein metrics $\omega$ on the complement $X^\circ:=X\setminus D$ of the divisor $D$. Unless one imposes some growth condition near $D$, such a metric $\omega$ may not be unique \--- one can typically find complete and incomplete KE metrics on the same $X^\circ$, possibly with the same Einstein constant too. 

The focus of the present paper is to investigate the relationship between these different metrics in the specific setting of Kähler-Einstein metrics with {\it cone singularities} along $D$. Recall that if $\beta\in (0,1)$, a Kähler metric $\omega$ on $X^\circ$ is said to have cone singularities along $D$ with cone angle $2\pi \beta$ if it is locally quasi-isometric to the model cone metric \[\omega_{\beta,{\rm mod}}:=\frac{idz_1\wedge d\bar z_1}{|z_1|^{2(1-\beta)}}+\sum_{j\ge 2} idz_j\wedge d\bar z_j\] on each coordinate chart $(U, (z_i))$ where $U\cap D=(z_1=0)$. Such a metric is incomplete, has finite volume and automatically extends to a closed, positive $(1,1)$-current on $X$. There is an analogue of the Aubin-Yau (resp. Yau) theorem guaranteeing the existence and uniqueness of a negatively curved (resp. Ricci-flat) Kähler-Einstein metric $\omega_\beta$ with cone angle $2\pi \beta$ along $D$ under the condition that the adjoint $\R$-line bundle $K_X+(1-\beta)D$ is ample (resp. numerically trivial), cf e.g. \cite{Brendle, CGP, GP, JMR}. 
The positive curvature case is more complicated, in analogy with the absolute case $D=\emptyset$ and it involves the properness of some suitable analogue of the Mabuchi or Ding functional. \\

Let us now shift our focus to the small angle regime, that is when $0<\beta\ll 1$. We raise the following broad and somewhat vague question, which is closely related to \cite[Conjecture~1.11]{CR15} and \cite[Conjecture~1.4]{Odaka20} in the positive curvature case.

\begin{question*}
Let $X$ be a complex projective manifold and let $D\subset X$ be a smooth divisor. Assume that for any $0<\beta\ll 1$, there exists a unique Kähler-Einstein metric $\omega_\beta$ with cone angle $2\pi \beta$ along $D$, i.e. 
\begin{equation}
\label{eq KE omk}
\Ric \omega_\beta = \sigma \omega_\beta+(1-\beta)[D]
\end{equation}
for some $\sigma=\pm 1$. Do the metrics $\omega_\beta$ converge when $\beta\to 0$ to some canonical metric on $X^\circ$, possibly after rescaling?
\end{question*}

The aim of this paper is to provide an answer to the above question in two different geometric situations, one for each sign of the curvature.\\

{\bf The negative case.} 

\noindent
As recalled above, the existence of a KE metric solving \eqref{eq KE omk} with $\sigma =-1$ is equivalent to $K_X+(1-\beta)D$ being ample. For instance, if one assumes that $K_X+D$ is ample, then the same will hold true for $K_X+(1-\beta)D$ as long as $\beta$ is small enough. In that situation, it was proved in \cite{C2C} that when $\beta\to 0$, the KE metric $\omega_\beta$ converges to the complete KE metric with Poincaré growth constructed by R. Kobayashi \cite{KobR} and Tian-Yau \cite{TY87}. 

Another interesting example is provided by toroidal compactifications of ball quotients $X^\circ=\quotient{\Gamma}{\mathbb B^n}$, where $\Gamma \subset \mathrm{Aut}(\mathbb B^n)$ is a torsion-free, discrete subgroup. It is well-known that one can embed $X^\circ \hookrightarrow X$ as a Zariski-open subset of a projective orbifold $X$ such that $D:=X\setminus X^\circ$ is a disjoint union of abelian varieties,  cf \textsection~\ref{setup0} for references and more details. Note that the Bergman metric on $\mathbb B^n$ descends to the complex (complete) hyperbolic metric $\om_{\rm hyp}$ on $X^\circ$, which we normalize to have $\Ric \om_{\rm hyp}=-\om_{\rm hyp}$. Moreover, $K_X+(1-\beta)D$ is ample for $0<\beta\ll 1$ (but certainly not for $\beta=0$ unless $D=\emptyset$) and therefore $X^\circ$ also comes equipped with KE metrics $\omega_\beta$ with cone angle $2\pi \beta$ along $D$ whenever $\beta>0$ is small enough. The relationship between these metrics is provided by the following

\begin{bigthm}
\label{thmA}
Let $(X,D)$ be a toroidal compactification of a ball quotient $X^\circ=\quotient{\Gamma}{\mathbb B^n}$, and let $\omega_\beta$ be the KE metric solving \eqref{eq KE omk} for small $\beta$, with $\sigma=-1$. Then, we have convergence
\[\omega_\beta \underset{\beta\to 0}{\longrightarrow} \om_{\rm hyp}\]
both in  $C^{\infty}_{\rm loc}(X^\circ)$ and weakly as currents on $X$. Moreover, we have precise asymptotics of $\omega_\beta$ near $D$ when $\beta\to 0$.
\end{bigthm}

A few remarks are in order here. 
\begin{enumerate}[label=$(\roman*)$]
\item The asymptotics of $\omega_\beta$ in $C^0$ are given in Theorem~\ref{KE asymptotics}. They are obtained by constructing a model metric on the normal bundle of $D$ using the Calabi Ansatz, cf \textsection~\ref{sec:calabi-metric}. 
\item The first half of the statement (i.e. the convergence part) remains true in the more general setting of quotients of bounded symmetric domains, cf~Theorem~\ref{weak convergence}. In that case, $D$ needs not be smooth anymore but has simple normal crossings up to a finite cover. 
\item Assume that the lattice $\Gamma$ is arithmetic. By choosing the angles carefully along each torus at the boundary, one can find a sequence $\om_{\beta_m}$ of orbifold KE metrics that can be globally desingularized so that $(X^\circ, \om_{\rm hyp})$ is the limit of {\it smooth, compact} KE spaces up to the action of a larger and larger group of isometries, cf \textsection~\ref{section covers}. In a nutshell, one can "close the complex hyperbolic cusp".  This gives the closest analog to the Dehn filling of real hyperbolic cusps by Einstein manifolds, due to Thurston in dimension 3 and Anderson \cite{And06} in higher dimension: in the complex case, one cannot fill the cusp, but this is possible up to some larger and larger covering. This answers a question of Misha Kapovich to the first author several years ago.
\end{enumerate}

\bigskip

{\bf The positive case.}

\noindent
In general, it is not so easy to characterize the existence of a metric $\omega_\beta$ solving \eqref{eq KE omk} with $\sigma=1$, for small values of $\beta$. However, a result of Berman \cite{rber} (later generalized by Song-Wang \cite{SW}) asserts that if $X$ is a Fano manifold (that is, $-K_X$ is ample) and $D\in |-K_X|$ is smooth, then there exists $\beta_0>0$ such that for any $0<\beta<\beta_0$, there exists a unique Kähler metric $\omega_\beta$ on $X^\circ$ such that $\Ric \omega_\beta=\omega_\beta$ and $\omega_\beta$ has cone singularities with cone angle $2\pi \beta$ along $D$, i.e. $\omega_\beta$ solves \eqref{eq KE omk}. 

The existence of such a metric had been conjectured by Donaldson \cite[\textsection~6]{Don} in relation with his program to prove that a $K$-stable Fano manifold admits a Kähler-Einstein metric by using the continuity path $\Ric \om_t = t\om_t+(1-t)[D]$ involving metrics with cone singularities. He also predicted that the (conjectural then) $\omega_\beta$ would actually converge to the Ricci flat complete Kähler metric $\om_{\rm TY}$ constructed by Tian and Yau in \cite{TY1}. 

If $n=1$, then the metrics $\omega_\beta$ on $\mathbb P^1\setminus \{0,\infty\}$ are completely explicit, given by the expression $\omega_\beta=\frac{\beta^2idz \wedge d\bar z}{|z|^{2(1-\beta)}(1+|z|^{2\beta})^2}$ and one sees immediately that $\beta^{-2}\omega_\beta$ converges locally smoothly to the cylinder $\om_{\rm cyl}=\frac{ idz \wedge d\bar z}{4|z|^2}$  while $(\mathbb P^1, \omega_\beta)$ converges in the Gromov-Hausdorff sense to the interval $([0,\frac \pi 2],dt^2)$ (set $r=|z|^\beta$ to that $g_\beta=\frac{dr^2+\beta^2r^2d\theta^2}{(1+r^2)^2}$ and reparametrize by $t=\tan^{-1}(r)$), cf also \cite{RubiZhang20}. Our second main result establishes the conjecture in full generality.

\begin{bigthm}
\label{thmB}
Let $X$ be a Fano manifold of dimension $n$ and let $D\in |-K_X|$ be a smooth anticanonical divisor. Then up to a rescaling factor, the conic KE metrics $\omega_\beta$ solving \eqref{eq KE omk} with $\sigma =1$ for small $\beta$ converge to the Tian-Yau metric: 
\[\beta^{-1-\frac 1n} \omega_\beta \underset{\beta\to 0}{\longrightarrow} \om_{\rm TY}\]
 in  $C^{\infty}_{\rm loc}(X\setminus D)$. Moreover, we have precise asymptotics of $\omega_\beta$ near $D$ when $\beta\to 0$. 

 Finally fix a point $p\in D$, denote $g_\beta$ the Riemannian metric associated to the Kähler form $\omega_\beta$ and consider the renormalized volume forms $\nu_\beta=\frac{d\vol_{g_\beta}}{\vol_{g_\beta}(X)}$. Then the spaces $(X,g_\beta,p,\nu_\beta)$ converge in the measured Gromov-Hausdorff sense to the interval
 \[
 \big([0,\tfrac \pi2],g_\infty=\tfrac2{n+1}ds^2,0,\nu_\infty= d(-\cos^{\frac{2n}{n+1}}s)\big).
\]
\end{bigthm}

As before, a few remarks:
\begin{enumerate}[label=$(\roman*)$]
\item  The fibers of the collapsing to an interval are the normal circle bundle of the divisor $D$.  The two endpoints of the interval correspond respectively to the conical divisor $D$ itself and to the Tian-Yau metric.  Over interior points of the interval, the fibres have two speeds of collapsing: speed $\beta$ for the circle directions and $\sqrt \beta$ for the divisor directions. See section~\ref{sec:geometry} for a detailed discussion, a relevant picture, as well as the computation of the other possible nontrivial Gromov-Hausdorff limits of the rescaled metrics $g_\beta$: we obtain $\mathbb{R}_+$, $\mathbb{R}_+\times D$ and $\mathbb{R}_+\times \mathbb{C}^{n-1}$ at $s=0$, and at $s=\frac \pi2$ the previous Tian-Yau metric on $X\setminus D$ and $\Bbb{R}_+$.

\item  Several recent papers study cases of collapsing of Ricci flat Kähler metrics to an interval, for K3 surfaces \cite{HSVZ} or in higher dimension \cite{SZ19}. Our theorem probably gives the first general example of collapsing of Kähler-Einstein metrics with positive Ricci: of course this is made possible by the presence of a cone angle going to zero.

\item In the process of the proof, we construct the Kähler-Einstein metrics $\omega_\beta$ for small $\beta$, therefore recovering Berman's result. 
\end{enumerate}

\bigskip

{\bf Strategy of the proof.}

Although Theorem~\ref{thmA} and Theorem~\ref{thmB} have a quite different flavor, their proofs share a common approach. Indeed, in both proofs, we rely on the existence of a model metric living in the neighborhood of the zero section in the normal bundle of $D$. That metric is provided by the Calabi Ansatz, cf \textsection~\ref{sec:calabi-metric}, and its curvature is computed in the following section, \textsection~\ref{sec:curv-calc}.

The proof of Theorem~\ref{thmA} goes as follows. We use pluripotential methods, especially the comparison principle (quite suited in negative curvature), in order to estimate the potential of $\omega_\beta$ with enough precision to establish the weak convergence. The local smooth convergence away from $D$ follows from a suitable use of Chern-Lu formula. In order to further compute the asymptotics of $\omega_\beta$ (at order zero), we show that $\omega_\beta$ is asymptotically close to the Calabi metric constructed and analyzed in \textsection~\ref{sec:calabi-metric}. This relies on the previous step as well as the application of the maximum principle and Chern-Lu formula, which in turn uses crucially that the curvature of the Calabi metric is bounded, cf~\textsection~\ref{sec:curv-calc}.

The proof of Theorem~\ref{thmB}, technicallly more involved than the previous one, relies on gluing methods. The general idea is to construct a model cone metric $\tilde \omega_\beta$ by gluing the Calabi metric $\om_{\beta,L}$ near $D$ and the Tian-Yau metric $\om_{TY}$ away from $D$. For $\beta\ll 1$,  the implicit function theorem allows us to find the Kähler-Einstein metric $\omega_\beta=\tilde \omega_\beta+dd^c \varphi_\beta$ with a control on $\varphi_\beta$ and its covariant derivatives that is sufficiently precise that one can derive the desired smooth convergence $\beta^{-1-\frac 1n}\omega_\beta \to \om_{TY}$ away from $D$ as well as the global Gromov-Hausdorff convergence of $(X,\omega_\beta)$ to an interval. 

Some of the main technical steps include: estimating the curvature of the Calabi metric; finely gluing the Calabi metric which lives on the normal bundle $L$ of $D$ onto a neighborhood of $D$ in $X$ using a fibration in extremal disks; establishing a Schauder estimate for the model cone metric on $\mathbb C^*\times \mathbb C^{n-1}$ with cone angle $2\pi \beta$ which is uniform in $\beta$; establishing a uniform Schauder estimate in suitable weighted Hölder spaces for the family of collapsing cone metrics $\tilde \omega_\beta$ mentioned above. This is similar in spirit to other gluing problems, especially the papers \cite{HSVZ,SZ19} mentioned above, but our techniques are different. 

 Applying the techniques used for Theorem~\ref{thmB} to Theorem~\ref{thmA} would probably enhance our $C^0$ estimates for the metrics to estimates on all derivatives, at the expense of a much more technical proof. On the other hand, the pluripotential techniques seem to fall short in the context of Theorem~\ref{thmB}.

\subsection*{Acknowledgements}
O.B. would like to thank Misha Kapovich for discussions a long time ago about closing complex hyperbolic cusps.
H.G. would like to thank Benoît Cadorel for the many insightful discussions about toroidal compactifications of quotients of bounded symmetric domains. The authors would like to thank the referee for reading the manuscript carefully and for the several suggestions that helped us improve the paper.

H.G has benefited from the support of the ANR project GRACK as well as from the state aid managed by the ANR under the "PIA" program bearing the reference ANR-11-LABX-0040, in connection with the research project HERMETIC.

\section{Closing the cusps of locally symmetric spaces}

\subsection{Setup.} 
\label{setup0}
Let $X=\quotient{\Gamma}{\Omega}$ be an $n$-dimensional quotient of a bounded symmetric domain $\Omega$ by a torsion-free lattice $\Gamma \subset \mathrm{Aut}(\Omega)^\circ$. 
 It is well-known that $X$ is a quasi-projective variety that can be compactified in several meaningful ways. 

The Satake-Baily-Borel compactification $X\hookrightarrow \oX_{\rm min}$ is a \textit{singular}, minimal compactification in the sense that given any normal compactification $X\hookrightarrow \oX'$, the identity morphism on $X$ extends to a holomorphic map $\oX'\to \oX_{\rm min}$. The variety $\oX_{\rm min}$ is normal, has log canonical singularities and $K_{\oX_{\rm min}}$ is ample. In a modern terminology, $\oX_{\rm min}$ is a (normal) stable variety. 

The Ash-Mumford-Rapoport-Tai \cite{AMRT} toroidal compactification $X\hookrightarrow \oX$ is a compactification with finite quotient singularities such that $\oX\setminus X$ is a (reduced) divisor with simple normal crossings that we denote by $D=\sum_{\lambda=1}^N D_\lambda$. Moreover, the birational morphism 
$\pi:\oX \to \oX_{\rm min}$ satisfies
$$K_{\oX}+D=\pi^*K_{\oX_{\rm min}}.$$

If $\Gamma$ is neat, the toroidal compactification $\overline X$ is actually smooth. Moreover, any torsion-free lattice  $\mathrm{Aut}(\Omega)$ admits a finite index subgroup which is neat. As a result, one can find $\Gamma'<\Gamma$ with finite index such that $Y=\quotient{\Gamma'}{\Omega}$ admits a smooth toroidal compactification $(\overline Y,D')$. Moreover, the finite étale morphism $f:Y\to X$ extends uniquely to a finite cover $f:\overline Y \to \overline X$ and one has $K_{\overline Y}+D'=f^*(K_{\oX}+D)$. 

In the case where $\Omega=\mathbb B^n$ is the euclidean unit ball in $\mathbb C^n$, $\oXm\setminus X$ consist of  finitely many singular points $\{x_1, \ldots, x_N\}$ and $D=\sqcup_{\lambda=1}^N D_\lambda$ is a disjoint union of abelian varieties $D_\lambda$ with negative normal bundle, which are contracted onto those singular points by $\pi$. \\

\subsection{The Kähler-Einstein metric.} 
\label{subsec KE}

The Bergman metric $\om_{\rm Berg}$ on $\Omega$ is invariant under the action of  $\mathrm{Aut}(\Omega)$ hence it descends to a Kähler-Einstein metric $\omke$ on $X$. Moreover, one can prove that $\omke$ extends to a closed, positive current $\omke \in c_1(K_{\oX}+D)$ and 
$$\int_{X} \omke^n=c_1(K_{\oX}+D)^n.$$
In particular, $\omke=\pi^*\omega_{\rm min}$ for some closed, positive current $\om_{\rm min} \in c_1(K_{\oXm})$ which coincides with the singular Kähler-Einstein metric constructed in \cite{BG}. 

If $\Gamma'<\Gamma$ has finite index, then the Kähler-Einstein metric of $Y=\quotient{\Gamma'}{\Omega}$ is simply $f^*\omke$ where $f:Y\to X$ is the finite étale cover induced by the lattice inclusion.

In the case where $\Omega=\mathbb B^n$ and $\Gamma$ is neat, one has a very precise description of $\omke$ near $D$, cf. e.g. \cite[Eq.~(8)]{Mok12}. In particular, if $(z_1, \ldots, z_n)$ is a system of holomorphic coordinates on some open set $U\subset \oX$ such that $D\cap U=(z_1=0)$, then $\om|_{U}$ is quasi-isometric to 
\begin{equation}
\label{model}
\frac{idz_1\wedge d\bar z_1}{|z_1|^2(-\log |z_1|)^2}+\frac{1}{(-\log |z_1|)}\sum_{k=2}^n idz_k\wedge d\bar z_k.
\end{equation}
One can actually say much more and exhibit an exact formula for $\omke$ on a small enough neighborhood $U$ of $D$ after identifying $U$ with a neighborhood of the zero section in the normal bundle $N_{D/X}\to D$ of $D$, cf~\eqref{hyp}.

\subsection{Monge-Ampère equation.} 

\noindent
One can write down the Monge-Ampère equation satisfied by $\omke$. In order to do so, we pick:

$\cdotp$ A Kähler metric $\om_{\oXm} \in c_1(K_{\oXm})$ and set $\chi:=\pi^*\om_{\oXm}$. It is a smooth, semipositive form on $\oX$. Recall that a Kähler metric on a singular complex space $Y$ is defined to be a Kähler metric on the regular locus $Y_{\rm reg}$ which is locally the restriction of an ambient Kähler form under local embeddings $Y \underset{\rm loc}{\hookrightarrow} \mathbb C^N$. In particular, its bisectional curvature is bounded above locally near any point. 

$\cdotp$  Holomorphic sections $s_\lambda\in H^0(\oX,\mathcal O_{\oX}(D_\lambda))$ such that $D_\lambda=(s_\lambda=0)$ and smooth hermitian metrics $h_\lambda$ on $\mathcal O_{\oX}(D_\lambda)$. with Chern curvature form $\theta_\lambda:=i\Theta_{h_\lambda}(D_\lambda)$. We set $s=\bigotimes s_\lambda$, $|s|:=\prod_k |s_\lambda|_{h_\lambda}$, $\theta=\sum_k \theta_\lambda$. Up to scaling $h_\lambda$, one can assume that $|s_\lambda|_{h_\lambda}<e^{-1}$.

$\cdotp$ A smooth volume form $dV$ on $\oX$ satisfying  $-\mathrm{Ric}(dV)+\theta=\chi$. 
Then one can write the Kähler-Einstein metric $\omke$ on $\oX$ as $\omke=\chi+dd^c\wh \vp$ for the unique $\chi$-psh function $\wh \vp$ solution of the (non-pluripolar) Monge-Ampère equation
\begin{equation}
\label{KE 3}
\omke^n=(\chi+dd^c\wh \vp)^n=\frac{e^{\wh \vp} dV}{|s|^2}.
\end{equation} 
One knows that for any $\ep>0$, there exists a constant $C_\ep$ such that the following set of inequalities
\begin{equation}
\label{ineq}
C_1 \ge\wh \vp\ge -(n+1+\ep) \log (-\log |s|) -C_{\ep}
\end{equation}
hold on $\oX$, cf. \cite[Prop.~D]{DGG}.\\

\subsection{Conic approximation.}

As $\pi$ can be obtained as a sequence of blow ups of smooth centers, there exist coefficients $a_\lambda \in \mathbb Q_+$ such that $-\sum a_\lambda D_\lambda$ is $\pi$-ample. In particular, for $\beta>0$ small enough, the $\Q$-line bundle $K_X+\sum_\lambda (1-\beta a_\lambda)D_\lambda$ is ample. Moreover, up to scaling down the $a_\lambda$ (by the same factor), one can assume that $\chi-\wt \theta$ is a Kähler form, where $\widetilde \theta:= \sum a_\lambda \theta_\lambda$. The Monge-Ampère equation 
\begin{equation}
\label{MAep}
(\chi- \beta \widetilde \theta+dd^c \wh\vp_\beta)^n=\frac{e^{\wh\vp_\beta} dV}{\prod_\lambda |s_\lambda|_{h_\lambda}^{2(1-\beta a_\lambda)}}
\end{equation}
has a unique solution $\wh\vp_{\beta}\in L^{\infty}(\oX)\cap \mathrm{PSH}(\oX,\chi-\beta \widetilde \theta)$ by \cite{Kolo}. Moreover, it is well-known that 
$$\omkh:=\chi-\beta \wt\theta+dd^c \wh\vp_\beta$$ 
is smooth outside $D$, has conic singularities along each $D_\lambda$ with cone angle $2\pi \beta a_\lambda$ (say if $\oX$ is smooth, otherwise this will be true only after a finite cover) and one has $\Ric \omkh= -\omkh$ on $X=\oX\setminus D$, cf e.g. \cite{GP}. In particular, we have as currents
\begin{equation}
\label{KE current}
\Ric \omkh=-\omkh+ \sum_{\lambda=1}^N (1-\beta a_\lambda) [D_\lambda]. 
\end{equation}


\subsection{Main result}
The aim of this section is to prove the following result. 

\begin{thm}
\label{weak convergence}
The conic Kähler-Einstein metrics $\omkh$  solution of \eqref{KE current} converge to $\omke$ when $\beta \to 0$, both weakly as currents on $\oX$ and locally smoothly on $X$. 
\end{thm} 

\begin{proof}
We divide the proof in three steps. In the first two steps, we assume that $\Gamma$ is neat so that $\oX$ is a smooth manifold. In the last step, we will explain how to work with the finite quotient singularities that $\oX$ has in general. \\

\noindent
\textbf{Step 1. Weak convergence.}

\noindent
Let $\beta\in [0,\frac 12]$ and let $\tau_\beta\in (0,1]$ be a number to be determined later. We set 
$$\wh\psi_{\beta}:=\frac{1}{1-\beta}\cdot ( \wh\vp_\beta+\tau_\beta);$$
this is a $\frac{1}{1-\beta}\cdot (\chi-\beta \wt \theta)$-psh function with finite energy (even bounded if $\beta>0$) satisfying the Monge-Ampère equation
\begin{equation}
\label{MA}
\left(\frac{1}{1-\beta}\cdot (\chi-\beta \wt \theta)+dd^c \wh\psi_\beta\right)^n = e^{(1-\beta)\wh\psi_\beta+F_\beta} \cdot \frac{dV}{|s|^2_h}
\end{equation}
where $F_\beta=\beta \sum_\lambda  a_\lambda \log |s_\lambda|^2_{h_\lambda}-n\log(1-\beta)-\tau_\beta$.
 We claim that for any $\beta'>\beta$ small enough, one has
\begin{equation}
\label{psh}
\frac{1}{1-\beta'}\cdot (\chi-\beta' \wt \theta)+dd^c\wh \psi_\beta \ge 0.
\end{equation}
This follows from the identity
\begin{equation}
\label{comparaison psh}
\frac{1-\beta}{1-\beta'}\cdot (\chi-\beta' \wt \theta) =  (\chi-\beta \wt \theta) +\frac{\beta'-\beta}{1-\beta'}\cdot \underbrace{(\chi - \wt \theta)}_{>0}.
\end{equation}
More precisely, we get
\begin{align*}
\left(\frac{1}{1-\beta'}\cdot (\chi-\beta' \wt \theta)+dd^c\wh \psi_\beta \right)^n  & \ge (1-\beta)^{-n} \left(\frac{1-\beta}{1-\beta'}\cdot (\chi-\beta' \wt \theta)+dd^c\wh \vp_\beta\right)^n \\
& \ge (1-\beta)^{-n}(\chi - \beta \wt \theta +dd^c \wh\vp_\beta)^n\\
&=e^{(1-\beta') \wh\psi_\beta+F_{\beta'}+H_{\beta',\beta}}\cdot \frac{dV}{|s|^2_h}
\end{align*}
where $H_{\beta',\beta}=(\beta'-\beta)(\wh\psi_\beta-\log |s|^2)-n \log\left( \frac{1-\beta}{1-\beta'}\right)-\tau_\beta+\tau_{\beta'}$. 

If we first choose $\beta=0$, $\tau_{\beta'}= C\beta'-n\log(1-\beta')$ where $C>0$ is a constant such that $\vp \ge \log |s|^2_h-C$, whose existence is guaranteed by \eqref{ineq}, then we see that $H_{\beta',0}\ge 0$ so that $\wh\vp=\wh\psi_{0}$ is a subsolution of \eqref{MA}, hence the comparison principle yields
\begin{equation}
\label{zero}
\frac{1}{1-\beta}\cdot \wh\vp_\beta +\frac{C\beta-n\log(1-\beta)}{1-\beta} \ge\wh \vp 
\end{equation}
Using the inequality above, we conclude that there is a constant $C'>0$ such that $\wh\psi_\beta\ge \log |s|^2-C'$ for any $\beta$. Then we set $\tau_{\beta}:=C'\beta-n\log(1-\beta)$ and it follows that $H_{\beta',\beta} \ge 0$. In other words, $\wh\psi_{\beta}$ is a subsolution of the Monge-Ampère equation satisfied by $\hat \psi_{\beta'}$, hence 
\begin{equation}
\label{un}
\wh\psi_{\beta'} \ge \wh\psi_{\beta}.
\end{equation}
The family $(\wh\psi_{\beta})_{\beta>0}$ is a decreasing family of quasi-psh functions with complex Hessian uniformly bounded from below. It follows that they converge when $\beta$ approaches zero to a $\chi$-psh function $\wt \vp$. It follows from \eqref{zero} that $\wt \varphi$ has finite energy, is locally bounded on $X$ has satisfies 
$$(\chi+dd^c \wt \vp)^n=\frac{e^{\wt \vp}dV}{|s|^2}$$
on $X$ by Bedford-Taylor theory, hence also globally on $\oX$. By uniqueness of such a solution (cf e.g. \cite[Prop.~4.1]{BG}), we get $\widetilde \vp =\wh \vp$, which proves the first part of the proposition. \\

\noindent
\textbf{Step 2. Smooth convergence locally on $X$.}

\noindent
We apply Chern-Lu inequality to the identity map from $(X,\omkh)$ to $(X,\chi)$, cf e.g. \cite[Proposition~7.1]{Rub14}. As $\Ric \, \omkh = -\omkh$ and the bisectional curvature of $(X,\chi)$ is bounded from above, there is a constant $A>0$ such that
\begin{equation}
\label{CL}
\Delta_{\omkh} \log \tr_{\omkh} \chi \ge -A(1+\tr_{\omkh} \chi).
\end{equation}
Next, we have 
\begin{align*}
\Delta_{\omkh}(-\wh\vp_\beta) &= \tr_{\omkh}\chi - n -\ep \tr_{\omkh}\wt \theta \\
& = \tr_{\omkh}\chi - n +\ep [\tr_{\omkh}(\chi-\wt \theta)-\tr_{\omkh}\chi] \\
&\ge (1-\beta)\tr_{\omkh}\chi - n 
\end{align*}
and, if $|\wt s|^2:=\prod |s_\lambda|_{h_\lambda}^{2a_\lambda}$, 
\begin{align*}
\Delta_{\omkh}\log |\wt s|^2 &= \tr_{\omkh}(-\wt \theta) \\
& =\tr_{\omkh}(\chi-\wt \theta)-\tr_{\omkh}\chi \\
&\ge-\tr_{\omkh}\chi.
\end{align*}
Fix some number $\delta\in (0,\frac 14)$; it follows from the previous inequalities that the following holds on $X$
$$\Delta_{\omkh} \big[\log \tr_{\omkh} \chi-(A+1)\wh \vp_\beta +\delta \log |\wt s|^2 \big]\ge (1-(A+1)\beta-\delta)\tr_{\omkh} \chi-B $$
where $B=(n+1)A+n$. Up to decreasing $\beta$, one can assume without loss of generality that $(A+1)\beta \le 1/2$ so that $(1-(A+1)\beta-\delta) \ge \frac 14$ . Set $H_{\beta, \delta}:=\log \tr_{\omkh} \chi-(A+1) \wh\vp_\beta +\delta \log |\wt s|^2 $; it is a smooth function on $X$ which tends to $-\infty$ near $D$ thanks to \eqref{ineq}-\eqref{zero}. At its maximum $x_{\beta, \delta}\in X\setminus D$, one has $\tr_{\omkh} \chi(x_{\beta, \delta})\le 4B$. Therefore, one has, for any $x\in X$ 
\begin{align*}
\log \tr_{\omkh} \chi (x) &= H_{\beta, \delta}(x)+(A+1)\wh\vp_\beta(x)-\delta \log |\wt s|^2(x) \\
&\le H_{\beta, \delta}(x_{\delta,\beta})+(A+1)\wh\vp_\beta(x)-\delta \log |\wt s|^2(x) \\
& \le \log 4B+ [\delta \log |\wt s|^2 -(A+1)\wh \vp_\beta](x_{\beta, \delta})+(A+1)\wh\vp_\beta(x)-\delta \log |\wt s|^2(x)\\
& \le C_{\delta}- \delta \log |\wt s|^2(x)
\end{align*}
where we used the fact that $\wh\vp_{\beta}$ is uniformly bounded above (e.g. \eqref{un}) and $\wh\vp_\beta \ge -(n+2) \log(-\log |s|) +O(1)$ by \eqref{zero}. As a result, we get 
$$\chi \le \frac{C_{\delta}}{|\wt s|^{2\delta}} \cdot \omkh$$ 
uniformly on $X$, for any $\delta \in (0,1/4)$. From the Monge-Ampère equation satisfied by $\omkh$ and the fact that $\chi$ is a smooth Kähler form on $X$, we deduce that given any compact subset $V\Subset X$, there is a constant $C_V$ independent of $\beta$ such that 
$$\sup_V |\Delta_{\chi} \wh\vp_{\beta} | \le C_V.$$ 
Using standard bootstrapping arguments, we get uniform bounds on the higher derivatives of $\vp_\beta$ on compact subsets of $X$, which ends the proof of the theorem in the case where $\Gamma$ is neat.\\

\noindent
\textbf{Step 3. General case when $\Gamma$ is not neat.}

\noindent
Let $\Gamma'<\Gamma$ be a neat, finite index sub-lattice. The quotient $Y=\quotient{\Gamma'}{\Omega}$ admits a smooth toroidal compactification $(\overline Y,D')$, and let $f:\overline Y \to \overline X$ be the associated finite cover. Set $m_\lambda$ to be the ramification order of $f$ along $D_\lambda$, and set $D_\lambda'=f^{-1}(D_\lambda)$. In summary, one has \[K_{\overline Y}+\sum_\lambda (1-\beta a_\lambda m_\lambda)D_\lambda'=f^*(K_{\oX}+\sum_\lambda (1-\beta a_\lambda)D_\lambda).\] 
The Kähler-Einstein metrics $\wh \omega_\beta':=f^*\wh \omega_\beta$ have cone singularities along $D'$ with cone angle $2\pi(1-\beta a_\lambda m_\lambda)$ along $D'_\lambda$. That is, they satisfy $\Ric \wh \omega_\beta'=-\wh \omega_\beta'+\sum_\lambda(1-\beta a_\lambda m_\lambda)[D_\lambda']$. Thanks to Steps 1-2 above, $\wh \omega_\beta'$ converge globally weakly on $\overline Y$ and locally smoothly on $Y$ towards the hyperbolic metric $Y$. Of course, one needs to perform a harmless adjustment by replacing $\chi$ with $f^*\chi$.  Since the hyperbolic metric on $Y$ is nothing but $f^*\omke$, cf. \textsection~\ref{subsec KE}, the theorem follows immediately. 
\end{proof}

\section{The Calabi ansatz}
\label{sec:calabi-metric}

We now construct some explicit model Kähler metrics in the total space of a holomorphic line bundle $L$ over $D$. This technique goes back to Calabi \cite{Cal79}. \\

\noindent
{\bf Model Setup.} Let $D$ be a compact Kähler manifold equipped with a Kähler form $\theta_D$ and let $(L,h)$ be a Hermitian holomorphic line bundle over $D$. We make the following assumptions: 
\begin{enumerate}[label=(\roman*)]
\item $\Ric(\theta_D)=0$.
\item   \label{eq:1}  $ i\Theta(L,h)= \sigma \theta_D, \qquad \sigma=\pm 1$.
\end{enumerate}

We think of $L$ as the normal bundle of $D$ which will be a divisor in a compact complex manifold $X$. The case $\sigma=-1$ corresponds to a quotient of a ball: in that case $D$ will be a torus. The case $\sigma=1$ corresponds to that of an anticanonical divisor in a Fano manifold $X$.

We consider the function $t=\log \|v\|_h^2$ defined on $L\setminus D$, the complement of the zero section in the total space of $L$. We also have on $L\setminus D$ a connection 1-form $\eta$ which coincides on each fibre of $L$ with the angular form $d\theta$, and satisfies 
\[d\eta=-i p^*\Theta(L,h)=-\sigma p^*\theta_D,\] where $p$ is the projection $p:L\rightarrow D$. Then $\xi=\frac12 dt+i\eta$ is a (1,0)-form on $L\setminus D$, coinciding with $\frac{dz}z$ in each fibre. In particular, $dt\wedge \eta = i \xi \wedge \bar \xi$ coincides with $\frac{idz\wedge d\bar z}{|z|^2}$ in each fiber. In the following, one will identify $p^*\theta_D$ with $\theta_D$ and view the latter as a $(1,1)$-form on the total space $L$. \\

We are looking for a Kähler metric $\omega= i\partial \bar \partial \varphi$ on $L\setminus D$ whose Kähler potential $\varphi=\varphi(t)$ only depends on $t$ and such that 
\begin{equation}
\label{KE equation}
\Ric(\om)=\sigma \omega.
\end{equation}

One can compute the coefficients of the metric in the frame introduced above as follows. First, $d^c\varphi=\frac{1}{i}(\partial-\bar \partial)\varphi = 2\varphi'(t) \eta$ and, then the associated Kähler form is
\begin{equation}
 \omega = i \partial \dbar \varphi =  \frac 12 d d^c \varphi = \varphi'' dt\wedge \eta - \sigma \varphi' \theta_D .\label{eq:3}
\end{equation}
In particular, we have as necessary conditions
\[\varphi''>0 \quad \mbox{and} \quad  -\sigma \varphi'>0.\]
Let $\varpi$ be a (maybe local) parallel $(n-1,0)$ form on $D$, then $\Omega=\xi\wedge\varpi$ satisfies $d\Omega=id\eta\wedge\varpi=0$, so $\Omega$ is holomorphic. Moreover, up to some positive constant, $|\Omega|^{-2} = (-\sigma\varphi')^{n-1}\varphi''$. As $\Ric \omega= i \partial \dbar \log |\Omega|^2$, in order for $\omega$ to be a solution of \eqref{KE equation}, it is enough to see that $\varphi$ is a solution of the following equation 
\begin{equation}
  \label{eq:2}
  (-\sigma\varphi')^{n-1}\varphi'' = c e^{-\sigma\varphi}
\end{equation}
for some constant $c>0$. This can be integrated into
\begin{equation}
(-\sigma\varphi')^{n+1}=a - b \sigma e^{-\sigma\varphi}\label{eq:4}
\end{equation}
for constants $a\in\Bbb{R}$ and $b=(n+1)c>0$.

As we shall see, the solutions $\varphi(t)$ will be defined on intervals of the form $(-\infty,t_0)$ for some arbitrary constant $t_0\in \mathbb R$, and they will satisfy $\sigma\varphi(t) \rightarrow \infty$ when $t\rightarrow -\infty$, that is when we go to the divisor $D$. It follows that $a\geq 0$ and actually $\varphi(t) \sim -\sigma a^{\frac1{n+1}}t$ when $t\rightarrow -\infty$, which says that $\omega$ extends over $D$ with $\omega|_D= a^{\frac1{n+1}} \theta_D$. Coming back to the equation (\ref{eq:4}) we obtain the first terms of the expansion of $\varphi$ when $t\rightarrow -\infty$, for some constant $\varphi_0$:
\begin{equation}
  \label{eq:5}
  \varphi(t) \sim -\sigma a^{\frac1{n+1}}t + \varphi_0 + \frac b{(n+1)a } e^{-\sigma\varphi_0+a^{\frac1{n+1}}t} + \cdots
\end{equation}
which shows that $\omega$ has actually a conical singularity around $D$ with angle $2\pi a^{\frac1{n+1}}$, so the angle goes to zero when $a\rightarrow 0$, this is the limit we want to study. 

Observe that if we have a solution $\varphi_1(t)$ of equation (\ref{eq:4}) with $a=b=1$, then $\varphi_1(\beta t)+\varphi_0$ is still a solution with $a=\beta^{n+1}$ and $b=\beta^{n+1}e^{\sigma \varphi_0}$. We use this remark to produce our model families $(\varphi_\beta(t))$ with angle $2\pi \beta$ degenerating to zero: \\

\noindent
{\bf Negative case. }

\noindent
{\it (i) The potential $\varphi_1$. }

\noindent
This is when $\sigma=-1$. The function $\varphi_1$ satisfies $(\varphi_1')^{n+1}=1+e^{\varphi_1}$ so we can take as solution $\varphi_1(t)=F_-^{-1}(t): (-\infty, 0)\to \mathbb R$ with $F_-:\mathbb R\to (-\infty, 0)$ defined by
  \[ F_-(x) = - \int_x^{+\infty} \frac{dx}{(1+e^x)^{\frac1{n+1}}}. \]
One can check that the precise behavior of $\varphi_1$ at $t=-\infty$ is given by
  \begin{equation}
  \label{t - infty}
  \varphi_1(t)=t+I_n+\frac{e^{I_n}}{n+1}\cdot e^t+O(e^{2t}) \quad \text{when }  t\to-\infty,
  \end{equation}
  where the constant $I_n$ is defined by
\begin{equation}
\label{intI}
I_n:=\int_0^{+\infty}\frac{du}{(e^u+1)^{\frac 1{n+1}}}-\int_{-\infty}^0\frac{(e^u+1)^{\frac 1{n+1}}-1}{(e^u+1)^{\frac 1{n+1}}}du
\end{equation}
while at $t=0^-$, one has
{\small
  \begin{equation}
  \label{t 0}
  \varphi_1(t)=-(n+1)\log\left(\frac{-t}{n+1}\right) +\frac{1}{n+2}\cdot \left( \frac{-t}{n+1}\right)^{n+1}+O(t^{2(n+1)})\quad \text{when }  t\to0^-.
  \end{equation}}

\noindent
{\it (ii) Degeneration. }
  
  \noindent
  We choose to fix $b=1$ by taking
  \begin{equation}
    \label{eq:6}
    \varphi_\beta(t) = \varphi_1(\beta t) + (n+1)\log \beta .
  \end{equation}
When $\beta\rightarrow0$, \eqref{t 0} implies that $\varphi_\beta(t) \rightarrow -(n+1) \log(\frac{-t}{n+1}) $ which is the Kähler potential of the hyperbolic cusp.\\

\noindent
{\bf Positive case. }

\noindent
{\it (i) The potential $\varphi_1$. }

\noindent
 In the positive case, $\sigma=1$.  The function $\varphi_1$ satisfies $(-\varphi_1')^{n+1}=1-e^{-\varphi_1}$ and we take $\varphi_1(t)=F_+^{-1}(-t):(-\infty, 0)\to \mathbb R_+$ with $F_+:\Bbb{R}_+\to \mathbb R_+$ by
  \[ F_+(x) = \int_0^x \frac{dx}{(1-e^{-x})^{\frac1{n+1}}}. \]

  \noindent
  Again, one can obtain the precise behavior of $\varphi_1$ at $t=-\infty$ as
  \begin{equation}
  \label{t - infty 2}
  \varphi_1(t)=-t-J_n+\frac{e^{J_n}}{n+1}\cdot e^t+O(e^{2t}) \quad \text{when }  t\to-\infty,
  \end{equation}
  where the constant $J_n$ is defined by
\begin{equation*}
\label{intI 2}
I_n:=\int_0^{+\infty} \frac{1-(1-e^{-u})^{\frac 1{n+1}}}{(1-e^{-u})^{\frac 1{n+1}}}du
\end{equation*}
while at $t=0^-$, one has
{\small
  \begin{equation}
  \label{t 0 2}
  \varphi_1(t)=c_n(-t)^{1+\frac 1n}(1+O((-t)^{1+\frac{1}{n}})\quad \text{when }  t\to0^-,
  \end{equation}}
where $c_n=(\frac{n}{n+1})^{\frac{n+1}{n}}$.\\

\noindent
{\it (ii) Degeneration. }
  
  \noindent
  We now choose the degeneration
  \begin{equation}
    \label{eq:7}
    \varphi_\beta(t) = \varphi_1(\beta t).
  \end{equation}
  Here the limit when $\beta\rightarrow0$ is just $0$. More precisely, the  asymptotics \eqref{t 0 2} imply that when $\beta\rightarrow0$ one has  \begin{equation}
    \label{eq:8}
    \varphi_\beta(t) \sim \left(\frac {-n\beta t}{n+1}\right)^{1+\frac 1n} .
  \end{equation}
Therefore the rescaling $\beta^{-1-\frac1n}\varphi_\beta$ converges to $ (-\frac{nt}{n+1})^{1+\frac1n}$ which is the Kähler potential of a Ricci flat metric on $L\setminus D$, the Tian-Yau metric: it gives the asymptotic behaviour of the Tian-Yau metric of $X\setminus D$ near $D$.\\

We have seen a limit of $\varphi_\beta(t)$ when $\beta\rightarrow 0$ on each compact set in $t$. To understand the global geometry of our models, we write from (\ref{eq:3})
\begin{equation}
  \label{eq:9}
  \omega_\beta = i \partial \dbar \varphi_\beta = \beta^2\varphi_1''(\beta t) dt\wedge\eta - \beta \sigma \varphi_1'(\beta t) \theta_D.
\end{equation}
The geometry is clear when one writes the associated Riemannian metric $g_\beta$, after the change of variable $u=\beta t\in (-\infty, 0)$:
\begin{equation}
  \label{eq:10}
  g_\beta = 2\varphi_1''(u) \big( \tfrac14 du^2 + \beta^2\eta^2 \big) - \beta \sigma \varphi_1'(u) g_D .
\end{equation}
The geometry collapses at speed $\sqrt \beta$ in the directions of $D$, and $\beta$ in the circle directions. Observe that, up to a multiplicative constant, we have when $u\rightarrow 0$ the asymptotics $\varphi_1''(u) \sim u^{-2}$ in the negative case, and $u^{-1+\frac1n}$ in the positive case, from which it follows that the diameter is infinite in the negative case and bounded in the positive case (as it should by Myers's theorem). More precisely, we have at $u=0$ the following asymptotics
\begin{equation}
  \label{eq:11bis}
  \varphi_1'(u) =-c_n'(-u)^{\frac1n}+O((-u)^{1+\frac 2n}),  \quad  \varphi_1''(u) =\frac {c_n'}{n}(-u)^{-1+\frac1n}+O((-u)^{\frac 2n}).
\end{equation}
with $c_n'=\big(\frac{n}{n+1}\big)^{\frac1n}$.

To see the behaviour of the metric near the divisor $D$, we consider the expansion at $u=-\infty$ given as in (\ref{eq:5}) by
\begin{equation}
  \label{eq:11}
  \varphi_1'(u) = -\sigma + \frac 1{n+1} e^u + O(e^{2u})
\end{equation}
from which follows, taking $r=e^{\frac u2}\in (0,1)$,
\begin{align}
  \label{eq:12}
  g_\beta & =\tfrac2{n+1}e^u \big( \tfrac14 du^2 + \beta^2\eta^2 \big) + \beta g_D  + O(e^u) \\
  & =  \tfrac2{n+1}\big( dr^2 + \beta^2r^2 \eta^2 \big) + \beta g_D + O(r^2) ,\label{eq:13}
\end{align}
where the $O(r^2)$ is with respect to $g_\beta$ and is uniform with respect to $\beta$.

\section{Curvature calculations}
\label{sec:curv-calc}

We now calculate the curvature of our model metrics $\omega_\beta$ on $L\setminus D$. This gives a control of the local geometry and will be also used in the Laplacian estimate.

The Kähler metric $\omega=i\partial\bar\partial \varphi$ induces a hermitian metric $h$ on the bundle $\Omega^{1,0}$ of $(1,0)$-forms on $L\setminus D$, and it will be convenient to compute its Chern curvature tensor $F^{\Omega^{1,0}}$. At this point, we take an arbitrary Kähler potential $\varphi=\varphi(t)$. We will use the $\mathcal C^{\infty}$ orthogonal splitting 
\[ \Omega^{1,0} = \Bbb{C}\xi \oplus p^* \Omega^{1,0}_D. \]
From (\ref{eq:3}) we see that the Hermitian metric preserves this decomposition, and is equal to
\[ h =
  \begin{pmatrix}
    (\varphi'')^{-1} & \\ & (-\sigma\varphi')^{-1} h^{\Omega^{1,0}_D}
  \end{pmatrix}. \]
  where $ h^{\Omega^{1,0}_D}$ is the hermitian metric induced on $\Omega^{1,0}_D$ by the Kähler metric $\theta_D$. 
From $d\xi=id\eta=-i\sigma\theta_D$ we deduce the $\dbar$ and $\partial$ operators of $\Omega^{1,0}$ in this splitting:
\[ \dbar = \begin{pmatrix} \dbar & 0 \\ a & \dbar^{\Omega^{1,0}_D} \end{pmatrix}, \qquad
  \partial = \begin{pmatrix} \partial - \partial \log \varphi'' & - \frac{\varphi''}{-\sigma\varphi'} a^*  \\ 0 & \partial^{\Omega^{1,0}_D} - \partial \log(-\sigma\varphi')
    
  \end{pmatrix}
\]
Here $a$ is the $(0,1)$-form with values in $\Hom(\Bbb{C},\Omega^{1,0}_D)=\Omega^{1,0}_D$ defined by $a_X=X\lrcorner d\xi=-i \sigma X  \lrcorner \theta_D$, and $a^*$ its adjoint $a^*_X\alpha=-i\sigma\Lambda(\alpha\wedge(X\lrcorner\theta))$. The familiar form for the curvature is then
\begin{equation}
  F^{\Omega^{1,0}}
=
  \begin{pmatrix}
    \partial \dbar \log \varphi'' - \frac{\varphi''}{-\sigma\varphi'} a^* \wedge a & - \dbar (\frac{\varphi''}{-\sigma\varphi'}a^*) \\ \partial^h a & \partial \dbar \log \varphi' + F^{\Omega^{1,0}_D} - \frac{\varphi''}{-\sigma\varphi'} a \wedge a^*
  \end{pmatrix}\notag \\
  \end{equation}
  so that $iF^{\Omega^{1,0}}$ is given by
{\footnotesize
\begin{equation}
  \begin{pmatrix}
    (\log \varphi'')'' i\xi\wedge \bar \xi -  \sigma(\log \frac{\varphi''}{-\sigma\varphi'})' \theta_D & -i (\frac{\varphi''}{-\sigma\varphi'})' \overline \xi \wedge a^* \\
    i \log(\frac{\varphi''}{-\sigma\varphi'})' \xi \wedge a & \log (-\sigma\varphi')'' i\xi\wedge \bar \xi - \sigma\log (-\sigma\varphi')'  (\theta_D+\Theta_D) +i F^{\Omega^{1,0}_D}  
  \end{pmatrix}\label{eq:14}
\end{equation}}
where the last $\Theta_D$ is the 2-form with values in the endomorphisms of $\Omega^{1,0}_D$ defined by $(\Theta_D)_{X,Y}(\alpha)=-(X\lrcorner\alpha)(Y\lrcorner\theta_D)$ for $X\in T^{1,0}$ and $Y\in T^{0,1}$.

\begin{lem}
\label{lemma curvature}
  One has the following bounds for the curvature of the model Kähler metric $\omega_\beta$ defined in (\ref{eq:9}):
  \begin{itemize}
  \item in the negative case ($\sigma=-1$), if $D$ is flat, then the curvature is bounded;
  \item in the positive case ($\sigma=1$), the curvature is bounded by \[ \mathrm{cst.} \big( \frac1{1-e^{-\varphi_\beta}}+\frac1{\beta(1-e^{-\varphi_\beta})^\frac1{n+1}} \big). \]
  \end{itemize}
\end{lem}
\begin{proof}
For now, $\omega=i\partial\bar\partial \varphi$ is still arbitrary, and only at the end we will choose $\varphi:=\varphi_\beta$ to be the potential constructed by the Calabi Ansatz above. Let $X=\lambda \xi^*+v$ with $\lambda \in \mathbb C, v\in T^{1,0}_D$ and $Y=\mu \xi+\alpha$ with $\mu \in \mathbb C, \alpha \in \Omega^{1,0}_D$. We assume that $\|\xi\|_{\omega} =\|Y\|_{h}=1$, so that
\begin{equation}
\label{lambda}
|\lambda|^2 \le \frac{1}{\varphi''}, \,\, \|v\|^2_{\theta_D}\le \frac1{-\sigma \varphi'}, \quad \mbox{and} \quad |\mu|^2 \le \varphi'', \,\,\|\alpha\|^2_{h_D}\le -\sigma \varphi'
\end{equation}
where $h_D$ is the metric on $\Omega^{1,0}_D$ induced by $\theta_D$. 
The hermitian matrix $h\cdot iF^{\Omega^{1,0}}_{X,\bar X}$ is equal to 
{\footnotesize
\[
\begin{pmatrix}
|\lambda|^2 \frac{(\log \varphi'')''}{\varphi''}-\frac{\sigma \|v\|^2_{\theta_D}}{\varphi''} \log(\frac{\varphi''}{-\sigma \varphi'})'& -i\bar \lambda a^*_v \frac{1}{-\sigma \varphi'} \log(\frac{\varphi''}{-\sigma\varphi'})'\\ 
i\lambda a_{\bar v} \frac{1}{-\sigma \varphi'} \log(\frac{\varphi''}{-\sigma\varphi'})' &  \frac{\log (-\sigma\varphi')''}{-\sigma \varphi'} |\lambda|^2 +\frac{\log (-\sigma\varphi')'}{ \varphi'}  (\theta_D+\Theta_D)_{v,\bar v} +\frac{i}{-\sigma \varphi'} F^{\Omega^{1,0}_D}_{v,\bar v}  
\end{pmatrix}
\]
}
As a result, the expansion of $\langle iF_{X,\bar X}^{\Omega^{1,0}}Y, Y\rangle_h$ involves the following terms
\[|\lambda\mu|^2\frac{(\log \varphi'')''}{\varphi''}, |\mu|^2 \frac{ \|v\|^2_{\theta_D}}{\varphi''} (\log(\frac{\varphi''}{-\sigma \varphi'})', \lambda\bar \mu \langle a_{\bar v},\alpha\rangle_{h_D}\log(\frac{\varphi''}{-\sigma\varphi'})',\]
\[|\lambda|^2\|\alpha\|_{h_D}^2 \frac{\log (-\sigma\varphi')''}{-\sigma \varphi'} ,  \|\alpha\|_{h_D}^2\frac{\log (-\sigma\varphi')'}{ \varphi'}  (\theta_D+\Theta_D)_{v,\bar v}, \frac{1}{ \varphi'} \langle i F^{\Omega^{1,0}_D}_{v,\bar v}\alpha, \alpha\rangle_{h_D}. \]
 Given the bounds \eqref{lambda},  we see that we need to bound the quantities
  \begin{equation}
   \frac{(\log \varphi'')''}{\varphi''}, \quad \frac{\log (-\sigma\varphi')''}{\varphi''}, \quad \frac{(\log \varphi'')'}{\varphi'}, \quad \frac{\log (-\sigma\varphi')'}{\varphi'},\label{eq:15}
 \end{equation}
  while the term $F^{\Omega^{1,0}_D}$ (present only in the positive case) is bounded by $\frac{\mathrm{cst.}}{-\sigma\varphi'}$.

  Up to an additive constant, we have $\varphi_\beta(t)=\varphi_1(\beta t)$, so we see that the factors $\beta$ in (\ref{eq:15}) cancel and it is enough to bound these quantities for $\varphi_1$, while the term involving the curvature of $D$ is bounded by $\frac{\mathrm{cst.}}{-\beta\varphi_1'(\beta t)}$. We now use the equation satisfied by $\varphi_1$, that is
  \[ -\sigma\varphi_1' = (1-\sigma e^{-\sigma\varphi_1})^{\frac1{n+1}}. \]
  Taking $x=e^{-\sigma\varphi_1}$ we obtain $x'=x(1-\sigma x)^{\frac1{n+1}}$. It is then convenient to write all the quantities in terms of $x$. We have $-\sigma\varphi_1' = (1-\sigma x)^{\frac1{n+1}}$, therefore $-\sigma\varphi_1''=\frac1{n+1} x(1-\sigma x)^{-\frac{n-1}{n+1}}$. Then one calculates all quantities in (\ref{eq:15}):
  \begin{align*}
   \frac{\log (-\sigma\varphi_1')'}{\varphi_1'}&=\frac x{1-\sigma x}, &
\frac{\log (-\sigma\varphi_1')''}{\varphi_1''}&=\frac {n+1-\sigma x}{(n+1)(1-\sigma x)}, \\
\frac{\log (-\sigma\varphi_1'')'}{\varphi_1'}&=\frac{n+1-2\sigma x}{(n+1)(1-\sigma x)}, &
\frac{\log (-\sigma\varphi_1'')''}{\varphi_1''}&=\frac{-n^2+n+2-2\sigma x}{(n+1)(1-\sigma x)}.
\end{align*}
If $\sigma=-1$ then $1-\sigma x=1+e^\varphi\geq 1$ therefore all these quantities are bounded. If $\sigma=1$ then $1-\sigma x=1-e^{-\varphi}$ and we obtain the bounds in the statement of the lemma.
\end{proof}
In the positive case, the divisor $D$ corresponds to $\varphi\rightarrow +\infty$ so the curvature is $O(\beta^{-1})$ in all sets of the form $u<-A<0$. At $u=0$ the metric $g_\beta$ from (\ref{eq:10}) degenerates, but this part will be cut out since we will glue with the rest of $X$.

\section{Asymptotics of the conical KE metrics on ball quotients}
\label{sec:asympt-conic-ke}

\subsection{Set-up}
\label{setup}
In this section, we borrow the setup and notation of \textsection~\ref{setup0} and we assume additionally that $\Omega=\mathbb B^n$ is the complex hyperbolic space of dimension $n$. 
In this section, we assume that  $\Gamma$ is neat,  so that $X=\quotient{\Gamma}{\mathbb B}$ can be compactified smoothly by adding finitely many disjoint tori $D_1, \ldots, D_N$ of dimension $n-1$. In general, this is only true up to the action of a finite group (locally in the neighborhood of each torus). The Kähler-Einstein metric $\omke$ is, up to a normalizing constant, the {\it hyperbolic} metric on $X$, described locally near $D$ by \eqref{model}.

\medskip
It will be important in the following to allow cone angles along $D_\lambda$ that are not necessarily of the form $2\pi \beta a_\lambda $ for some given $a_\lambda>0$ and a \textit{single} parameter $\beta>0$ going to zero. For that reason and from now on, we denote by  $\beta:=(\beta_1, \ldots, \beta_N)$ a $N$-tuple of positive numbers. Since the components $D_1, \ldots, D_N$ of the boundary divisor $D$ are disjoint, the divisor $-\sum_\lambda a_\lambda D_\lambda$ is relatively ample for {\it any} $a_\lambda>0$. 
In particular, up to changing $h_\lambda$ one can find $\delta_0>0$ such that $\chi-\delta\theta_\lambda$ is semi-positive globally on $\overline X$ and Kähler on $X\setminus \sqcup_{\mu \neq \lambda}D_\mu$, for any $\delta \le 2\delta_0N$. 
As a result, $\chi-\sum_{\lambda}a_\lambda \theta_\lambda$ is globally Kähler on $\overline X$ for any $a_\lambda \in (0,2\delta_0]$. In the following, one will assume that $\beta_\lambda \le \delta_0/2N$ for any $\lambda$.\\

The Kähler-Einstein metric $\omkh=\chi- \sum_{\lambda=1}^N\beta_\lambda  \theta_\lambda+dd^c\wh \vp_\beta$ solution of 
\begin{equation}
\label{KE current 2}
\Ric \omkh = -\omkh+ \sum_{\lambda=1}^N(1-\beta_\lambda)[D_\lambda]
\end{equation}
for $\|\beta\|$ small enough solves the following Monge-Ampère equation
\begin{equation}
\label{MAep2}
(\chi-  \sum_{\lambda=1}^N\beta_\lambda  \theta_\lambda+dd^c \wh\vp_\beta)^n=\frac{e^{\wh \vp_\beta} dV}{\prod_\lambda |s_\lambda|_{h_\lambda}^{2(1-\beta_\lambda)}}
\end{equation}
where $dV$ is a smooth volume form such that $-\mathrm{Ric}(dV)+\sum_{\lambda=1}^N \theta_\lambda = \chi$.\\

One can reproduce the arguments in the proof of Theorem~\ref{weak convergence} verbatim to show that $\wh \phi_\beta$ almost decreases to $\wh \phi$ when $\beta$ goes to zero. More precisely, one can find a sequence of real numbers $\tau_\beta\to0$ such that $\frac{1}{1-\delta_0^{-1}\sum_\lambda \beta_\lambda}(\wh \phi_\beta+\tau_\beta)$ decreases to $\wh \phi$ when $\beta \searrow 0$ component-wise. The main point is that if $\beta'>\beta$ component-wise and if we set $B=\sum \beta_\lambda$ (resp. $B'=\sum \beta_\lambda'$), we have
\begin{align*}
\frac{1-\delta_0^{-1}B}{1-\delta_0^{-1}B'}\cdot(\chi-\sum \beta_\lambda'\theta_\lambda) =& (\chi-\sum \beta_\lambda \theta_\lambda)\\
&+\frac{B'-B}{\delta_0-B'}\cdot \Big[\chi-\sum_{\lambda}\underbrace{\big(\delta_0\beta_\lambda+\frac{\delta_0-B}{B'-B}\cdot(\beta_\lambda'-\beta_\lambda)\big)}_{\in (0,2\delta_0)}\cdot \theta_\lambda\Big].
\end{align*}
which replaces the identity \eqref{comparaison psh}. 

Moreover, the Laplacian estimate from the proof of Theorem~\ref{weak convergence} carries over with no significant change, and therefore \begin{equation}
\label{smooth convergence}
\wh \omega_\beta \underset{\beta\to 0}\longrightarrow \omke \quad \mbox{in } \quad \mathcal C^{\infty}_{\rm loc}(U^*)
\end{equation}

\subsection{Comparison to the model metric}
We now aim to compare the global Kähler-Einstein metric $\wh \omega_\beta$ to the model $\omega_\beta$ constructed via the Calabi Ansatz in \textsection~\ref{sec:calabi-metric}. One 

Given any torus $D\in \{D_1, \ldots, D_N\}$, one can identify an open neighborhood $U$ of $D$ in $\oX$ to a neighborhood of the zero section in the total space of the normal bundle  $L:=N_{D/\overline X} \to D$.  Moreover, $L$ comes naturally equipped with a smooth hermitian metric $h$ such that $\theta_D:=\pi \cdot i\Theta(L^{-1}, h^{-1})$ is a flat Kähler metric on $D$. We let $p:U\to D$ be the projection induced by the identification of $U$ to an open subset of the total space of $L$. Under this identification and given a point $(x,v)\in U$ (i.e. $x\in D, v\in L_x$), we can consider the quantity $\|v\|^2_h$ and assume that $\|v\|_h^2<e^{-1}$ on $U$. On $U^*:=U\setminus D$, the smooth function $t=\log \|v\|_h^2:U^*\to (-\infty, -1)$ satisfies 
\begin{equation}
i\d\dbar t = p^*\theta_D.
\end{equation} 
Moreover, the Kähler-Einstein metric $\omke$ on $X$ has an {\it exact} expression in restriction to $U^*$; namely

\begin{align}
\label{hyp} 
 \omke|_{U^*} & = i \d\dbar\big[-(n+1)\log(-t)\big]. \\
\nonumber & = (n+1)\Big[\frac{\xi \wedge \bar \xi}{(-t)^2}+\frac{p^*\theta_D}{-t}\Big]
\end{align}
where $\xi=\frac 12dt+i\eta$ has been defined in \textsection~\ref{sec:calabi-metric}. We have observed in {\it ibid.} that the potential $\varphi(t)=-(n+1)\log(-t) $ of $ \omke$ is the limit of the potentials \[\psi_\beta:=\varphi_\beta(t)+(n+1)\log(n+1)\] of $\omega_\beta$ (cf. \eqref{eq:6}) when $\beta\to 0$ and that the convergence is smooth on the compact subsets of $\overline U \setminus D$. In particular, we get
\begin{equation}
\label{smooth convergence 2}
 \omega_\beta \underset{\beta\to 0}\longrightarrow \omke \quad \mbox{in } \quad \mathcal C^{\infty}_{\rm loc}(U^*).
\end{equation}

\medskip
Let $\Omega=\xi \wedge \varpi$ be the holomorphic $n$-form with logarithmic poles along $D$ constructed on $U$ in the previous section.The Monge-Ampère equation solved by $\psi_\beta$ reads
\[(i\d\dbar \psi_\beta)^n=e^{\psi_\beta}i^{n^2}\Omega\wedge \overline \Omega.\]
The Monge-Ampère equation solved by $\wh \omega_\beta$ has a similar form. Indeed, let $\psi_\chi$ be a smooth potential for $\chi$ on $U$ and let us set $\wh \psi_\beta:=\psi_\chi-\beta t+\wh \varphi_\beta$ which is well-defined on $U^*$. 
Recall that $\wh \varphi_\beta\in L^{\infty}(U^*)$ so that $\wh \psi_\beta-\varphi_\beta$ is globally bounded on $U^*$ (only qualitatively at this point). Moreover, we have 
\[(i\d\dbar \wh \psi_\beta)^n=e^{\wh \psi_\beta+F_\beta}i^{n^2}\Omega\wedge \overline \Omega\]
where $F_\beta$ is a smooth function on $U^*$, globally bounded independently of $\beta$, i.e. $\|F_\beta\|_{L^{\infty}(U^*)}\le C_1$. 

\begin{lem}
The following bound holds 
\begin{equation}
\label{sup bound}
\|\wh\psi_\beta-\psi_\beta\|_{L^{\infty}(U)} \le C_1.
\end{equation}
\end{lem}

\begin{proof}
This is a simple application of the maximum principle. Indeed, let $\delta>0$ arbitrarily small and let $H=H_{\beta,\delta}:=\wh\psi_\beta-\varphi_\beta+\delta t$. Since $H_{\beta,\delta}$ goes to $-\infty$ along $D$, its maximum is attained at a point $x=x_{\beta, \delta}\in \overline U\setminus D$ at which the complex Hessian of $H$ is non-positive. In particular, we get $i\d\dbar \wh \psi_\beta\le i\d\dbar \psi_\beta-\delta \theta_D\le i\d\dbar \psi_\beta$ at $x$. Taking the top wedge product and using the Monge-Ampère equations above, we find $H(x)\le -F_\beta(x)+\delta t(x)\le C_1$. In particular, $H\le C_1$ everywhere on $U^*$ and passing to the limit when $\delta\to 0$, we get the first half of \eqref{sup bound}. The other half is obtained in a similar way. 
\end{proof}

\begin{rem}
In the lemma above, we could have use Bedford-Taylor's comparison principle instead of the maximum principle (with the tweak by $\delta t$), see e.g. \cite[Lemma~3.4]{CKZ}.
\end{rem}

Next, we claim that $\wh \omega_\beta$ and $\omega_\beta$ are uniformly quasi-isometric on $U^*$.

\begin{lem}
\label{lemma laplacian bound}
There exists $C_2>0$ independent of $\beta$ such that
\begin{equation}
\label{laplacian bound}
C_2^{-1}\omega_\beta\le \wh \omega_\beta \le C_2 \omega_\beta.
\end{equation}
\end{lem}

\begin{proof}
Consider the smooth function 
\[H_\beta:=\log \tr_{\wh \omega_\beta}  \omega_\beta \quad \mbox{ on } U^*.\] 
Since $\Ric \, \wh \omega_\beta = -\wh \omega_\beta$ and the holomorphic bisectional curvature of $\omega_\beta$ is bounded above independently of $\beta$ by Lemma~\ref{lemma curvature}, an application of Chern-Lu formula (see e.g. \cite[Proposition~7.1]{Rub14}) yields a constant $B>0$ independent of $\beta$ such that
\begin{equation}
\label{CL2}
\Delta_{\wh \omega_\beta} H_\beta \ge -1 -Be^{H_\beta} \quad \mbox{on } U^*.
\end{equation}
Thanks to \eqref{smooth convergence}- \eqref{smooth convergence 2}, we have 
\begin{equation}
\label{boundary bound}
H_\beta \le (n+1) \quad \mbox{on } \d U
\end{equation}
for $\beta$ small enough. 

Since $\Delta_{\wh \omega_\beta}(\psi_\beta-\wh \psi_\beta)=e^{H_\beta}-n$ and $i\d\dbar t \ge 0$, we get for any $\delta>0$ 
\[\Delta_{\ome} \big(H_\beta-(B+1)\cdot (\psi_\beta-\wh \psi_\beta)+\delta t\big)= e^{H_\beta}-n(B+1)-1.\]
The maximum of the function inside the Laplacian is attained at $ x\in \overline U\setminus D$. If $x\in \d U$, then \eqref{sup bound}-\eqref{boundary bound} and the inequality $t\le 0$ imply that $H_\beta\le (n+1)+2(B+1)C_1-\delta t$ on $U^*$. If  $x\in U$, then the maximum principle implies that $H_\beta\le (B+1)(n+2C_1)+1-\delta t$ on $U^*$. Passing to the limit when $\delta\to 0$, one finds $H_\beta\le C_2$ on $U^*$.  The result follows (up to enlarging $C_2$) since the Monge-Ampère of $\om_{\beta}$ and $\wh \omega_\beta$ are commensurable \-- which itself relies on the estimate \eqref{sup bound}. 
\end{proof}

Since we know that $\omega_\beta$ and $\wh \omega_\beta$ are asymptotically close at any order away from $D$, one can improve Lemma~\ref{lemma laplacian bound} as follows. 

\begin{lem}
\label{lemma improved laplacian}
There exists a sequence of numbers $\ep_\beta \searrow 0$  such that
\begin{equation}
\label{laplacian bound 2}
(1-\ep_\beta)\omega_\beta\le \wh \omega_\beta \le (1+\ep_\beta) \omega_\beta \quad \mbox{on } U^*. 
\end{equation}
\end{lem}

\begin{proof}
We introduce for any $\delta>0$ the quantities
\[F_{\beta}:=\log\left(\frac{\wh \omega_\beta^n}{\omega_\beta^n}\right) \quad \mbox{and} \quad F_{\beta,\delta}:=F_\beta+\delta t.\]
The function $F_\beta$ is bounded on $U$ and smooth away from $D$. If we can show that $F_\beta$ converges uniformly to $0$ on $U$, then we will be done since we know that $\wh \omega_\beta$ and $\omega_\beta$ are uniformly quasi-isometric thanks to Lemma~\ref{lemma laplacian bound}.
First, we observe that 
\begin{equation}
\label{boundary}
\lim_{\beta \to 0} \|F_\beta\|_{L^{\infty}(\partial U)}=0
\end{equation} 
thanks to  \eqref{smooth convergence}- \eqref{smooth convergence 2}. Let $x\in U$ be a point where $F_{\beta,\delta}$ attains its maximum. If $x\in \partial U$, we have $F_{\beta,\delta}\le \|F_\beta\|_{L^{\infty}(\partial U)}$ which goes to zero by \eqref{boundary}. Otherwise, $x\in U^*$ and we have $dd^c F_{\beta,\delta}(x)\le 0$.
Since both metrics are Kähler-Einstein with the same constant, we have $dd^c F_\beta=\wh \omega_\beta-\omega_\beta$. In particular, we get at the point $x$ the following inequality
\[\wh \omega_\beta(x)\le \omega_\beta(x)-\delta dd^c t \le \omega_\beta(x)\]
It follows that $F_\beta(x)\le 0$, hence $F_{\ep,\delta}\le 0$. Passing to the limit when $\delta\to 0$, we obtain that in any case, $\sup_U F_\beta \le o(1)$ when $\beta \to 0$. 

One can proceed similarly with $G_{\beta,\delta}=\log\left(\frac{\omega_\beta^n}{\wh \omega_\beta^n}\right) +\delta t$ to see that $\inf_U F_{\beta} \ge o(1)$ when $\beta\to 0$. The lemma is proved.
\end{proof}

To finish this section, we put together the Laplacian estimate \eqref{laplacian bound 2} with the asymptotics \eqref{t - infty}-\eqref{t 0}, which yields

\begin{thm}
\label{KE asymptotics}
The conical Kähler-Einstein metric $\wh \omega_\beta$ has the following behavior on $U$ as $\beta$ approaches zero:
\begin{enumerate}[label=$\bullet$]

\item On $\{\beta t\to 0\}$, it is quasi-isometric to 
\[\omke= (n+1)\Big[\frac{i\xi \wedge \bar \xi}{(-t)^2}+\frac{\theta_D}{-t}\Big]\]
with quasi-isometry constant converging to $1$ as $\beta t \to 0$. 

\item On $\{\beta t\to -\infty\}$, it is quasi-isometric to 
\[ a_n \beta^2\cdot e^{\beta t}  i\xi \wedge \bar  \xi+\beta\theta_D\]
with quasi-isometry constant converging to $1$ as $\beta t \to -\infty$ and $\beta \to 0$ and where $a_n=\frac{e^{I_n}}{n+1}$, $I_n$ being defined in \eqref{intI}. \\

\item Elsewhere, i.e. on $\{-C\le \beta t\le C^{-1}\}$; it is quasi-isometric to 
\[\beta^2\cdot  e^{\beta t}  i\xi \wedge \bar  \xi + \beta\theta_D\]
with quasi-isometry constant uniformly bounded as $\beta \to 0$. 
\end{enumerate}
\end{thm}
 The picture below illustrates the result. 

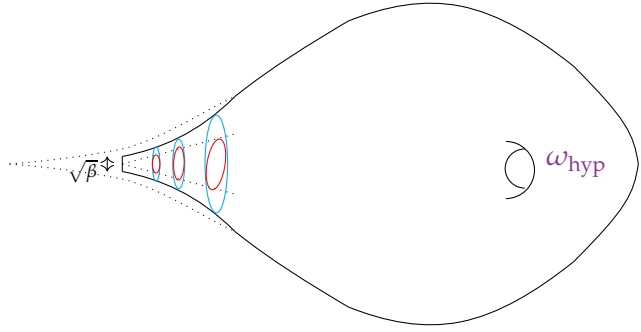
\begin{center}
\begin{tikzpicture}
\draw[dotted] (1,0) .. controls (1.5,0.04) and (2,0.1) .. (2.5,0.2);
\draw[dotted] (2.5,0.2) .. controls (3,0.32) and (3.5,0.66) .. (4,0.9);
\draw (4,0.9)  .. controls (4.5,1.3) and (5,1.6) .. (5.5,1.9); 
\draw (5.5,1.9) .. controls (6,2.1) and (6.5,2.2) .. (7, 2.1);
\draw  (7,2.1) .. controls (7.5,2) and (8,1.7) .. (8.5,1.3) ;
\draw (8.5, 1.3) .. controls (9.2, 0.6) and (9.3, 0.4) .. (9.36, 0) node[left,Fuchsia] { $\omega_{\rm hyp}$ \,\,\,};

\draw (2.5,0.1) .. controls (3,0.2) and (3.5,0.4) .. (4,0.9);
\draw (2.5,0.1) -- (2.5,-0.1);
\draw[ <->] (2.3,0.1) -- (2.3,-0.1) node[left] {\tiny$\sqrt \beta$};

\draw[dotted] (2.5,0) -- (4,0.4);
\draw[dotted] (2.5,0) -- (4,-0.4);

\draw (2.5,-0.1) .. controls (3,-0.2) and (3.5,-0.4) .. (4,-0.9);

\draw[dotted] (1,0) .. controls (1.5,-0.04) and (2,-0.1) .. (2.5,-0.2);
\draw[dotted] (2.5,-0.2) .. controls (3,-0.32) and (3.5,-0.66) .. (4,-0.9);
\draw (4,-0.9)  .. controls (4.5,-1.3) and (5,-1.6) .. (5.5,-1.9); 
\draw (5.5,-1.9) .. controls (6,-2.1) and (6.5,-2.2) .. (7,- 2.1);
\draw  (7,-2.1) .. controls (7.5,-2) and (8,-1.7) .. (8.5,-1.3) ;
\draw (8.5, -1.3) .. controls (9.2, -0.6) and (9.3, -0.4) .. (9.36, 0);

\draw (7.6,0.3) arc (90:-90:0.38)  ;
\draw (7.85,0.2) arc (90:270:0.26);

\draw[cyan] (2.95,0) ellipse (0.05 and 0.23);
\draw[cyan] (3.25,0) ellipse (0.08 and 0.33);
\draw[cyan] (3.75,0) ellipse (0.15 and 0.65);

\draw[red, rotate=0] (2.95,0) ellipse (0.05 and 0.12);
\draw[red, rotate=-2] (3.25,0.12) ellipse (0.07 and 0.22);
\draw[red, rotate=-9] (3.7,0.58) ellipse (0.12 and 0.34);

\end{tikzpicture}
\end{center}

\subsection{Ramified covers}
\label{section covers}
In Set-up~\ref{setup}, assume additionally that $\Gamma$ is arithmetic, so that $\Gamma$ can be realized as the integral points $G(\mathbb Z)$ of an algebraic group $G$ defined over $\mathbb Z$. Given an integer $m\ge 1$, the congruence subgroup $\Gamma(m)=\mathrm{Ker} \Big[ G(\mathbb Z)\to G\left(\quotientd{\mathbb Z}{m\mathbb Z}\right)\Big]$ induces an étale cover \[\pi_m:\quotient{\Gamma(m)}{\mathbb B}\to\quotient{\Gamma}{\mathbb B}.\]
Let $X_m:=\quotient{\Gamma(m)}{\mathbb B}$ and let $\overline{X}_m$ be a log smooth compactification of $X_m$. The étale cover $\pi_m:X_m\to X$ can be uniquely extended to a cover $\overline \pi_m:\overline{X}_m\to \oX$. Up to taking a further cover, one can assume that $\overline \pi_m$ is Galois, with group $\Lambda_m$. Moreover, Mumford shows in \cite[p270-271]{Mum77} that $\overline \pi_m$ is highly ramified along $D$ in the following sense. Let $\nu_{m,\lambda}$ be the ramification order of $\overline \pi_m$ along $D_\lambda$. Then, given any integer $\ell \ge 1$, there exists $m=m(\ell)$ such that $\ell \vert \nu_{m,\lambda}$ for any $\lambda=1,\ldots,N$. 

Pick $\ell$ arbitrary large and consider the ramified cover $\overline{\pi}_m: \overline{X}_m\to \oX$ for $m=m(\ell)$ as above. Set $\beta_m:=(\frac{1}{\nu_{m,1}}, \ldots, \frac{1}{\nu_{m,N}})$ and consider the conical Kähler-Einstein metric $\om_{\beta_m}$ with cone angles $2\pi (\beta_m)_\lambda$ along $D_\lambda$. By the choice of $\beta_m$, $\om_{\beta_m}$ is an orbifold Kähler metric for the pair $(X,\sum_\lambda(1-\frac{1}{\nu_{m,\lambda}})D_\lambda)$, hence $\omega_m:=\overline \pi_m^*\om_{\beta_m}$ is a genuine Kähler-Einstein metric on the compact Kähler manifold $\overline{X}_m$, \ie
\[\Ric \om_m=-\om_m \quad \mbox{on } \overline{X}_m.\] 

As $\ell\to +\infty$, so does $m$ and $\beta_m$ converges to $0$, so that $\om_{\beta_m}$ converges to the hyperbolic (Bergman) metric on $X$ by the previous results. Schematically, one can summarize the situation as below 
\[\quotientd{(\overline X_m, \omega_m)}{\Lambda_m} \quad \underset{m\to +\infty}{\longrightarrow} (\quotient{\Gamma}{\mathbb B}, \omega_{\rm Berg}).\]

\section{Gluing with the Tian-Yau metric}
\label{sec:gluing-with-tian}

We now pass to the setting of a compact Fano manifold $X$ of dimension $n\ge 2$ endowed with a {\it smooth} anticanonical divisor $D\subset X$. Note that $D$ is connected by the Lefschetz hyperplane theorem. We denote by $L$ the normal bundle of $D$. The objects on $L$ constructed in section \ref{sec:calabi-metric} will now carry an index $L$ ($h_L$, $\Omega_L$, $\varphi_{\beta,L}$, $\omega_{\beta,L}$, etc.) to distinguish them from the objects constructed on $X$.

\subsection{The Tian-Yau metric}
\label{sec:tian-yau-metric}

The Tian-Yau metric was obtained in \cite{TY1} and precise asymptotics are derived in \cite{Hein12}. A nice summary is written in \cite[§~3]{HSVZ}, and the asymptotics written below are taken from this reference.

We choose the holomorphic $(n-1)$-form $\varpi$ on $D$ of section \ref{sec:calabi-metric} so that $\frac{i^{(n-1)^2}}n \varpi \wedge \overline\varpi = \omega_D^{n-1}$.
We have a global holomorphic $n$-form $\Omega$ on $X$ with a simple pole along $D$, normalized by $\varpi=\mathrm{Res}_D\Omega$, so that the form induced by $\Omega$ on $L$ is $\Omega_L=\xi\wedge p^*\varpi$.

The normal bundle $L$ gives the infinitesimal neighbourhood of $D$ in $X$. One can identify a neighbourhood of $D$ in $X$ with a disc bundle in $L$: one method uses the Riemannian exponential of a Hermitian metric on $X$, but we prefer a more intrinsic identification using the theory of extremal discs, which produces a (non-holomorphic) fibration in holomorphic discs.  The theory was especially used by Lempert \cite{Lem92} to study fillings of 3-dimensional Cauchy-Riemann manifolds, both on the pseudoconvex and pseudoconcave sides,  and also by Bland-Duchamp \cite{BlaDuc91}. We will use the following proposition, whose proof is similar to that for 3-dimensional pseudoconcave domains \cite[Theorem 10.1]{Lem92}; it can also be extracted from the general statement in \cite[Theorem 4.1]{Biq02}, which is valid for any signature of the (non degenerate) Levi form.

\begin{prop}\label{prop:extremal-discs}
  There exists a diffeomorphism $\Phi:\Delta_L\rightarrow U_L\subset X$ from the disc bundle $\Delta_L\subset L$ to a neighbourhood $U_L\subset X$ of $D$, such that $\phi:=\Phi^*J_X-J_L\in \Omega^{0,1}(T^{1,0})$ satisfies $\phi|_D=0$, $\phi$ is a section of $(p^*\Omega^{0,1}_D)\otimes\ker \eta^{1,0}$ (that is, $\phi$ is purely horizontal), and $\phi$ is holomorphic along the discs of $\Delta_L$.

  Moreover $\Phi^*\Omega=v\,  (1-\phi)^*\Omega_L$, where $v$ is a function on $\Delta_L$ such that $v|_D=1$ and $v$ is holomorphic along the discs.
\end{prop}

 The advantage of this canonical diffeomorphism is to simplify a number of estimates or calculations, for example in Lemma \ref{lem:est-diff-g}, for a potential $\varphi$ depending only of the distance in the normal bundle, we will see that the Kähler form $i\partial\dbar \varphi$ is the same when calculated with respect to $J_L$ or $J_X$.

We still denote $t=\log \|v\|_{h_L}^2$, which via the diffeomorphism $\Phi$ we can also see as a function on $U_L$. We modify the function $t$ on $\{ -2\leq t\leq -1\}$ to get a smooth function $\tilde t$ on $X\setminus D$ such that
\begin{equation}
  \label{eq:24}
  \tilde t =
  \begin{cases}
    t & \text{on } t \leq -2, \\ -1 & \text{on }t\geq -1 \text{ and }X\setminus U_L.
  \end{cases}
\end{equation}

We denote $\omega_{TY,L}=(\frac n{n+1})^{1+\frac 1n} i\partial \dbar (-t)^{1+\frac 1n}$ the Tian-Yau metric defined on $\{t<0\} \subset L$, and $g_{TY,L}$ the corresponding Riemannian metric. From (\ref{eq:3}) it is given by the formula
\begin{equation}
  \label{eq:19}
  \omega_{TY,L} = (\tfrac n{n+1})^{\frac1n} \big( \tfrac1n (-t)^{-1+\frac 1n} dt\wedge\eta + (-t)^{\frac 1n} \theta_D \big). 
\end{equation}

Take some Hermitian metric $h$ on $K_X^{-1}$ such that $h|_D=h_L$, and with positive curvature on $X$.  In order to construct such a metric, we can first extend $h_L$ to a positively curved metric $\tilde h$ on a neighborhood of $D$ using a distance function (cf e.g. \cite[Proposition~3.3(i)]{DP}). Then, extend $\tilde h$ arbitrarily to $X$, and consider $h:=\tilde he^{-A|s|^2}$ where $s$ is a section of $\mathcal O_X(D)$ cutting out $D$, $|\cdot|$ is a hermitian metric on the latter bundle with positive curvature and $A>0$ is a large enough constant.

Note that $h|_D$ is only well-defined up to a constant, which will be fixed later in order to have \eqref{eq:16}. Then, the asymptotics of $\omega_{TY,L}$ coincide with those of the metric 
\begin{equation}
 \omega_0 = i \partial \dbar \big(-\tfrac n{n+1}\log |\Omega^{-1}|^2_h\big)^{1+\frac 1n}\label{eq:17}
\end{equation}
on $X\setminus D$, with the corresponding Riemannian metric $g_0$. More precisely, for any $\varepsilon>0$:
\begin{align}
  \label{eq:25}
  |\nabla_{g_{TY,L}}^j(\Phi^*\Omega-\Omega_L)|_{g_{TY,L}}&= O(e^{(\frac 12-\varepsilon)t}),\\
  \label{eq:18}
  |\nabla_{g_{TY,L}}^j(\Phi^*g_0-g_{TY,L})|_{g_{TY,L}} &= O(e^{(\frac 12-\varepsilon)t}). 
\end{align}
This comes from the fact that the objects on $L$ and on $X\setminus D$ coincide near $D$ up to order $O(z)=O(e^{\frac t2})$, but then the form of the metric (\ref{eq:19}) introduces powers of $t$ in the estimates for the differences and their derivatives, so we simply write $O(e^{(\frac12-\varepsilon)t})$ which will be enough for us.

The Tian-Yau metric on $X\setminus D$ is a Kähler metric $\omega_{TY}= i \partial \dbar \varphi_{TY}$ satisfying
\begin{equation}
  \label{eq:26}
  \omega_{TY}^n = \frac{i^{n^2}}{n+1} \Omega \wedge \overline \Omega
\end{equation}
and asymptotic to our Tian-Yau metric $\omega_{TY,L}$ near $D$. Of course (\ref{eq:26}) implies that it is Ricci flat. For a suitable (unique) normalization of $h$, we have the asymptotics 
\begin{equation}
  \label{eq:16}
  \varphi_{TY} = \big(-\tfrac n{n+1}\log |\Omega^{-1}|^2_h\big)^{1+\frac 1n} + \psi, \quad
  |\nabla_{g_{TY,L}}^j\psi|_{g_{TY,L}} = O(e^{-\varepsilon \sqrt{- t}})
\end{equation}
for all $j\geq0$ and for some $\varepsilon>0$. Compared to \cite{HSVZ}, we have a different normalization of the constants in order to match our models of section \ref{sec:calabi-metric}. The rate $e^{-\varepsilon \sqrt{-t}}$ comes from the fact that harmonic functions which go to zero in the metric $g_0$ have exponential decay in $\sqrt{-t}$.

\subsection{The gluing}
\label{sec:gluing}

We now define Kähler metrics $\omega_\beta$ on $X$, with a cone singularity of angle $2\pi \beta$ around $D$, which are close to be Kähler-Einstein with constant $1$, by gluing the metrics $\omega_{\beta,L}$ of section \ref{sec:calabi-metric} with the Tian-Yau metric $\omega_{TY}$. This is done by gluing the corresponding Kähler potentials around $t_\beta=-\beta^{-1+\mu}$ for some $\mu\in(0,1)$, which will be fixed in the final argument in section \ref{sec:resol-kahl-einst}.

We define on $X\setminus D$
\begin{equation}
  \label{eq:20}
  \varphi_\beta =
  \begin{cases}
    \chi(\tfrac t{t_\beta}) \Phi_*\varphi_{\beta,L} + (1-\chi(\tfrac t{t_\beta})) \beta^{1+\frac1n} \varphi_{TY} & \text{ on } U_L, \\
    \beta^{1+\frac1n} \varphi_{TY} & \text{ on } X\setminus U_L,
\end{cases}
\end{equation}
where $\chi:\Bbb{R}_+\rightarrow\Bbb{R}_+$ is a nondecreasing function such that $\chi(u)=1$ for $u\geq2$ and $\chi(u)=0$ for $u\leq \frac12$. We denote $\omega_\beta = i \partial \dbar \varphi_\beta$ and $g_\beta$ the corresponding Kähler form and Riemannian metric.

This metric is very close to our model $\omega_{\beta,L}$ for $t<t_\beta/2$:
\begin{lem}\label{lem:est-diff-g}
 For any $\varepsilon>0$, one has for $\beta$ small enough, uniformly with respect to $\beta$:
\begin{equation}
  \label{eq:23}
  |\nabla_{g_{\beta,L}}^j(g_\beta-g_{\beta,L})|_{g_{\beta,L}} =
  \begin{cases}
    O(e^{(\frac 12-\varepsilon)t})  & \text{ for } t\leq 2t_\beta, \\
 O((-\beta t)^{(1-\frac j2)(1+\frac1n)}) & \text{ for } 2t_\beta \leq t \leq \frac12 t_\beta.
  \end{cases}
\end{equation}
\end{lem}
\begin{proof}
  The main point here is the uniformity with respect to $\beta$. For $t\leq2t_\beta$, we have the same potential $\varphi_{\beta,L}$ but with respect to two different complex structures, that of $L$ and of $X$, that we shall denote $J_L$ and $J_X$. It follows from Proposition~\ref{prop:extremal-discs} that $J_X-J_L$ vanishes on the vertical directions of $L$, and reduces to an endomorphism of $\ker \eta$. Since on $t\leq 2t_\beta$ both Kähler forms have potential $\varphi_{\beta,L}(t)$, and $J_Ldt=J_Xdt=2\eta$, it follows that actually $\omega_{\beta,L}=\frac12 dJ_Ld\varphi_{\beta,L}$ coincides with $\omega_\beta=\frac12 dJ_Xd\varphi_{\beta,L}$ on $t\leq 2t_\beta$. Therefore $g_\beta-g_{\beta,L}=\omega_{\beta,L}(\cdot,(J_X-J_L)\cdot)$, so estimating $g_\beta-g_{\beta,L}$ on this region is the same as estimating $J_X-J_L$.

Since $J_X-J_L$ vanishes on $D$, it follows from formula (\ref{eq:9}) that
  \begin{equation}
|J_X-J_L|_{g_{\beta,L}} = O(e^{\frac t2})\label{eq:27}
\end{equation}
uniformly in $\beta$, since the factor $-\beta \varphi_1'(\beta t)$  in front of $\theta_D$ does not change the norm of the endomorphisms. The covariant derivatives include terms $\frac1{\beta\sqrt{\varphi_1''(\beta t)}} \frac \partial{\partial t}$ and $\frac1{\sqrt{-\beta \varphi_1'(\beta t)}} \frac \partial{\partial x}$ (for $x$ coordinate on $D$). From the behaviour of $\varphi_1$ given in (\ref{t - infty 2}) it follows that the worst coefficient introduced by a covariant derivative is $\beta^{-1} e^{-\beta \frac t2}$. As a result, for any $\varepsilon>0$, we have for $\beta$ small enough and $t\leq 2t_\beta = 2\beta^{-1+\mu}$, uniformly in $\beta$,
\begin{equation}
  \label{eq:29}
  |\nabla^j_{g_{\beta,L}}(J_X-J_L)|_{g_{\beta,L}} = O(e^{(\frac 12-\varepsilon)t}).
\end{equation}

Now pass to the region $2t_\beta \leq t \leq \frac12 t_\beta$. Here we have
\[ \varphi_\beta = \beta^{1+\frac 1n} \varphi_{TY}(t) + \chi(\tfrac t{t_\beta}) \big(\varphi_{\beta,L}(t) - \beta^{1+\frac 1n} \varphi_{TY}(t) \big), \]
with
\begin{align*}
  \varphi_{TY}(t) & = \big(\tfrac{-nt}{n+1}\big)^{1+\frac1n} + \psi, \quad \psi=O(e^{-\varepsilon\sqrt{-t}}), \\
  \varphi_{\beta,L}(t) & = \big(\tfrac{-\beta nt}{n+1}\big)^{1+\frac1n} + O((-\beta t)^{2(1+\frac1n)}).
\end{align*}
(The second line is actually a complete expansion in powers of $(-\beta t)^{1+\frac1n}$).
Since $t_\beta=-\beta^{-1+\mu}$ goes to $-\infty$, the term coming from $\psi$ is negligible and we obtain
\[ \beta^{1+\frac 1n} \varphi_{TY}(t) - \varphi_{\beta,L}(t)  = O\big((-\beta t)^{2(1+\frac1n)}\big).\]
The Kähler form $\omega_{\beta,L}$ is asymptotic to the Tian-Yau form
\[ \beta^{1+\frac1n} \omega_{TY,L} = \beta^{1+\frac1n} (\tfrac n{n+1})^{\frac1n} \big( \tfrac1n (-t)^{-1+\frac1n} dt\wedge\eta + (-t)^{\frac1n} \theta_D \big). \]
Therefore we have
\[ \big|\nabla^j(\beta^{1+\frac 1n} \varphi_{TY}(t) - \varphi_{\beta,L}(t))\big|_{\beta^{1+\frac 1n}g_{TY,L}}=O\big((-\beta t)^{(2-\frac j2)(1+\frac1n)}\big).\]
On the other hand, $|\partial_t^j(\chi(\frac t{t_\beta}))| = O(t_\beta^{-j}) = O(t^{-j})$ so we have the same estimate on the derivatives of $\chi(\frac t{t_\beta})$:
\[ \big|\nabla^j(\chi(\tfrac t{t_\beta}))\big|_{\beta^{1+\frac 1n}g_{TY,L}}=O\big((-\beta t)^{-\frac j2(1+\frac1n)}\big).\]
Since $\varphi_\beta-\varphi_{\beta,L}= (1-\chi(\frac {\cdot}{t_\beta})) \left[ \beta^{1+\frac 1n}\varphi_{TY}-\varphi_{\beta,L}\right]$ we eventually obtain
\begin{equation}
 \big| \nabla^j(\varphi_{\beta}(t) - \varphi_{\beta,L}(t)) \big|_{\beta^{1+\frac 1n}g_{TY,L}}  = O\big((-\beta t)^{(2-\frac j2)(1+\frac1n)}\big).\label{eq:28}
\end{equation}
Since the difference $J_X-J_L$ is exponentially small, differentiating the estimate (\ref{eq:28}) gives the lemma.
\end{proof}

We will solve the Kähler-Einstein equation $\Ric(\omega_\beta+i \partial \dbar \varphi)=\omega_\beta+i \partial \dbar \varphi$ under the form
\begin{equation}
 P_\beta (\varphi) := \log \frac{(\omega_\beta+i \partial \dbar \varphi)^n}{i^{n^2}\Omega \wedge \overline \Omega} + (\varphi_\beta + \varphi) - C_\beta = 0 \label{eq:21}
\end{equation}
where the constant $C_\beta$ is the constant obtained for the model Calabi metric $g_{\beta,L}$, that is $C_\beta=C_1+(n+1)\log \beta$. One can calculate $C_1=-\log(n+1)$, in accordance with (\ref{eq:26}) when one checks that the Tian-Yau metric must be the limit of $\frac{\omega_\beta}{\beta^{1+\frac1n}}$ when $\beta\rightarrow 0$.
We can now estimate the initial error term:
\begin{lem}
 For any $\varepsilon>0$, one has for $\beta$ small enough, uniformly with respect to $\beta$:
\begin{equation}
  \label{eq:22}
  |\nabla^j_{g_\beta}P_\beta (0)|_{g_\beta} =
  \begin{cases}
    O(e^{(\frac 12-\varepsilon)t}) & t \leq 2t_\beta, \\
    O((-\beta\tilde t)^{(1-\frac j2)(1+\frac 1n)}) & \tilde t\geq 2 t_\beta.
  \end{cases}
\end{equation}
\end{lem}
\begin{proof}
  This follows from the estimates in Lemma~\ref{lem:est-diff-g}:
  \begin{itemize}
  \item The form $\omega_{\beta,L}$ solves $\om_{\beta,L}^n=C_\beta e^{\varphi_{\beta,L}}i^{n^2}\Omega_L\wedge \overline \Omega_L$, therefore on $t\leq2t_\beta$, since $\omega_\beta$ and $\Omega$ differ from $\omega_{\beta,L}$ and $\Omega_L$ respectively  by an exponentially decreasing term, we obtain the estimate of the lemma.
  \item On $\tilde t\geq\frac12t_\beta$ the Tian-Yau form $\omega_{TY}$ is Ricci flat and solves \eqref{eq:26}, so the error term in (\ref{eq:22}) is $\varphi_\beta=O((-\beta\tilde t)^{1+\frac 1n}) $ (and the corresponding estimates for the derivatives).
  \item On the gluing region $2t_\beta\leq t\leq\frac12 t_\beta$, the estimates from lemma \ref{lem:est-diff-g} are still sufficient to prove (\ref{eq:22}).
  \end{itemize} 
\end{proof}

\subsection{Convergence of $(X,g_\beta)$ and its rescalings}
\label{sec:geometry}

In this section, we determine all possible Gromov-Hausdorff limits of our model space $(X,\ep_\beta^{-1}g_\beta, p)$ where $p$ is a fixed point and  $(\ep_\beta)_\beta$ is non-decreasing family of positive numbers. At the very end of the paper, we will see that the results of this section continue to hold for the Kähler-Einstein metric $\hat g_\beta$, cf Lemma~\ref{rescaling}.

\subsubsection*{Case A. Convergence of $(X,g_\beta,p)$ and limit renormalized measure.}

In the following, we will use the variable $u=\beta t$, so that the gluing region is $-2\beta^\mu \leq u \leq -\frac12 \beta^\mu$; we set $u_\beta:=-\beta^{\mu}$. For later purposes, we extend the variable $u$ to the whole $X$ by setting $\tilde u=\beta\tilde t$, so that $\tilde u=\beta$ in $X\setminus U_L$. We have the form (\ref{eq:10}) for the metric, which we rewrite here:
\[g_{\beta, L}= 2\varphi_1''(u) (\tfrac14 du^2 + \beta^2\eta^2) - \beta \varphi_1'(u) g_D.\]
Next, we have by Lemma~\ref{lem:est-diff-g} the following estimates 
\begin{equation}
\label{appro}
g_\beta= 
\begin{cases}
g_{\beta, L}+ O(e^{-\beta^{\mu-1}})& \mbox{ if } \quad u\le 2 u_\beta, \\
g_{\beta, L}+ O(\beta^{\mu(1+\frac 1n)}) & \mbox{ if } \quad 2u_\beta \le u\le \frac 12 u_\beta, \\
 \beta^{1+\frac 1n} g_{\rm TY} & \mbox{ if } \quad  u\ge \frac 12 u_\beta.
\end{cases}
\end{equation}
At this point, we can already see that the size of $(u\ge 2u_\beta)$ with respect to $g_\beta$ goes to zero, hence that zone does not contribute to the limit. 

To go further, it is convenient to introduce the moment map $x=-\varphi_1'(u)\in (0,1)$ and then set $\cos(s)=x^{\frac{n+1}2}$ for $s\in (0, \frac \pi 2)$. Relying on the identities $x^{n+1}=1-e^x$ and $\varphi_1''= \frac{1}{n+1}\frac{1-x^{n+1}}{x^{n-1}}$, we get 
\begin{equation}
\label{g s}
g_{\beta, L}= \frac 2{n+1} ds^2+ \frac 2{n+1} \frac{\sin^2 s}{\cos(s)^{2\frac{n-1}{n+1}}} \beta^2 \eta^2 + \beta \cos(s)^{\frac 2{n+1}} g_D.
\end{equation}
Given \eqref{eq:23}, we easily see that $(X,g_\beta,p)$ converges to 
\[(X_\infty, g_\infty,p_\infty)=
\begin{cases}
([0,\tfrac\pi 2], \tfrac 2{n+1} ds^2,0 ) & \mbox{if }  p\in D, \\
([0,\tfrac\pi 2], \tfrac 2{n+1} ds^2,1 )&  \mbox{otherwise.} 
\end{cases}
\]
Moreover, the normalized measures $\nu_\beta=\frac{\mathrm{dvol}_{g_\beta}}{\mathrm{vol}_{g_\beta}(X)}$ converge to the measure $\nu_\infty=c\sin s \cos(s)^{\frac{n-1}{n+1}}ds$ for some $c>0$. In other words, 
\[\nu_\infty = d(-\cos^{\frac{2n}{n+1}}s).\]
The asymptotic behavior of $(X,g_\beta)$ is summarized by the picture below. \\

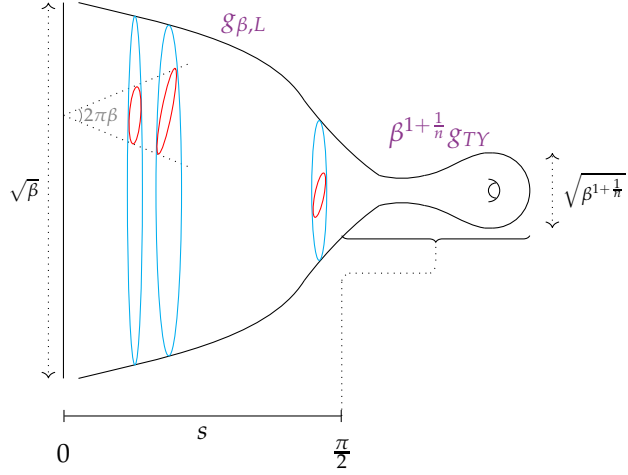
\begin{center}
\begin{tikzpicture}

\draw[dotted, <->] (-0.4,2.5) -- (-0.4,-2.5) node[midway, left] {\tiny $\sqrt \beta$}; 
\draw[dotted, <->] (6.3,-0.5) -- (6.3,0.5) node[midway, right] {\tiny $\sqrt {\beta^{1+\frac 1n}}$}; 

\draw[dotted] (-0.2,1) -- (1.5,1.7);
\draw[dotted] (-0.2,1) -- (1.5,0.3);

\draw[dotted] (3.5,-3) -- (3.5,-1.2);
\draw[decorate, decoration= {brace, mirror}] (3.5,-0.6) -- (6,-0.6);
\draw[dotted] (3.5,-1.2) .. controls (3.525,-1.12) and (3.575,-1.11) .. (3.6,-1.1);
\draw[dotted] (3.6,-1.1) -- (4.675, -1.1);
\draw[dotted] (4.675,-1.1) .. controls (4.7,-1.09) and (4.725,-1.08) .. (4.75, -1);
\draw[dotted] (4.75,-0.75) -- (4.75, -1);

\draw[|-|] (-0.2,-3) -- (3.5,-3) node[midway, below] {\small $s$};
\node[align=left] at (-0.2,-3.5) {\small $0$};
\node[align=right] at (3.5,-3.5) {\small $\frac \pi2$};
\node[align=left, Fuchsia]  at (4.8, 0.8) {\small $ \beta^{1+\frac 1n} g_{TY}$};
\node[align=left, Fuchsia]  at (2.2, 2.2) {\small $ g_{\beta,L}$};

\node[gray, align=left]  at (0.3,0.98) {\tiny $ 2\pi \beta$};

\draw (0,2.5) .. controls (1,2.25) and (2.5,2) .. (3,1.2);
\draw (3,1.2) .. controls (3.3,0.8) and (3.7,0.4) .. (4,0.2);
\draw  (4,0.2) .. controls (4.2,0.15) and (4.4,0.15) ..  (4.6,0.2);
\draw  (4.6,0.2) .. controls (4.9,0.25) and (5.1,0.45) ..  (5.4,0.5);
\draw  (5.4,0.5) --  (5.5,0.5) ;
\draw (-0.2,2.5) -- (-0.2,-2.5) ;

\draw (0,-2.5) .. controls (1,-2.25) and (2.5,-2) .. (3,-1.2);
\draw (3,-1.2) .. controls (3.3,-0.8) and (3.7,-0.4) .. (4,-0.2);
\draw  (4,-0.2) .. controls (4.2,-0.15) and (4.4,-0.15) ..  (4.6,-0.2);
\draw  (4.6,-0.2) .. controls (4.9,-0.25) and (5.1,-0.45) ..  (5.4,-0.5);
\draw  (5.4,-0.5) --  (5.5,-0.5);

\draw (5.5,0.5) arc (90:-90:0.5) ;
\draw (5.45,0.15) arc (90:-90:0.15);
\draw (5.55,0.1) arc (90:270:0.1);
\draw[gray] (0.01,0.91) arc (-40:45:0.13) ;

\draw[cyan] (0.75,0) ellipse (0.1 and 2.32);
\draw[cyan] (1.2,0) ellipse (0.17 and 2.2);
\draw[cyan] (3.2,0) ellipse (0.1 and 0.93);

\draw[red, rotate=-4] (0.68,1.05) ellipse (0.075 and 0.38);
\draw[red, rotate=-11] (0.95,1.26) ellipse (0.068 and 0.58);
\draw[red, rotate=-11] (3.155,0.55) ellipse (0.061 and 0.3);
\end{tikzpicture}
\end{center}

\subsubsection*{Case B. Convergence of $(X,\ep_\beta^{-1} g_\beta,p)$ with $\ep_\beta \to 0$.}
First we consider the case $p\in D$. In other words, we work near $s=0$. Clearly, the zone $(u\ge 2 u_\beta)$ does not contribute to the limit since it escapes any ball centered at $p$ of fixed radius. From \eqref{appro}-\eqref{g s} we get  
\[g_{\beta}\simeq \frac 2{n+1} (ds^2+ s^2 \beta^2 \eta^2) + \beta  g_D\]
and the limit of $(X,\ep_\beta^{-1}g_\beta,p)$ when $p\in D$ is
\[ \begin{cases}
 (\R_+, dt^2,0) & \mbox{if } \quad \beta \ll \ep_\beta \ll 1, \\
 (\R_+\times D, dt^2+g_D, (0,p)) & \mbox{if } \quad \ep_\beta=  \beta, \\
  (\R_+\times \mathbb C^{n-1}, dt^2+g_{\mathbb C^{n-1}}, (0,0)) & \mbox{if } \quad \ep_\beta\ll  \beta. \\
\end{cases}
\]
where the $\mathbb R_+$ factor comes from the rescaling $t=\ep_\beta^{-1/2}s$. \\

Second we consider the case when $p\in X\setminus D$. In other words, we work near $s=\frac \pi 2$. It is convenient to set $\sigma= \frac \pi 2-s$. 

Let us first consider the case where $\ep_\beta \le \beta^{ 1+\frac 1n}$. In that case,  the "Tian-Yau" zone $(u\ge \frac 12 u_\beta)$ has size of order $\ep_\beta^{-\frac 12}\beta^{\frac \mu 2(1+\frac 1n)} \to +\infty$ with respect to $\ep_\beta^{-1} g_\beta$ hence the limit is that of $(X\setminus D, \ep_{\beta}^{-1}\beta^{1+\frac 1n}g_{\rm TY},p)$. 

From now on, one can assume that  $\ep_\beta \gg  \beta^{ 1+\frac 1n}$. Therefore, the compact part is contracted and everything is concentrated near the zone where $u\sim 0$, where the asymptotics \eqref{eq:11bis} hold for $g_{\beta, L}$ as well as for $\beta^{1+\frac 1n}g_{\rm TY}$. Given \eqref{appro}, we get in terms of the variable $\sigma \simeq (-u)^{\frac{n+1}{2n}}$
\begin{equation}
\label{asymp 2}
g_{\beta}\simeq \frac 2{n+1} (d\sigma^2+ \frac{ \beta^2}{\sigma^{2\frac{n-1}{n+1}}} \eta^2) + \beta \sigma^{\frac{2}{n+1}}  g_D.
\end{equation}
In a ball centered at $p$ of fixed radius for $\ep_\beta^{-1} g_\beta$, we have $\ep^{-\frac12}\sigma \lessapprox 1$ (since $\ep_\beta^{-1/2}\sigma(p)\ll 1$), so that \[\ep_\beta^{-1}\beta \sigma^{\frac 2 {n+1}} \lessapprox (\ep_\beta^{-1}\beta^{1+\frac 1n})^{\frac{n}{n+1}}   \ll 1.\]
Moreover, we have $\sigma \gtrapprox \beta^{\frac{n+1}{2n}}$ (since $u \le -\beta$), hence
\[\ep_\beta^{-1}\frac{\beta^2}{ \sigma^{2\frac {n-1}{n+1}}}\lessapprox \ep_\beta^{-1} \beta^{1+\frac 1n} \ll 1.\] 
From \eqref{asymp 2}, we can deduce that in this case the limit is just a ray. \\

In summary, if $p\in X\setminus D$ the limit of $(X,\ep_\beta^{-1}g_\beta,p)$ is
\[ \begin{cases}
   (\R_+, dt^2,0) & \mbox{if } \quad \beta^{ 1+\frac 1n} \ll  \ep_\beta \ll 1  , \\
  (X\setminus D, g_{\rm TY},p) & \mbox{if } \quad \ep_\beta =\beta^{ 1+\frac 1n} , \\
\end{cases}
\]
and for $\ep_\beta \ll \beta^{ 1+\frac 1n}$ we only have the trivial bubble $\mathbb C^{n}$.

\section{Uniform Schauder estimate for cones}
\label{sec:unif-scha-estim}

\subsection{Preliminaries}
In this section, we consider the flat Kähler metric on $\mathbb C^*\times \mathbb C^{n-1}$ with cone angle $2\pi \beta$ along $D:=(z_1=0)$, that is \[dd^c (|z_1|^{2\beta}+\|z'\|^2)=\beta^2|z_1|^{2(\beta-1)}idz_1\wedge d\bar z_1+\sum_{j=2}^n idz_j\wedge d\bar z_j\] where $z'=(z_2, \ldots, z_n)$. Using the real coordinates $r:=|z_1|^\beta, \theta:=\mathrm{arg}(z_1)$, the Riemannian metric associated that Kähler metric is \[\gk:=(dr^2+\beta^2r^2d\theta^2)+g_{\mathbb C^{n-1}}.\] 
It will be convenient to introduce the notation \[\wh g_\beta:=dr^2+\beta^2r^2d\theta^2\] for the one-dimensional complex cone with cone angle $2\pi \beta$ at $0\in \mathbb C$. \\

{\it On balls.} 

\noindent
We are interested in the behavior of $\gk$ near the divisor when $\beta$ approaches $0$. When $n=1$, the zone $\{C\ge |z_1|\ge 1/C\}$ is collapsed onto a point which is at distance exactly one of the origin. This means that the asymptotic geometry is concentrated extremely close to the divisor. In the following, we will only consider with points $p$ at distance at most $\frac 12$ from the origin with respect to $\gk$; in particular, $|z_1(p)|\le 2^{-1/\beta}$ converges exponentially fast to zero.

If $p\in \mathbb C^n$ and $\rho>0$, we denote by $B_p(\rho)$ the geodesic ball of radius $\rho$ centered at $p$, with respect to $\gk$. If $p\in D$, then $B_p(\rho)=\{q=(r_q,\theta_q,z'); r_q^2+\|z'-z'(p)\|^2<\rho^2\}$. 

In the following, we set $B_\beta:=B_0(\frac 12)$ for the ball centered at the origin with radius $1/2$ with respect to $\gk$. It will be convenient to also set $B_\beta':=B_0(\frac 14)$. As explained above, we will exclusively focus on the behavior of $\gk$ on $B_\beta$. Punctured balls $B^*$ are defined as $B\setminus D$.\\

We decompose the gradient of $\gk$ as  $\nabla^{\gk}=(D',D'')$ where $D'=(D_1, D_2)$ with $D_1=\d_r, D_2=\frac{1}{\beta r}\d_\theta$ and $D''= (D_3, \ldots, D_{2n}) $ where $D_{2j-1}=\d_{x_j}, D_{2j}=\d_{ y_j}$ if $z_j=x_j+iy_j$. The laplacian is
\[\Delta_{\gk}=\underbrace{\d^2_{r}+\frac1r\d_r+\frac{1}{\beta^2r^2}\d^2_\theta}_{=\Delta_{\wh g_\beta}}+\Delta_{\mathbb C^{n-1}}.\]
In complex coordinates, the first order derivatives are given by  \[\d_r=\frac{1}{\beta|z_1|^\beta}(z_1\d_{z_1}+\bar z_1\d_{\bar z_1}), \quad \text{and} \quad \frac1{\beta r}\d_\theta=\frac{i}{\beta |z_1|^\beta}(z_1\d_{z_1}-\bar z_1\d_{\bar z_1})\] 
while the second order derivatives are 
\begin{equation}
\label{dr2}
\d_r^2=\frac{1}{\beta^2|z_1|^{2\beta}}\left(2|z_1|^2\d^2_{z_1\bar z_1}+(1-\beta)(z_1\d_{z_1}+\bar z_1\d_{\bar z_1})+(z_1^2\d_{z_1}+\bar z_1^2\d_{\bar z_1})\right)
\end{equation}
and 
\begin{equation}
\label{dt2}
\frac1{\beta^2r^2}\d^2_\theta=\frac{-1}{\beta^2|z_1|^{2\beta}}\left(z_1^2\d^2_{z_1}+ z_1\d_{z_1}-2|z_1|^2 \d^2_{z_1\bar z_1}+ \bar z_1\d_{\bar z_1}+\bar z_1^2\d_{\bar z_1}\right).
\end{equation}

For any real number $\alpha \in (0,1)$, we define the $\mathcal C^{\alpha}$ norm with respect to $\gk$ in the classical way. That is, if $p\in B_\beta$  i.e. if $f\in \mathcal C^0(B_\beta)$, then 
\[\|f\|_\alpha= \sup_{x,y \in B_\beta}\frac{|f(x)-f(y)|}{d_{\gk}(x,y)^\alpha}.\]
We defined the $\mathcal C^{2,\alpha}$ norm of a function $f$ on $B_\beta$ as
\[\|f\|_\alpha=\|f\|_{0}+\|\nabla^{\gk}f\|_0+\sum_{i=1}^{2n}\sum_{j=3}^{2n} \|D_iD_jf\|_\alpha+\|\Delta_{\wh g_\beta} f\|_\alpha.\]
This means that we do not require a control on the derivatives $D_1^2f$, $D_2^2f$, $D_1D_2f$, $D_2D_1f$, following Donaldson \cite{Don}. 


The main result of this section is the following
\begin{thm}[Weak Schauder estimate]
\label{Schauder conique}
Let $u\in L^{\infty}(B_\beta)$ solve $\Delta_{\gk}u=f$ for some $f\in \mathcal C^{\alpha}(B_\beta)$. Then $u\in \mathcal C^{2,\alpha}(B_\beta)$ and there exists a constant $C=C(n,\alpha, \delta)$ such that for all $\beta\in (0,1-\delta]$ one has 
\[\|u\|_{\mathcal C^{2,\alpha}(B_\beta')} \le C \left( \|f\|_{\mathcal C^{\alpha}(B_\beta)}+\|u\|_{\mathcal C^0(B_\beta)}\right).\]
\end{thm}

The novelty of the result above relies on the uniformity of the "Schauder constant" $C$ above with respect to $\beta$ (say when $\beta\to 0$), since the result for fixed $\beta$ has been known for a while. It is initially due to Donaldson \cite{Don} and was later reproved and generalized via many different methods, cf \cite{GS, GY, BE}. 

We will follow the approach of Bin Guo and Jian Song \cite{GS}, itself based on an original and quite direct proof of the usual Schauder estimate by Xu-Jia Wang \cite{Wang06}. 

In what follows, we will systematically assume that $\beta<\frac 12$. 

\subsection{Gradient estimates}

In this section, we provide two types of gradient estimates for the metric $\gk$ that will be useful later. 

\begin{lem}
\label{gradient}
Assume that $n=1$. Let $u\in L^{\infty}(B_0(\rho))$, smooth outside $0$ solving $\Delta_{\wh g_\beta}u=f$. Then there exists a constant $C>0$ such that for any $r\le \frac{\rho}{2^\beta}$, one has
\[|\nabla^{\wh g_\beta}u|(r,\theta) \le C \Big[ \frac{1}{\beta}\left(\frac r{\rho}\right)^{\frac1\beta}\cdot \frac 1r  \sup_{B_0(\rho)} |u|+r \sup_{B_0(\rho)} |f|.\Big]\]
\end{lem}

\begin{rem}
A trivial but crucial observation is that when $r\le \rho/2$, we have $ \frac{1}{\beta}(\frac r{\rho})^{\frac1\beta}\to 0$ when $\beta\to 0$. In particular, and this is all we will use in the following, the latter quantity is bounded when $\beta$ approaches zero. 
\end{rem}

\begin{proof}
The function $v(z):=u(\rho^{1/\beta}z)$ is defined for $|z|\le 1$ and satisfies $\Delta_{\rm eucl}v(z)=\rho^2\frac{\beta^2}{|z|^{2(1-\beta)}}f(\rho^{1/\beta}z)$. Classically, we have
\begin{equation}
\label{green}
v(z)=h(z)+\underbrace{\rho^2\int_{|w|<1}\frac{\beta^2}{|w|^{2(1-\beta)}} f(\rho^{1/\beta} w)\log\left|\frac{z-w}{1-\bar w z}\right| |dw|^2}_{=:I(z)},
\end{equation}
where $h$ is the harmonic function on the unit disk $\mathbb D \subset \mathbb C$ whose boundary values are $v|_{\d \mathbb D}$. In particular, there exists a universal constant $C_1$ such that 
\begin{equation*}
\sup_{|z|\le \frac 1 2}|\nabla^{\rm eucl} h(z)|\le C_1\sup_{|z|=1} |v(z)|= C_1 \sup_{B_0(\rho)} |u|.  
\end{equation*}
In particular, we get for $|\rho^{-1/\beta}z|\le 1/2$ (or, equivalently, $r\le 2^{-\beta}\rho$): 
\begin{equation}
\label{harmonic gradient}
|\nabla^{\wh g_\beta}h(\rho^{-1/\beta}z)|\le \frac{C_1}{\beta}\left(\frac r{\rho}\right)^{\frac1\beta}\cdot \frac 1r  \sup_{B_0(\rho)} |u|,
\end{equation}
which takes care of the first part in the RHS of \eqref{green}. To take care of the integral summand $I(z)$, we assume that $|z|\le 1/2$ so that $\nabla^{\rm eucl}I(z)$ is controlled by $\rho^2\beta^2\sup |f| \cdot \int_{|w|<1}\frac{ |dw|^2}{|w|^{2(1-\beta)}\cdot |z-w|} $. Performing the change of variable $x:=w/z$ in the integral, we get $|z|^{2\beta-1}\int_{|x|<1/|z|}\frac{ |dx|^2}{|x|^{2(1-\beta)}\cdot |x-1|}$. There are three zones: around $x=0$ the integral is equivalent to $1/\beta$, around $x=1$ we have uniform integrability while around $|x|=1/|z|$ it is equivalent to $|z|^{1-2\beta}$ hence it is uniformly integrable too. All in all, we find that $|\nabla^{\rm eucl}I(z)|\lesssim \beta\rho^2|z|^{2\beta-1}\cdot \sup |f|$. In terms of $\gk$-gradient, this means that 
\[|\nabla^{\wh g_\beta}I(\rho^{-1/\beta}z)|\lesssim \frac 1\beta r^{\frac 1\beta-1}\rho^{-1/\beta}\cdot \beta\rho^2|\rho^{-1/\beta}z|^{2\beta-1} \cdot \sup |f|= r\cdot \sup |f|.\]
Combined with \eqref{harmonic gradient}, this yields the desired gradient estimate for $u$. 
\end{proof}

We will also need the following gradient estimate in any dimension for harmonic functions: 

\begin{lem}
\label{gradient 2}
Let $u\in L^{\infty}(B_p(\rho))$ be a harmonic function, i.e. $\Delta_{\gk}u=0$. Assume that either $p=0$ or $\rho<r(p)$. Then there exists a universal constant $C=C(n)>0$ such that 
\[\sup_{B_p(\rho/2)}|\nabla^{\gk}u| \le  \frac C \rho  \sup_{B_p(\rho)} |u|.\]
In particular, for any integer $\ell>0$ that 
\[\sup_{B_p(\rho/2)}|(D'')^\ell u| \le  \frac {C(n,\ell)} {\rho^{\ell}}  \sup_{B_p(\rho)} |u|; \quad \sup_{B_p(\rho/2)}|(D'')^\ell D' u| \le  \frac {C(n,\ell)} {\rho^{\ell+1}}  \sup_{B_0(\rho)} |u|\]
as well as, if $p$ is not on the divisor
\begin{equation}
\label{gradient 5}
 \left|\nabla^{\gk}\big[\d_r(D'')^\ell u\big](z)\right|+ \left|\nabla^{\gk}\big[\frac{1}{\beta r}\d_\theta(D'')^\ell u\big](z)\right|  \le C (1+\frac{r(p)}r) \cdot \frac {\|u\|_{L^{\infty}(B_p(\rho))}} {\rho^{\ell+2}} .
 \end{equation}
\end{lem}

\begin{proof}
Following \cite[Lemma~2.4 \& Proposition~2.2]{GS}, one can approximate $\gk$ by smooth metrics $\bar g_{\beta,\ep}$ with non-negative Ricci curvature and use Cheng-Yau's gradient estimate \cite{CY3} to get the first two sets of inequalities. For the last estimate, set $v:=(D'')^{\ell}u$ and  observe that $\frac 1\beta\d_{\theta} v$ is harmonic on $B_p(\rho)$ and therefore
\[\left|\frac 1\beta\d_{\theta} v\right| \le C \frac {r\|u\|_{L^{\infty}(B_p(\rho))}} {\rho^{\ell+1}} \quad \text{on } B_p(\rho/2)  \] thanks to the previous gradient estimate. Here, $C=C(n,\ell)$ and may change from line to line. Iterating that argument, we get 
\begin{equation}
\label{nabla 2}
\left|\nabla^{\gk}\frac 1\beta\d_{\theta} v\right| \le    C \frac {r(p) \|u\|_{L^{\infty}(B_p(\rho))}} {\rho^{\ell+2}} \quad \text{on } B_p(\rho/2).
 \end{equation}
Since $|\nabla^{\gk} \frac1{\beta r} \d_\theta v|\le  \frac{1}{r^2}\left|\frac 1\beta\d_{\theta} v\right|+\frac2r \left|\nabla^{\gk} \frac1{\beta} \d_\theta v\right|$, we get
\[|\nabla^{\gk} \frac1{\beta r} \d_\theta v| \le C (\frac \rho r+\frac{r(p)}r) \cdot \frac {\|u\|_{L^{\infty}(B_p(\rho))}} {\rho^{\ell+2}}\]
which provides half of the desired inequality. For the second half, observe that $\nabla^{\gk}\d_rv$ involves the following terms: $\d^2_r v$,$\frac 1r \d_r(\frac 1\beta\d_\theta v)$ and $\nabla''\d_rv$. The last term is controlled by the gradient estimate already established for the harmonic function $D''v$ and the second one is controlled by \eqref{nabla 2}. Finally, the first one can be written $\d_r^2v=\underbrace{\Delta_{\gk}v}_{=0}-\frac 1r\d_rv-\frac{1}{\beta^2r^2}\d_\theta^2v-\Delta_{\mathbb C^{n-1}}v$ and the estimate follows from the previous ones.
\end{proof}


In Lemma~\ref{gradient 2} above, one can adapt the proof of the estimate \eqref{gradient 5} in the case where $p=0$ is centered on the divisor, and the RHS becomes $C  \frac {\|u\|_{L^{\infty}(B_p(\rho))}} {r\rho^{\ell+1}}$, which turns out to be too coarse for our later purposes. Instead, we will use the input of Lemma~\ref{gradient} to obtain the following estimate, valid for balls centered on the divisor. 

\begin{lem}
\label{gradient 3}
Let $u\in L^{\infty}(B_0(\rho))$ be a harmonic function, i.e. $\Delta_{\gk}u=0$. There exists $C=C(n,\ell)$ such that for all $z=(r,\theta, z')$ in $B_0(\rho/4)$, one has 
\[ \left|\nabla^{\wh g_\beta}\big[(D'')^\ell u\big](z)\right|+\left|\nabla^{\wh g_\beta}\Big[\frac 1\beta \d_\theta (D'')^\ell u\Big](z)\right| \le C \frac {\|u\|_{L^{\infty}(B_0(\rho))}} {\rho^{\ell+1}} \cdot \left(\frac{1}{\beta}\left(\frac r{\rho}\right)^{\frac1\beta-1}+\frac r\rho\right) \]
as  tas 
\begin{equation}
\label{gradient 6}
 \left|\nabla^{\wh g_\beta}\big[\d_r(D'')^\ell u\big](z)\right|  \le C \frac {\|u\|_{L^{\infty}(B_0(\rho))}} {\rho^{\ell+2}} \cdot \left(\frac{1}{\beta}\left(\frac r{\rho}\right)^{\frac1\beta-2}+1\right).
 \end{equation}
 \end{lem}
 
 \begin{proof}
 The two important points are that $\Delta_{\gk}$ commutes with both $D''$ and $\d_\theta$ and that for any $\gk$-harmonic function $w$, one has  $\Delta_{\wh g_\beta} w=-\Delta_{\mathbb C^{n-1}}w$. Set $v:=(D'')^\ell u$. By Lemma~\ref{gradient 2}, \[\sup_{B_0(\rho/2)} \big[|v|+|\frac 1\beta \d_\theta v|\big] \le C \frac {\|u\|_{L^{\infty}(B_0(\rho))}} {\rho^{\ell+1}}\] 
so that the first inequality now easily follows from Lemma~\ref{gradient}.

The second inequality requires a bit more work.  We start by decomposing
\[\nabla^{ \gk}\d_rv = (\d_r^2v,\frac{1}{\beta r}\d^2_{r\theta}v, D''\d_rv) \]
and observing that the last two components are controlled on $B_0(\rho/4)$ by 
\[\sup_{B_0(\rho/2)}\left[ |\nabla^{\wh g_\beta}D'' v|+\frac 1r |\nabla^{\wh g_\beta}\frac 1\beta \d_\theta v|\Big]\right],\]
which in turn is controlled by $ M:=  \frac {\|u\|_{L^{\infty}(B_0(\rho))}} {\rho^{\ell+2}} \cdot (\frac{1}{\beta}\left(\frac r{\rho}\right)^{\frac1\beta-2}+1)$ thanks to the first inequality. We are left to estimating $ \d_r^2v$. We write
\begin{align*}
 \d_r^2v&=\Delta_{\wh g_\beta}(v)-\frac 1r \d_r v-\frac 1{\beta^2r^2}\d^2_{\theta}v\\
 &=-\Delta_{\mathbb C^{n-1}}(v)-\frac 1r \d_r v-\frac 1{\beta^2r^2}\d^2_{\theta}v
  \end{align*} 
  and observe that each summand is controlled by \[\sup_{B_0(\rho/2)}\left[ |(D'')^{2}v|+\frac 1r\Big[|\nabla^{\wh g_\beta}v|+|\nabla^{\wh g_\beta}\frac 1\beta \d_\theta v|\Big]\right] \le C M\]
 thanks to the first inequality again. The lemma is proved.
 \end{proof}

\subsection{Strategy of the proof of Theorem~\ref{Schauder conique}}
The main idea is to consider, for a given point $p\in B_\beta$ a sequence of functions $(u_\k)$ defined on neighborhoods $U_\k$ of $p$ getting smaller and smaller when $\k$ increases and such that 
\[\begin{cases} 
\Delta_{\gk}u_\k=f(p) &\text{on } U_\k\\
u_\k|_{\partial U_\k}= u|_{\partial U_\k}
\end{cases} 
\]
More precisely, let us set $\lambda:=\frac 12$, $r(p):=d(p,D)$ and choose $U_\k=B_p(\lambda^\k)$ if $r(p)> \lambda^{\k}$ and $U_\k:=B_{\tilde p}(2\lambda^\k)$ otherwise, where $\tilde p$ is the projection of $p$ onto $D$ under the natural map $\mathbb C^n\to \mathbb C^{n-1}, (z_1, \ldots, z_n)\to (z_2, \ldots, z_n)$. Note that if $\sin(\beta\pi) r(p)<\rho<r(p)$, then the geodesic ball $B_p(\rho)$ is essentially an annulus (times an euclidean ball).   

Since $u_\k-\frac1{n-1}\|z'\|^2$ is harmonic, $u_\k$ enjoys all the regularity properties shared by harmonic functions.\\

The strategy is, given two indices $i,j$ and two points $p,q\in B_\beta$, to estimate $D_iD_ju(p)-D_iD_ju(q)$ by the analogous quantity for $u_\k$, for some $\k=\k(p,q)$ chosen carefully. More precisely, the choice of $\k$ with be such that $\lambda^{\k}\simeq 8d$ where $d=d(p,q)$. In particular, $U_\k$ will contain $B_p(2d)$ and thus the geodesic joining $p$ to $q$.

\subsection{$\mathcal C^2$ estimates}
It is convenient to write $\omega(r)$ for the modulus of continuity of $f$. By assumption, one has $\omega(r)=O(r^{\alpha})$. By considering $u-u_\k\pm \omega(\lambda^\k)\cdot \big[(r-r(p))^2+\|z'-z'(p)\|^2\big]$, one easily deduces from the maximum principle that 
\begin{equation}
\label{u and u_k}
\|u-u_\k\|_{L^{\infty}(U_\k)} \le C(n) \lambda^{2\k} \omega(\lambda^\k). 
\end{equation}
By the triangle inequality, these inequality extend to quantify the {\it harmonic} functions 
\[h_\k:=u_{\k+1}-u_\k\] 
on $U_{\k+1}$ along with their derivatives thanks to Lemma~\ref{gradient 2}
\begin{equation}
\label{u and u_k derivatives}
\|D''h_\k\|_{L^{\infty}(U_{\k+2})} \le C(n) \lambda^{\k} \omega(\lambda^\k); \quad \|(D'')^2h_\k\|_{L^{\infty}(U_{\k+2})} \le C(n) \omega(\lambda^\k).
\end{equation}
For $\k\gg1$, one can define on $U_\k$ a single-valued branch $w=z_1^\beta$ realizing an isomorphic biholomorphism between $(U_\k,\gk)$ and a euclidean ball $(B_{\rm eucl}(\lambda^\k), g_{\rm eucl})$.  Using this, one gets that whenever $p\notin D$,  
\begin{equation}
\label{u = lim u_k}
\lim_{\k \to +\infty} D'' u_\k(p)=D''u(p); \quad \lim_{\k \to +\infty} (D'')^2 u_\k(p)=(D'')^2u(p), 
\end{equation}
cf \cite[Lemma~2.8]{GS}.
Write $(D'')^2u_\k=(D'')^2u_1+\sum_{j=1}^{\k-1}(D'')^2(u_{j+1}-u_j)$ on $U_\k$ and  then combine \eqref{u = lim u_k}, \eqref{u and u_k derivatives} and Lemma~\ref{gradient 2} to obtain
\begin{equation}
\label{control 2 derivatives}
\|(D'')^2u\|_{L^{\infty}(B'_\beta)} \le C(n,\alpha)\big[ \|u\|_{L^{\infty}(B_\beta)}+\|f\|_{\mathcal C^{\alpha}(B_\beta)} \big].
\end{equation}
Since $\Delta_{\wh g_\beta}u=f-\sum_{j= 3}^{2n} D_j^2u$, \eqref{control 2 derivatives} provides a bound for $\|\Delta_{\wh g_\beta}u\|_{L^{\infty}(B_\beta')}$ in terms of $ \|u\|_{L^{\infty}(B_\beta)}$ and $ \|f\|_{L^{\infty}(B_\beta)}$.

\subsection{$\mathcal C^\alpha$ estimates for the tangential derivatives}
\label{subsec tangential}
Let $p,q\in {B'_\beta}^*$ and let $d=d_{\gk}(p,q)$. By \cite[Proposition~2.3]{GS}, we have
\begin{equation}
\label{tangential}
\left|(D'')^2u(p)-(D'')^2u(q)\right| \le C(n,\alpha) \Big[ d\|u\|_{L^{\infty}(B_\beta)}+d^\alpha\|f\|_{\mathcal C^{\alpha}(B_\beta)}\Big].
\end{equation}
For the reader's convenience, we recall the main steps. We introduce the functions $v_\k$ playing the role of $u_\k$ but for the point $q$ instead of $p$. Choose $\k$ such that $d\simeq \lambda^{\k+3}$ and assume $r(p)=\min(r(p),r(q))\le 2d$ for simplicity. We have essentially three terms to treat
\[\underbrace{(D'')^2u(p)-(D'')^2u_\k(p)}_{=:\text{(I)}};  \underbrace{(D'')^2u_\k(p)-(D'')^2u_\k(q)}_{=:\text{(II)}}; \underbrace{(D'')^2u_\k(q)-(D'')^2v_\k(q)}_{=:\text{(III)}}.\]
The first term is easy to handle: 
\begin{equation*}
\left|\text{(I)}\right|=\lim_{N\to +\infty}\left| \sum_{j=\k}^N (D'')^2h_j(p)\right|
\le C(n) \sum_{j=\k}^{+\infty} \om(\lambda^j) 
 \le C(n,\alpha) d^\alpha\|f\|_{\mathcal C^{\alpha}(B_\beta)}.
\end{equation*}
For the second term, we use the gradient estimate \eqref{gradient 2} for the harmonic function $(D'')^2h_j$ ($2\le j\le \k-1$): 
\[\sup_{B_{\tilde p}(\lambda^{j})}|\nabla^{\gk}(D'')^2h_j| \le   C(n) \lambda^{-j} \omega(\lambda^j)\]
and after integration along the geodesic joining $p$ anq $q$ (which lies in $B_{\tilde p}(\lambda^\k)^*$)
\[\left|(D'')^2u_{j+1}(p)-(D'')^2u_{j+1}(q)\right|\le \left|(D'')^2u_{j}(p)-(D'')^2u_{j}(q)\right|+C(n) d \lambda^{-j} \omega(\lambda^j)\]
and by iterating
\[\left|\text{(II)}\right|\le \left|(D'')^2u_{2}(p)-(D'')^2u_{2}(q)\right|+C(n) d^\alpha\|f\|_{\mathcal C^{\alpha}(B_\beta)}.\]
The first term in the RHS is almost harmonic on a ball of definite size, so by using the gradient estimate \eqref{gradient 2}, on can dominate it by $C(n)d \|u\|_{L^{\infty}(B_\beta)}$.
As for the third term, the function $u_\k-v_\k$ is well-defined on $B_{\tilde q}(\lambda^\kappa)$, it is almost harmonic and its sup-norm on that ball is of order $\lambda^{2\k}\omega(\lambda^\k)$ by \eqref{u and u_k}. The gradient estimate for harmonic functions (Lemma~\ref{gradient 2}) then provides the desired estimate.

\subsection{$\mathcal C^\alpha$ estimates for the normal-tangential derivatives}
\label{normal tangential} In this paragraph, we explain the following estimate, cf  \cite[Propositions~2.4\&2.5]{GS}. The argument is somehow simplified here because we can choose an angle $2\pi \beta <\pi$; this will simplify the application of  Lemma~\ref{gradient 3}.  Let $p,q\in {B'_\beta}^*$ and let $d=d_{\gk}(p,q)$; then

\begin{equation}
\label{normal}
\sum_{i=1}^{2}\sum_{j=3}^{2n} \left|D_iD_ju(p)-D_iD_ju(q)\right|  \le C(n,\alpha) \Big[ d\|u\|_{L^{\infty}(B_\beta)}+d^\alpha\|f\|_{\mathcal C^{\alpha}(B_\beta)}\Big].
\end{equation}
Again, we will only highlight the main steps, and focus on the $i=1$ case; i.e. we estimate the Hölder constant of $\d_rD_ju$ for any $j\ge 3$. The case $i=2$, i.e. estimating $\frac{1}{\beta r}\d_\theta D_ju$ is very similar. Borrowing the notation from \textsection~\ref{subsec tangential}, Lemma~\ref{gradient 2} shows that the harmonic function $h_\k$ satisfies 
\begin{equation}
\label{gradient 4}
\sup_{U_{\k+2}}|\nabla^{\gk}D''h_\k| \le   C(n) \omega(\lambda^\k).
\end{equation}
Similarly to \eqref{u = lim u_k}, we have 
\begin{equation}
\label{u = lim u_k 2}
\lim_{\k \to +\infty} \d_r u_\k(p)=\d_r u(p); \quad \lim_{\k \to +\infty} \d_rD'' u_\k(p)=\d_rD''u(p), 
\end{equation}
cf \cite[Lemma~2.10]{GS}. To estimate $\d_rD''u(p)-\d_rD''u(q)$, we fix the integer $\k$ so that $d\simeq \lambda^{\k+3}$ and we need to analyze the analogous terms $\text{(I)'}:=\d_rD''u(p)-\d_rD''u_\k(p)$, \[\text{(II)'}:=\d_rD''u_\k(p)-\d_rD''u_\k(q),\] 
and $\text{(III)'}:= \d_rD''u_\k(q)-\d_rD''v_\k(q)$. The first term is dealt with just as in \textsection~\ref{subsec tangential} and the third one relies on the same arguments as before along with \eqref{gradient 4}-\eqref{u = lim u_k 2}, cf \cite[Lemma~2.11]{GS}. In the following, we thus focus on $\text{(II)'}$. As for its analog $\text{(II)}$, the key point is to estimate
 \[\text{(II)''}:=\d_rD''h_j(p)-\d_rD''h_j(q),\] 
for any $2\le j \le \k-1$ since $u_2$ is almost harmonic on $B_{\tilde p}(\lambda)$ and Lemma~\ref{gradient 3} shows that the $\gk$-gradient of $\d_rD''u_2$ is bounded on that ball. Set $w_j:=\d_rD''h_j$, defined on $U_{j+1}$. We distinguish two cases. 

\smallskip

$\bullet$ {\it Case 1}. $U_\k$ is centered on the divisor.

\noindent 
In particular, any $U_j$ ($2\le j \le \k-1$) is centered on the divisor as well. The estimate \eqref{gradient 6} in Lemma~\ref{gradient 3} shows that for $\beta$ small enough 
($\beta<\frac 12$ would suffice), $\nabla^{\gk}w_j$ is bounded on $B_{\tilde p}(\frac 32 \lambda^{j}) \supset B_p(2d)$ by $C(n)\lambda^{-j}\omega(\lambda^j)$. 

\smallskip 

$\bullet$ {\it Case 2}. $U_\k$ is centered at $p$.

\noindent
Necessarily, we have $r(p)\ge \lambda^\k \simeq 8d$. Along a geodesic $\gamma(t)$ joining $p$ to $q$, we have the inequality $r(\gamma(t))\ge \frac {r(p)}2$ since the distance from $p$ to a point $p'$ with $r(p')\le r(p)/2$ is at least $r(p)/2 \ge 2d$. The geodesic $\gamma$ lies in $B_p(\lambda^\k)$ hence equation~\eqref{gradient 5} in Lemma~\ref{gradient 2} shows that for any $j$, $\nabla^{\gk}w_j$ is bounded by $C(n)\lambda^{-j}\omega(\lambda^j)$ along $\gamma$. 

\medskip
The case by case analysis above has therefore shown that $|w_j(p)-w_j(q)| \le C(n) \lambda^{-j}\omega(\lambda^j) \cdot d$ and we can conclude as in \textsection~\ref{subsec tangential}. 
%

\subsection{Strong Schauder estimate}

In this section, we intend to improve Theorem~\ref{Schauder conique} by controlling the $\mathcal C^{\alpha}$ norm of the non-mixed derivatives of order two of a solution $u$ of the equation $\Delta_{\gk}u=f$, that is to get an estimate of $\|\nabla^2_{\gk}u\|_{\mathcal C^{\alpha}}$. 
As we will later see, it all comes down to the following one-dimensional problem.

\begin{prop}
\label{schauder dim 1}
Assume that $n=1$. Let $u\in L^{\infty}(B_\beta)$ solve $\Delta_{\gk}u=f$ for some $f\in \mathcal C^{\alpha}(B_\beta)$. Then $u\in \mathcal C^{2,\alpha}(B_\beta)$ and there exists a constant $C=C(n,\alpha)$ such that for all $\beta\in (0,\frac 14]$ one has 
\begin{equation}
\label{eq schauder} 
\|u\|_{\mathcal C^{2,\alpha}(B_\beta')} \le C \left( \|f\|_{\mathcal C^{\alpha}(B_\beta)}+\|u\|_{\mathcal C^0(B_\beta)}\right).
\end{equation}
\end{prop}

\begin{cor}[Strong Schauder estimate]
\label{cor strong schauder}
The full Schauder estimate \eqref{eq schauder} holds in any dimension. 
\end{cor}

Here, the $\mathcal C^{2,\alpha}$ norm $D\subset B_\beta$ is defined by 
\[  \|u\|_{C^{2,\alpha}} =  \sup_{B_\beta} \sum_{0\leq j\leq 2} |\nabla^j u|_{\gk} + [ \nabla^2 u]_\alpha\]
where
\begin{equation}
\label{alpha}
[v]_{\alpha} = \sup_{x,y \in B_\beta; |r(x)-r(y)| < \frac{r(x)}2} \frac{|v(x)-v(y)|}{d_{\gk}(x,y)^\alpha}.
\end{equation}

This Hölder semi-norm is quite convenient to manipulate as we will see later, and it is well-known that it is equivalent to the usual Hölder semi-norm $\|v\|_{\alpha}:=  \sup_{x,y } \frac{|v(x)-v(y)|}{d_{\gk}(x,y)^\alpha}$.

\begin{rem}\label{rem:strong-scha-estim}
  By classical arguments (see e.g. \cite[\textsection~4.2]{Don}, \cite[\textsection~3.5]{GS2}), the result of Corollary \ref{cor strong schauder} for the flat cone metric extends to perturbations $\Delta_{\gk}+ a_\beta \cdot \nabla^2 + b_\beta \cdot \nabla$, provided the tensors $a_\beta$ and $b_\beta$ are small enough in $C^\alpha$ norm. It is easy to check that the uniformity with respect to $\beta$ is preserved provided $a_\beta$ and $b_\beta$ are uniformly small, more precisely for some $\varepsilon=\varepsilon(n,\alpha)>0$ small enough, one has for all $\beta\in(0,\frac14]$
  \[ \|a_\beta\|_{C^\alpha(\gk)} + \|b_\beta\|_{C^\alpha(\gk)} < \varepsilon. \]

\end{rem}

\begin{proof}[Proof of Proposition~\ref{schauder dim 1}]
The proof of Proposition~\ref{schauder dim 1} consists in three steps. First, we show that is is enough to prove the estimate for functions $u$ which vanish on $\partial B_\beta$ and whose integral on every circle $\{r=\mathrm{cst}\}$ is zero. Next, we show that for such functions $u$, the norm $\|u/r^{2+\alpha}\|_{\mathcal C^0}$ is controlled by $\|u\|_{\cC^0}+\|f\|_{\cC^\alpha}$. Finally, we combine the previous results and Schauder's estimate for the cylindrical metric to conclude. \\

$\bullet$ {\it Step 1. The reduction step}. First, we decompose $u=h+v$ where $h$ is harmonic on $B_\beta$ with $h|_{\partial B_\beta}=u|_{\partial B_\beta}$ and $v$ solving $\Delta_{\gk}v=f$, $v|_{\partial B_\beta}\equiv 0$. All the usual derivatives of $h$ are controlled by it sup norm, itself controlled by its boundary value, hence by $\|u\|_{\cC^0}$. Moreover, the formulas \eqref{dr2}-\eqref{dt2} show that $\|\nabla^3_{\gk}h\|_{\cC^0}$ is controlled by $\beta^{-3}r^{1/\beta-3} \|\nabla^3_{\rm eucl} h\|_{\cC^0}$. Therefore, Schauder's estimate for $v$ implies that for $u$.

Next, we write $\mathbb C^*=\mathbb R_+^*\times  S^1$ and we expand $v$ in Fourier series $v=v_0(r)+\tilde v$ where $\tilde v= \sum_{n\ge 1} v_n(r) e^{in\theta} $. The function $\tilde v$ has integral zero on each circle, hence $\tilde v|_{\partial B_\beta} \equiv 0$. As $v_0(r)=\frac{1}{2\pi r}\int_{|z|=r}v$, both $\|v_0\|_{\cC^0}$ and $\|\tilde v\|_{\cC^0}$ are controlled by $\|v\|_{\cC^0}$. 
Since $\Delta_{\gk}$ respects the decomposition, the Fourier series expansion of $f=f_0(r)+\tilde f$ is given by $\Delta_{\gk}v_0+\Delta_{\gk}\tilde v$. It is easy to check that $\|f_0\|_{\cC^\alpha}\le \|f\|_{\cC^\alpha}$. From this one infers two things: first, $\|v_0\|_{\cC^{2,\alpha}}$ is under control (e.g. by explicitly solving the ODE $(\d^2_r+r^{-1}\d_r)v_0=f_0$) and next,  $\tilde f$  is $\cC^\alpha$ and $\|\tilde f\|_{\cC^\alpha}$ is under control as well.  

Therefore, Schauder's estimate for $\tilde v$ implies Schauder's estimate for $v$, hence for $u$ as well. This shows that it is enough to restrict ourselves to functions $u$ which vanish on $\partial B_\beta$ and whose integral on each circle $\{|z|=r\}$ is zero.  \\

$\bullet$ {\it Step 2. The improved uniform estimate}. In this step, we show that for $u$ satisfying $u|_{\partial B_\beta}\equiv 0$ and $\int_{|z|=r}u=0$ for any $r$, then there exists a constant $C>0$ independent of $\beta$ such that 
\begin{equation}
\label{improved}
\|u/r^{2+\alpha}\|_{\cC^0(B'_\beta)} \le C \|f\|_{\cC^\alpha(B_\beta)}. 
\end{equation}
Set $\rho:=\frac 12$. From \eqref{green}, we have 
\begin{equation}
\label{green2}
u(z)= \rho^2\int_{|w|<1}\frac{\beta^2}{|w|^{2(1-\beta)}} f(\rho^{1/\beta} w)\log\left|\frac{\rho^{-1/\beta}z-w}{1-\rho^{-1/\beta}\bar z }\right| |dw|^2.
\end{equation}
The integral of $u$ along any circle $\{|w|=s\}$ is zero, so the same is true for $f$. This implies that 
\begin{align*}
u(z)&= \rho^2\int_{|w|<1}\frac{\beta^2}{|w|^{2(1-\beta)}} f(\rho^{1/\beta} w)\log\left|\frac{\rho^{-1/\beta}z/w-1}{1-\rho^{-1/\beta}\bar z\bar w }\right| |dw|^2\\
&=\beta^2|z|^{2\beta}\int_{|t|>\rho^{-1/\beta}|z|} f(\rho^{1/\beta} w) \log\left|\frac{t-1}{1-\rho^{-2/\beta} |z|^2/t }\right| \frac{ |dt|^2}{|t|^{2+2\beta}} 
\end{align*}
after performing the change of variables $t:=\rho^{-1/\beta}z/w$. Since $f(0)=0$, we have $|f| \le C_\alpha \|f\|_{\cC^\alpha} \cdot r^\alpha$ and therefore 
\[\left|\frac{u(z)}{r^{2+\alpha}}\right| \le  C_\alpha \|f\|_{\cC^\alpha}\cdot  \underbrace{\beta^2\int_{|t|>\rho^{-1/\beta}|z|} \log\left|\frac{t-1}{1-\rho^{-2/\beta} |z|^2/t }\right| \frac{ |dt|^2}{|t|^{2+2\beta}}}_{=:I(z)}.\]
We are left to bounding the integral $I(z)$ uniformly for all $\beta$ and all $z\in B_\beta'$. 
When $|t|$ is very small, say $|t| \le \ep$, then $\rho^{-1/\beta} |z| < \ep$ and the log term is a $O(t)+O(\rho^{-1/\beta}|z|)=O(t)$ hence this portion of the integral is dominated by $\beta^2\int_{s=0}^{\ep}\frac{ds}{s^{2\beta}}=O(1)$. 
For the rest of the integral, we first observe that for $z\in B_\beta'$, we have $|\rho^{-2/\beta} |z|^2/t| \le \rho^{1/\beta}$ hence \[\beta^2\int_{|t|>\ep}- \log |1-\rho^{-2/\beta} |z|^2/t|\frac{ |dt|^2}{|t|^{2+2\beta}} \le \beta^2\int_{s=\ep}^{+\infty}\frac{ds}{s^{1+2\beta}} = O(1).\] 
We are left to estimating $\beta^2\int_{|t|>\ep}\log|t-1| \frac{ |dt|^2}{|t|^{2+2\beta}}$. The region $\ep\le |t|\le 2$ is trivially dealt with, while the remaining region is estimated by \[\int_{s=2}^{+\infty}\frac{\log s}{s^{1+2\beta}} ds=\frac{1}{\gamma \beta}\Big[s^{-\gamma \beta}(\log s-\frac{1}{\gamma \beta})\Big]_2^{+\infty}=O(\beta^{-2})\]
where $\gamma=(2+\alpha)$. The estimate \eqref{improved} is now proved.  \\

$\bullet$ {\it Step 3. Schauder estimates for the cylinder.}
Set $t:=\log r$ so that $\gk=r^2 g_c$ where $g_c:=dt^2+\beta^2d\theta^2$ to be the cylindrical metric on $\mathbb R \times S^1$, where the circle has length $2\pi \beta$. It is complete with bounded curvature hence it satisfies uniform Schauder estimates independent of $\beta$ and the chosen ball of a given radius (small balls may not be simply connected but one can pass to the universal cover).   

Let us pick an arbitrary point $x_0\in B_\beta'$ and set $r_0:=r(x_0)$. We define the regions $D:=\{|t-t_0|<2\}$ and $D':=\{|t-t_0|<1\}$; these depend on the base point $x_0$. On $D$, the function $\frac{r}{r_0}=e^{t-t_0}$ has bounded $g_c$-derivatives at every order (and the same holds for its inverse), and these bounds are independent of $x_0$. 
On $D$, we have $\|\nabla^j_{g_c} u\|_{\cC^0} \sim r_0^j \|\nabla^j_{\gk} u\|_{\cC^0}$ and $\|v\|_{\cC^{\alpha}_{g_c}(D)} \sim r_0^{\alpha} \|v\|_{\cC^{\alpha}_{\gk}(D)}$ by the definition of the Hölder norm for $\gk$, cf \eqref{alpha}. By the same token, $\|r^2f\|_{\cC^{\alpha}_{g_c}(D)} \sim r_0^{2+\alpha} \|f\|_{\cC^{\alpha}_{\gk}(D)}$.

This implies that
\begin{align}
\label{C2a}
 \|u\|_{\cC^{2,\alpha}_{\gk}(D')} & \lesssim r_0^{-2-\alpha} \|u\|_{\cC^{2,\alpha}_{g_c}(D')}\\
 & \le C(\|u/r_0^{2+\alpha}\|_{\cC^0(D)}+ \|f\|_{\cC^{\alpha}_{\gk}(D)}) \nonumber
\end{align}
where the last inequality follows from Schauder estimates for the cylindrical metric, since $\Delta_{g_c}u=r^2f$. Putting \eqref{improved} and \eqref{C2a} together, we conclude that $ \|u\|_{\cC^{2,\alpha}_{\gk}(D')} \le C \|f\|_{\cC^{\alpha}_{\gk}(D)}$ for some constant $C$ independent of $\beta$ and $x_0$. By varying the point $x_0$ across $B_\beta'$ and using the first reduction step, we obtain the proposition. 
\end{proof}
  
  \begin{proof}[Proof of Corollary~\ref{cor strong schauder}]
  We are left to showing that the $\mathcal C^\alpha$ norms of $\d_r^2u$ and $\frac{1}{\beta r^2}\d^2_\theta u$ are controlled by $\|u\|_{\cC^0}+\|f\|_{\cC^\alpha}$. We will treat the term $v:=\d_r^2u$, the other one being entirely similar. Let $x,y\in B_\beta'$ which we write $(z_1, z_1')$ and $(z_2, z_2')$ where $z_i'\in \mathbb C^{n-1}$ for $i=1,2$. We set $t=(z_2, z_1')$ and decompose the difference $v(x)-v(y)=(v(x)-v(t))+(v(t)-v(y))$. On the slice $S_{z_1'}:=\mathbb C^*\times \{z_1'\}$, the function $u$ satisfies $\Delta_{\hat g_\beta} u = f-\Delta_{\mathbb C^{n-1}}u$ and the $\cC^\alpha$ norm of the RHS is controlled by Theorem~\ref{Schauder conique}. By Proposition~\ref{schauder dim 1}, we get that 
  \[|v(x)-v(t)|\le C d_{\hat g_\beta}(z_1, z_2)^\alpha\le Cd_{\gk}(x,y)^\alpha\] where $C$ is some uniform multiple of $\|u\|_{\cC^0}+\|f\|_{\cC^\alpha}$. 
  
  We now have to take care of $v(t)-v(y)$. The geodesic from $t$ to $y$ lies in the "euclidean slice" $\tilde S_{z_2}=\{z_2\}\times \mathbb C^{n-1}$ so we only have to estimate $D''v=\d_r^2D''u$ along $\tilde S_{z_2}$. But this is entirely similar to what has been done in \textsection~\ref{normal tangential} by using almost harmonic approximations and relying on \eqref{gradient 5}-\eqref{gradient 6} depending on whether $d_{\gk}(t,y)=d_{\rm eucl}(z_1',z_2')$ dominates or not $r(t)=r(y)=|z_2|^\beta$. 
  \end{proof}
\section{Schauder estimate for collapsed metrics}
\label{sec:semi-local-schauder}

In this section we establish the uniform Schauder estimates which are at the heart of our argument. The important point here is the existence of different scales at which we can look at the geometry, since the metrics $g_{\beta,L}$ (hence $g_\beta$) collapse much quicker in the circle direction of the bundle $L$ over $D$ (scale $\beta$) than in the directions of $D$ (scale $\sqrt \beta$). Since we have here two different speeds for the collapsing, we call this geometry `2-collapsed'. We will also need to consider an intermediate '1-collapsed' geometry, more precisely:
  \begin{itemize}
  \item \textbf{1-collapsed geometry}: this is a scale at which the circle collapses but the divisor does not collapse; this roughly amounts to consider the metrics $\frac1\beta g_\beta$, which turn out to have controled curvature so that we can obtain Schauder estimates from standard arguments far from the divisor $D$, and near the divisor $D$ with the conical singularities from the estimates developed in section \ref{sec:unif-scha-estim}; this is done in section \ref{sec:local-scha-estim};
  \item \textbf{2-collapsed geometry}: this is the scale of $g_{\beta,L}$, and it turns out that in our problem we need estimates at this scale, on balls of fixed radius, say $\varrho$, for $g_{\beta,L}$; this corresponds to balls of larger and larger radius $\frac \varrho{\sqrt \beta}$  in the 1-collapsed geometry of $\frac1\beta g_{\beta,L}$; we obtain these estimates in section \ref{sec:semi-local-schauder-1} from a global estimate on some limit of $\frac1\beta g_{\beta,L}$.
  \end{itemize}

\subsection*{Notation}
Since we have a lot of constants appearing in our estimates, from now we simplify the notation by introducing the relations \[ A \lessapprox B \quad\text{ resp. } \quad A\approx B\]
defined by the fact that there is some constant $\mathfrak{C}_n$ depending only on the dimension $n$ (and certainly not on $\beta$) such that
\[ A \leq \mathfrak{C}_n B \quad\text{ resp. }\quad \mathfrak{C}_n^{-1}B \leq A \leq \mathfrak{C}_nB. \]

\subsection{Functional spaces}
\label{sec:functional-spaces}

We now define the functional spaces in which we will solve the equation. These are weighted Hölder spaces.

With the notation of section~\ref{sec:geometry}, we define a weight on $X$ by
\begin{equation}
  \label{eq:30}
  w_\beta = \chi(-\tilde u) - (1-\chi(-\tilde u)) \tilde u
\end{equation}
so that $w_\beta=1$ for $u\leq -2$ and $w_\beta=\tilde u$ for $u\geq-\frac12$.

Fix a real number $\delta$. For any section $f$ of a tensor bundle, we define the weighted norm (which depends on $\beta$):
\begin{equation}
  \label{eq:32}
  \|f\|_{C^{\ell,\alpha}_\delta} = \sup \sum_{0\leq j\leq \ell} w_\beta^{\delta+\frac j2(1+\frac1n)}|\nabla^jf|_{g_\beta} + [w_\beta^{\delta+\frac \ell2(1+\frac1n)}\nabla^\ell f]_\alpha
\end{equation}
where the semi-norm $[f]_\alpha$ is also weighted:
\begin{equation}
  \label{eq:31}
  [f]_\alpha = \sup_{d_{g_\beta}(x,y) < \rho(\beta w_\beta^{\frac1n}(x))^{\frac12}} \big(\beta w_\beta^{\frac1n}\big)^{\frac \alpha2} \frac{|f(x)-f(y)|}{d_{g_\beta}(x,y)^\alpha}.
\end{equation}
One can be surprised by these definitions, since (\ref{eq:32}) and (\ref{eq:31}) do not correspond to the same weight: roughly speaking the norm $\|f\|_{C^0}+[f]_\alpha$ is a $C^\alpha$ norm with respect to the metric $\beta^{-1}w^{-\frac1n}g_\beta$, a metric which looks like:
\begin{itemize}
\item $\beta^{-1}g_\beta$ on $u\leq -1$: this has bounded curvature by Lemma~\ref{lemma curvature};
\item $\beta^{-1-\frac1n}g_\beta$, that is $g_{TY}$ on the compact part of $X\setminus U_L$;
\end{itemize}
in both cases the point is that the geometry is controlled so that there are uniform Schauder estimates.

On the other hand, the $C^\ell$ norm defined by (\ref{eq:32}) has a different weight, motivated by the fact that $|\nabla^jw_\beta^\bullet|_{g_\beta} = O(w_\beta^{\bullet-\frac j2(1+\frac1n)})$, so it is well-adapted to functions depending on $u$ only: we will see that these are the functions on which we have the worst estimates, because the other directions collapse.

\begin{rem}\label{rem:product}
Despite the presence of the weight, one still has the usual estimate for products in Hölder spaces: $\|fg\|_{C^\alpha} \leq \|f\|_{C^\alpha} \|g\|_{C^\alpha}$. This is because for the standard Hölder norms one has actually $[fg]_\alpha \leq \|f\|_{C^0} [g]_\alpha + [f]_\alpha \|g\|_{C^0}$, so adding our weight $\beta w_\beta^{\frac1n}$ does not change the estimate.
\end{rem}

\subsection{1-collapsed Schauder estimate}
\label{sec:local-scha-estim}

From Lemma \ref{lemma curvature} we have the following bound for the curvature of $g_{\beta,L}$:
\begin{equation}
  \label{eq:34}
  | K(g_{\beta,L}) | \lessapprox \big( \frac1{(-\varphi_1'(u))^{n+1}}+\frac1{-\beta\varphi_1'(u)} \big).
\end{equation}
Since $\varphi_1'(u) \sim c_n (-u)^{\frac1n}$ when $u\rightarrow0$, then for $u<-\frac \beta A$ for some $A>0$ which will be fixed below, we have for $B=\mathrm{cst.} (1+A)$,
\begin{equation}
  \label{eq:35}
  | K(g_{\beta,L}) | \leq \frac B{-\beta\varphi_1'(u)}.
\end{equation}
Since $u=\beta t$, the region $u<-\frac \beta A$ corresponds to the region $t<-\frac1A$ in the Tian-Yau space, that is the exterior of a compact region. We choose that compact region large enough, that is $A>0$ small enough, so that the asymptotics of the Tian-Yau metric written in section \ref{sec:tian-yau-metric} are valid. From the estimate (\ref{eq:23}) on the difference $g_\beta-g_{\beta,L}$ we see that (\ref{eq:35}) remains true for $g_\beta$. For simplicity we write the sequel for $g_{\beta,L}$ but these bounds imply that our estimates will remain true for the small perturbation $g_\beta$.

Since $\varphi_1''(u) \approx e^u$ by \eqref{eq:11}, the distance $r_D$ from the divisor is of order $r_D \approx e^{\frac u2}$. Therefore a region $r_D>\varepsilon$ corresponds to $u>2\log \varepsilon+\mathrm{cst}$. 

We will also use the function $r_0(u)$ which is the distance to the point $u=0$, so that when $u\rightarrow 0$ one has \[ r_0(u) \approx |u|^{\frac12+\frac1{2n}}. \]

Note $\rho$ the injectivity radius of the metric $g_D$, and fix a finite number of balls of radius $\rho$ covering $D$. Near some $u_0<-\frac \beta A$ we can consider the rescaled metric $h=\frac{g_{\beta,L}}{-\beta\varphi_1'(u_0)}$. This is the metric where only the $S^1$ fibres are collapsed, at speed $\sqrt \beta$ (`1-collapse'). For $u\approx u_0$ the curvature of $h$ is uniformly bounded and $h \simeq \frac{\varphi_1''(u)}{-4\beta\varphi'_1(u_0)}(2du^2+\beta^2\eta^2)+g_D$. Because of the $S^1$-bundle over $D$, small balls of radius $\rho$ for $h$ are not simply connected but we can use Schauder estimates in local universal coverings. We define the domain 
\[ D_{u_0}(\tau) = \big\{ |r_0(u)-r_0(u_0)| \leq \tau \sqrt{\beta w_\beta(u_0)^{\frac 1n}} \big\} \]
and we assume the extra condition $r_D(u_0)>\varepsilon\sqrt \beta$, that is $u_0>\log \beta+2\log \varepsilon+\mathrm{cst}$.  Said otherwise, we only look at points at a fixed, positive distance to the divisor $D$ with respect to $h$. We can then check that $D_{u_0}(\rho)\subset B_{h}(u_0, C\rho)$ and that on $D_{u_0}(\rho)$, we have $\frac{u_0}{u} \lessapprox 1$ (this can be done easily by treating each case $u_0\approx -\infty, -1, 0$ separately and using the assumption $r_D(u_0)>\varepsilon\sqrt \beta$.)
These considerations lead to the following Schauder estimate for the metric $h$ outside the divisor:
\begin{equation}
  \label{eq:36}
  \| f \|_{C^{2,\alpha}_h(D_{u_0}(\frac12 \rho))}
  \lessapprox \left( \|f\|_{C^0(D_{u_0}(\rho))} + \| \Delta_h f\|_{C^\alpha_h(D_{u_0}(\rho))}
  \right)
\end{equation}
which we rewrite in terms of $g_\beta$ using the weighted norm (\ref{eq:31}):
\begin{equation}
  \label{eq:37}
    \|\nabla^2f\|_{C^\alpha_0(D_{u_0}(\frac12 \rho))} 
  \lessapprox \frac1{\beta w_\beta(u_0)^{\frac1n}}\|f\|_{C^0(D_{u_0}(\rho))} + \|\Delta_{g_\beta} f\|_{C^\alpha_0(D_{u_0}(\rho))}.
\end{equation}
Note that this estimate extends everywhere:
\begin{itemize}
\item Near the divisor we have the same by taking balls centered on
  the divisor and applying the uniform Schauder estimate established
  in section \ref{sec:unif-scha-estim} to the metric $h=\frac {g_\beta} \beta$.  More precisely, if $u_0=-\infty$ and $\rho>0$ is small enough one has $D_{-\infty}(\rho)=\{r < \rho \sqrt \beta \}$ which is topologically a disk in $\mathbb C$ times $D$.
 The change of variable $R= r/\sqrt \beta$ in (\ref{eq:13}) gives
  \[ h=\frac{g_\beta}\beta = \frac2{n+1}\big( dR^2+\beta^2R^2\eta^2\big) + g_D + O(\beta R^2). \]
  Maybe up to scaling again by a large constant, we see that we have locally a uniformly small (in the sense of Remark \ref{rem:strong-scha-estim}) perturbation of the product of the cone metric $dR^2+\beta^2R^2d\theta^2$ with the flat metric, and therefore we can apply Corollary~\ref{cor strong schauder} to get \[\| f \|_{C^{2,\alpha}_h(D_{-\infty}(\frac12 \rho))}
  \lessapprox \left( \|f\|_{C^0(D_{-\infty}(\rho))} + \| \Delta_h f\|_{C^\alpha_h(D_{-\infty}(\rho))}\right)\] and the corresponding inequality
    (\ref{eq:37}) in terms of $g_\beta$ (with the weighted norm \-- in this region we have $w_\beta\equiv1$) is also valid for $D_{-\infty}(\rho)$. Therefore, it holds  any $u_0<-\frac \beta A$. 
\item On the Tian-Yau part $X\setminus U_L$, this is the standard Schauder estimate, since the scaling factor with the Tian-Yau metric is $\beta w_\beta(u_0)^{\frac1n}=\beta^{1+\frac1n}$. 
\end{itemize}

\subsection{2-collapsed Schauder estimate}
\label{sec:semi-local-schauder-1}

Our aim now is to give an estimate more suitable for the scale of $g_\beta$, where the the divisor $D$ is also collapsed at speed $\sqrt \beta$, and the circle at speed $\beta$ (`2-collapse'). More precisely, we want to replace by a better coefficient the factor $\frac1{\beta w_\beta(u_0)^{\frac1n}}$ in (\ref{eq:37}). For $u_0$ bounded away from zero, this is just the scaling factor $\frac 1\beta$. We do this in two steps: decomposing in Fourier series along the circle, a direct application of the maximum principle gives the required estimates on nonzero modes. For zero modes, the argument is more complicated: instead of estimates on balls of radius $\frac 1{\sqrt \beta}$ as above, we would like estimates on balls of fixed radius, say $\rho$ if $u$ is far from $0$ or $-\infty$. In the rescaled metric $\frac 1\beta g_{\beta,L}$ this corresponds to a cylinder of length approximately $\frac \rho{\sqrt \beta}$, converging to an infinite cylinder. The estimates on this limit will provide the estimates we need.

For small values of $\varepsilon$ and $u_0$ away from the divisor, we define the region
\begin{equation}
  \label{eq:41}
  E_{u_0}(\varepsilon) = \{ | r_0(u)-r_0(u_0) | \leq \varepsilon r_0(u_0) \}.
\end{equation}
If $u_0=-\infty$ that is we consider the divisor $D$, we use
\begin{equation}
  \label{eq:43}
  E_{-\infty}(\varepsilon) = \{ | r_D(u) | \leq \varepsilon \}.
\end{equation}
Notice that for $\varepsilon>0$ and $u_0$ small, $\varphi_1'(u)$ and $\varphi_1''(u)$ do not vary much in $E_{u_0}(\varepsilon)$, that is remain comparable to their value at $u_0$.

The regions $E_{u_0}$ correspond to the scale of the geometry of $g_\beta$, and are therefore very large for the geometry of the previously used $h=\frac{g_{\beta,L}}{-\beta\varphi_1'(u_0)}$. We then obtain the better estimate:
\begin{prop}\label{prop:semi-local-schauder}
  \begin{equation}
       \|\nabla^2f\|_{C^\alpha_0(E_{u_0}(\frac12 \varepsilon))} 
     \lessapprox \|\Delta_{g_\beta} f\|_{C^\alpha_0(E_{u_0}(\varepsilon))} + w_\beta(u_0)^{-1-\frac1n} \|f\|_{C^0(E_{u_0}(\varepsilon))} .\label{eq:44}
   \end{equation}
 \end{prop}

 Remark that on the Tian-Yau part $X\setminus U_L$ the function $w_\beta$ takes the value $\beta$ so this is the same estimate as in (\ref{eq:37}). But the important point is that on $\{u\le -1\}$ the factor $w_\beta$ does not depend on $\beta$, contrarily to the initial estimate (\ref{eq:37}). We deduce the following corollary:

 \begin{cor}\label{coro:semilocal_estimate}
   One has the following uniform estimate, for all functions $f$ on $X$:
   \[ \|f\|_{C^{2,\alpha}_\delta} \lessapprox \|f\|_{C^0_\delta} + \|\Delta_{g_\beta}f\|_{C^\alpha_{\delta+1+\frac1n}}. \]
  \qed
 \end{cor}
 
 \begin{rem}
 \label{gen operator}
 The estimate $ \|f\|_{C^{2,\alpha}_\delta} \lessapprox \|f\|_{C^0_\delta} + \|L_{\beta}f\|_{C^\alpha_{\delta+1+\frac1n}}$ follows for any operator of the shape say $L_\beta =\Delta_{g_\beta}+c$ where $c$ is a constant, as one sees immediately using the interpolation estimate $\|f\|_{C^\alpha}\lessapprox \ep^{-\alpha} \|f\|_{C^0}+\ep^\alpha \|f\|_{C^1}$ for any fixed $0<\ep\ll 1$.   
 \end{rem}

 The rest of this section is devoted to the proof of the estimate (\ref{eq:44}). Again from the bounds on $g_\beta-g_{\beta,L}$ it is sufficient to prove the estimate on $g_{\beta,L}$.

 \subsubsection*{First step} We first decompose along each circle into Fourier series $f=\sum_{\Bbb{Z}} f_\ell$ and control nonconstant modes. Here, more precisely, we see $L$ as a $S^1$ bundle over $\Bbb{R}_-\times D$, and $f_\ell$ is induced from a section $F_\ell$ of $p^*L^{-\ell}$ over $\Bbb{R}_-\times D$ by the formula $f_\ell(x)=\langle F_\ell(p(x)),x^{\otimes \ell}\rangle$. On $p^*L^{-\ell}$ over $D$ we have the rough Laplacian $-\Delta_D := \nabla^*\nabla$ constructed from the given connection on $L$ and the metric $g_D$.

 The Laplacian of the metric $g_{\beta,L}$ preserves the Fourier decomposition and acts on $F_\ell$ by
\begin{align}
  \label{eq:38}
  \Delta_{g_{\beta,L}} F_\ell &= \frac 2{\varphi_1''(u)} \big( \partial_u^2F_\ell - \frac{\ell^2}{4 \beta^2}F_\ell \big)
  + \frac{2(n-1)}{\varphi_1'(u)} \partial_u F_\ell +\frac1{-\beta\varphi_1'(u)} \Delta_D F_\ell  \\
  &= \Delta_{\Bbb{R}_-\times D}F_\ell - \frac{\ell^2}{2 \varphi_1''(u)\beta^2}F_\ell. \nonumber \label{eq:59}
\end{align}

\begin{lem}
  Fix $\varepsilon>0$ small enough and $\ell\neq0$. We have the estimate
  \begin{equation}
    \label{eq:39}
    \sup_{E_{u_0}(\frac12 \varepsilon)} |f_\ell| \lessapprox  \frac \beta{\ell^2} \sup_{E_{u_0}(\varepsilon)} \big( w_\beta(u_0)^{\frac1n} |\Delta f_\ell| + w_\beta(u_0)^{-1} |f_\ell| \big).
  \end{equation}
  It follows that
  \begin{equation}
    \label{eq:40}
    \sup_{E_{u_0}(\frac12 \varepsilon)} |f-f_0| \lessapprox \beta \sup_{E_{u_0}(\varepsilon)} \big( w_\beta(u_0)^{\frac1n} |\Delta(f-f_0)| + w_\beta(u_0)^{-1} |f-f_0| \big).
  \end{equation}
\end{lem}
\begin{proof}

Let us first explain how to derive \eqref{eq:40} from \eqref{eq:39}. Set $g:=f-f_0$. The function $F_\ell$ can be recovered as an integral of $ge^{-i\ell \cdot}$ along each circle. Hence for $\ell \neq 0$, we have $\sup |f_\ell| \le \sup |g|$ and similarly 
 $\sup |\Delta f_\ell| \le \sup |\Delta g|$.  Now we can just sum all the estimates \eqref{eq:39} for $\ell \in \mathbb Z^*$ and  get \eqref{eq:40}.
  
  We now get to proving \eqref{eq:39}. We use the inequality
 \begin{equation}
 \frac12 \Delta_{\Bbb{R}_-\times D}|F_\ell|^2 \geq \langle \Delta_{\Bbb{R}_-\times D}F_\ell,F_\ell\rangle
= \langle\Delta  F_\ell, F_\ell\rangle + \frac{\ell^2}{2 \varphi_1''(u)\beta^2} |F_\ell|^2\label{eq:60}
\end{equation}
 twice.

First, if $f_\ell$ vanishes at $\partial E_{u_0}(\varepsilon)$, then from the maximum principle and the fact that $\varphi_1''(u) \lessapprox \varphi_1''(u_0)$ on $E_{u_0}(\varepsilon)$ we have
\[ \sup_{E_{u_0}(\varepsilon)} |F_\ell| \lessapprox \frac{\beta^2\varphi_1''(u_0)}{\ell^2} \sup_{E_{u_0}(\varepsilon)} |\Delta F_\ell|
\]
and the result follows since $\frac \beta{|u_0|} \leq A$ and $\varphi_1''(u)\approx u^{-1+\frac1n}$ when $u\rightarrow0$. 

Therefore it is sufficient to consider the case where $\Delta f_\ell=0$.  Indeed, let $h_\ell$ be the harmonic function on  $E_{u_0}(\varepsilon)$ with same boundary values as $f_\ell$ and let $g_\ell:=f_\ell-h_\ell$. By what was said above, we have \[ \sup_{E_{u_0}(\varepsilon)} |g_\ell| \lessapprox  \frac \beta{\ell^2} \sup_{E_{u_0}(\varepsilon)} w_\beta(u_0)^{\frac1n} |\Delta f_\ell|\]
hence 
\[ \sup_{E_{u_0}(\frac12 \varepsilon)} |f_\ell| \lessapprox  \frac \beta{\ell^2} \sup_{E_{u_0}(\varepsilon)} w_\beta(u_0)^{\frac1n} |\Delta f_\ell| + \sup_{E_{u_0}(\frac12 \varepsilon)} |h_\ell|. 
\]
Recall that $\sup_{E_{u_0}(\varepsilon)}  |h_\ell|=\sup_{E_{u_0}(\varepsilon)}  |f_\ell|$. So, if we assume that $\sup_{E_{u_0}(\frac12 \varepsilon)} |h_\ell| \lessapprox \frac \beta{\ell^2}w_\beta(u_0)^{-1} \sup_{E_{u_0}(\varepsilon)}  |h_\ell|$, we are done.

Let us consider the comparison function $g(u)=a\cosh(b(u-u_0))$. Then
\begin{align*}
 \frac12 \Delta g = \frac{\partial_u^2g}{\varphi_1''(u)} + (n-1)\frac{\partial_u g}{\varphi_1'(u)}
  &\leq \Big( \frac{b^2}{\varphi_1''(u)} - \frac{(n-1)b}{\varphi_1'(u)} \Big) g \\
    &\leq \frac{\ell^2}{2\beta^2\varphi_1''(u)} g
\end{align*}
if we take $b=\gamma \frac \ell \beta$ for some small constant $\gamma>0$,   thanks to the condition $u\le -\frac \beta A$ for $A$ small enough. Combining with (\ref{eq:60}) we obtain
\begin{equation}
 \Delta(|F_\ell|^2-g) \geq \frac{\ell^2}{\varphi_1''(u)\beta^2}(|F_\ell|^2-g).\label{eq:62}
\end{equation}
Choose $a$ large enough so that $|F_\ell|^2 \leq g$ at $\partial E_{u_0}(\varepsilon)$, 
then it follows from (\ref{eq:62}) that $|F_\ell|^2 \leq g$ on $E_{u_0}(\varepsilon)$, which in particular on $E_{u_0}(\frac \varepsilon2)$ we obtain
\[ \sup_{E_{u_0}(\frac \varepsilon2)} |F_\ell|^2
  \lessapprox \frac{\cosh(\frac12 \varepsilon\gamma \frac \ell \beta u_0)}{\cosh(\varepsilon\gamma \frac \ell \beta u_0)}
  \sup_{E_{u_0}(\varepsilon)} |F_\ell|^2. \]
Since $x:=\frac{|u_0|}\beta \geq \frac1A$, we certainly have
\[ \frac{\cosh(\frac12 \varepsilon\gamma \frac \ell \beta u_0)}{\cosh(\varepsilon\gamma \frac \ell \beta u_0)}
  \lessapprox \frac 1{(x\ell)^4} \leq \frac {A^2}{(x\ell^2)^2} \]
which gives (\ref{eq:39}).

The case where $u_0=-\infty$ (that is centered on the divisor) is similar,  so we only highlight the modifications to perform. To obtain the inequality  
\begin{equation}
\label{lap control}
\sup_{E_{-\infty}(\varepsilon)} |F_\ell| \lessapprox \frac{\beta^2}{\ell^2} \sup_{E_{-\infty}(\varepsilon)} |\Delta F_\ell|
\end{equation}
we can modify $E_{-\infty}(\ep)$ and assume that its boundary is $\{u=2\log \ep\}$. Let us set $u_\ep:=u-2\log \ep$ so that $\Delta u_\ep=\frac{2(n-1)}{\vp_1'(u)}=O(1)$ and fix $\delta>0$.  Then, we apply the maximum principle to $|F_\ell|^2+\delta u_\ep$ to obtain, for any fixed $x$:
\[|F_\ell(x)|^2\le \frac{\beta^2}{\ell^2} \sup_{E_{-\infty}(\varepsilon)} |\Delta F_\ell| \cdot \sup_{E_{-\infty}(\varepsilon)} |F_\ell|-\delta u_\ep(x)+O(\delta)\] and we get \eqref{lap control} by taking $\delta\to 0$ and passing to the supremum over $x\in E_{-\infty}(\varepsilon)$. The next step is very similar to the case $u_0\neq -\infty$ but one chooses instead $g(u)=a e^{\frac{\ell u}{\beta}}$, satisfying $\Delta g\le\frac{2\ell^2}{\beta^2\vp_1''(u)}g$. Then we apply the maximum principle using the same barrier function as before, to obtain  $\sup_{E_{-\infty}(\frac \varepsilon2)} |F_\ell|
  \le  e^{-\frac{\ell}{\beta} \log 2} \sup_{E_{-\infty}(\varepsilon)} |F_\ell|$. The conclusion follows from the inequality $e^{-\frac{\ell}{\beta} \log 2} \lessapprox \frac{\beta^2}{\ell^2}.$ 
  \end{proof}
Combining (\ref{eq:40}) with the local Schauder estimate (\ref{eq:37}) for $g=f-f_0$, we obtain
\begin{equation}\label{eq:42}
     \|\nabla^2g\|_{C^\alpha_0(E_{u_0}(\frac12 \varepsilon))} 
     \lessapprox  \|\Delta_{g_{\beta,L}} g\|_{C^\alpha_0(E_{u_0}(\varepsilon))} + w_\beta(u_0)^{-1-\frac1n}\|g\|_{C^0(E_{u_0}(\varepsilon))} ,
\end{equation}
which is the estimate stated in Proposition~\ref{prop:semi-local-schauder}.

\subsubsection*{Second step} We can now restrict the proof of Proposition~\ref{prop:semi-local-schauder} to the case when $f$ is circle invariant. In that case (\ref{eq:38}) reduces to
\begin{align*}
  \Delta_{g_{\beta,L}} f &= 2\frac{\partial_u^2f_\ell}{\varphi_1''(u)} + 2(n-1)\frac{\partial_u f}{\varphi_1'(u)} +\frac1{-\beta\varphi_1'(u)} \Delta_D f  \\
  &= \frac1{-\beta\varphi_1'(u)} \big( \partial_v^2f + \Delta_Df \big),
\end{align*}
where $dv^2=\frac{\varphi_1''(u)}{-2\beta\varphi_1'(u)} du^2$ and we choose $v(u_0)=0$. We now restrict to the case where for every $u$ one has
\begin{equation}
  \label{eq:46}
  \int_{\{u\}\times D} f =  0.
\end{equation}
We also suppose that we are not close to the divisor $D$, that is $r_D(u_0)>\varepsilon$.
On each slice $\{u\}\times D$ the function $f$ is therefore orthogonal to the kernel of $\Delta_D$, which corresponds to erasing the critical weight $0$ of the cylindrical Laplacien $\partial_v^2 + \Delta_D$. It follows that this Laplacian is an isomorphism $C^{2,\alpha}(\Bbb{R}\times D)\rightarrow C^\alpha(\Bbb{R}\times D)$, see for example \cite[formula (2.3)]{LocMcO85}. In particular we have, still under condition (\ref{eq:46}),
\begin{equation}
  \label{eq:47}
  \sup_{\Bbb{R}\times D} |f| \lessapprox  \sup_{\Bbb{R}\times D}  | (\partial_v^2 + \Delta_D)f |.
\end{equation}
A quick elementary derivation of (\ref{eq:47}) is as follows: suppose $(\partial_v^2+\Delta_D)f=g$, note $E(a)=\int_{[-a,a]\times D} |df|^2$ and $F(a)=\int_{[-a,a]\times D} |g|^2$. The hypothesis on $f$ implies that for each $v$ we have $\int_{\{v\}\times D} |d_Df|^2 \geq \lambda_1^2 \int_{\{v\}\times D} |f|^2$, where $\lambda_1^2$ is the first nonzero eigenvalue of $\Delta_D$. By elliptic regularity and translation invariance, it is enough to prove that $E(1)\lessapprox \sup |g|^2$. By integration by parts we have
  \begin{equation}
 E(a)=-\int_{[-a,a]\times D} fg + \big(\int_{v=a}-\int_{v=-a}\big) f \partial_vf . \label{eq:61}
\end{equation}
But $|\int_{v=a} f \partial_vf | \leq \frac12 \int_{v=a} \frac1{\lambda_1} |\partial_vf|^2 + \lambda_1 |f|^2 \leq \frac1{2\lambda_1} \int_{v=a} |df|^2$. Since $E'(a)=(\int_{v=a}+\int_{v=-a})|df|^2$, it follows from (\ref{eq:61}) that $\lambda_1 E(a) - E'(a) \leq \frac{F(a)}{\lambda_1}$, and after integration $E(1) \leq \int_1^{+\infty} \frac{e^{-\lambda_1(a-1)}}{\lambda_1} F(a) da \leq (\sup |g|^2) \frac2{\lambda_1}\int_1^{+\infty} a e^{-\lambda_1 (a-1)} da$.

The region $E_{u_0}(\varepsilon)$ gives only a bounded region in $\Bbb{R}\times D$, of diameter $2d$ with respect to the variable $v$. Using a cut-off function for $\frac d2 \leq |v| \leq d$, we deduce from (\ref{eq:47}) combined with the interpolation inequality $\|\partial_v f\|_{\infty}^2 \lessapprox \|\partial^2_v f\|_{\infty} \cdot \| f\|_{\infty}$ the following estimate
\begin{equation}
  \label{eq:48}
  \sup_{[-\frac d2,\frac d2]\times D} |f| \lessapprox \sup_{[-d,d]\times D} \Big[| (\partial_v^2 + \Delta_D)f | + \frac1{d^2}|f|\Big].
\end{equation}
Now we come back to the variable $u$: we have $d^2 \approx \frac{\varphi_1''(u_0)}{-\beta\varphi_1'(u_0)}w_\beta(u_0)^2$, so we obtain (remember $-\varphi_1'(u_0)\approx w_\beta(u_0)^{\frac1n}$):
\begin{equation}
  \label{eq:49}
  \sup_{E_{u_0}(\frac \varepsilon2)} |f| \lessapprox \beta w_\beta(u_0)^{\frac1n} \sup_{E_{u_0}(\varepsilon)} |\Delta_{g_{\beta,L}}f| + \frac1{\varphi_1''(u_0)w_\beta(u_0)^2} |f| .
\end{equation}
Combining with (\ref{eq:37}) we finally get
\begin{equation}
  \label{eq:50}
  \|\nabla^2f\|_{C^\alpha_0(E_{u_0}(\frac12 \varepsilon))} \lessapprox \|\Delta_{g_\beta} f\|_{C^\alpha_0(E_{u_0}(\varepsilon))} + \frac1{\varphi_1''(u_0)w_\beta(u_0)^2} \|f\|_{C^0(E_{u_0}(\varepsilon))}.  
\end{equation}
Since $\varphi_1''(u) \approx w_\beta(u)^{-1+\frac1n}$ if $r_D(u)>\varepsilon$, the estimate (\ref{eq:44}) follows in that case.

For the case where $u_0=-\infty$, that is $E_{u_0}(\varepsilon)$ is centered on the divisor $D$, we proceed similarly, except that the limit after rescaling is a half-cylinder instead of a cylinder. The same estimate (\ref{eq:44}) follows.

\subsubsection*{Third step} There remains only the case of a function $f(u)$ of the variable $u$ alone. But it is then almost obvious that the weight $w_\beta(u)^{1+\frac1n}$, which equals $|u|^{1+\frac1n}$ for $u$ small enough, is the correct weight for the operator $\Delta_{g_{\beta,L}}$.

\section{Convergence in the positive case: proof of Theorem~\ref{thmB}}
\label{sec:conv-posit-case}

\subsection{Bound on the inverse of the linearisation}
\label{sec:bound-inverse-line}

We first give a bound on the inverse of the linearisation $L_\beta =\frac12 \Delta_{g_\beta}+1$ of the operator $P_\beta $ defined in (\ref{eq:21}). We first prove:
\begin{prop}\label{prop:bound_L-1}
  Fix $\delta\in(0,1)$. There exists a constant $c$ such that for any function $f$ on $X$ such that $\int_X f d\vol_{g_\beta}=0$ one has
  \begin{equation}
   \|f\|_{C^{2,\alpha}_\delta} \leq c \| L_\beta f \|_{C^\alpha_{\delta+1+\frac1n}}.\label{eq:45}
 \end{equation}
\end{prop}

We deduce:
\begin{cor}
\label{coro1}
  Fix $\delta\in(0,1)$. The operator $L_\beta :C^{2,\alpha}_\delta\rightarrow C^\alpha_{\delta+1+\frac1n}$ satisfies
  \begin{equation}
 \| L_\beta^{-1} \| \lessapprox \beta^{-\frac1n-\delta}.\label{eq:51}
\end{equation}
\end{cor}
\begin{proof}[Proof of Corollary~\ref{coro1}]
  The constant function $1$ satisfies $\|1\|_{C_\delta}=1$ as soon as $\delta\geq0$. Since $L_\beta (1)=1$, the estimate (\ref{eq:45}) is also satisfied on constants.

 Set $\delta':=\delta+1+\frac 1n$. Given $f\in C^\alpha_{\delta'}$ we decompose $f=\bar f+f_1$ with $\bar f$ constant and $\int_X f_1 d\vol_{g_\beta}=0$. 
We have $L_\beta^{-1}f=L_{\beta}^{-1}f_1+\bar f$ so by \eqref{eq:45}, we have $\|L_\beta^{-1}f\|_{C^{\alpha}_\delta}\lessapprox \|f\|_{C^{\alpha}_{\delta'}}+|\bar f|$ and we are reduced to showing that
 \begin{equation}
 \label{constantes}
 |\bar f| \lessapprox \beta^{-\frac1n-\delta}  \|f\|_{C^0_{\delta'}}.
 \end{equation}
 Now it follows from \eqref{eq:23} that on $\{u\le -\beta^\mu/2\}$, we have $d\vol_{g_\beta} \approx \beta^n d\vol_{g_X}$ for some fixed Riemannian metric $g_X$ on $X$ while on $\{\tilde u \ge -\beta^\mu/2\}$, we have $g_\beta=\beta^{1+\frac 1n}g_{TY}$ and, in particular, $\int_{\{\tilde u \ge -\beta^\mu/2\}}d\vol_{g_\beta} =O(\beta^{n+1-\mu})$. It follows that  $\Vol(g_\beta) \approx \beta^n$.

Now, since $|f| \leq \|f\|_{C^0_{\delta'}} w_\beta^{-\delta'}$ we have,
  \begin{align*}
    |\bar f| & \leq \frac1{\Vol(g_\beta)}\int_x |f| d\vol_{g_\beta} \\
           & \lessapprox \frac{\|f\|_{C^0_{\delta'}}}{\beta^n} \int_X w_\beta^{-\delta'} d\vol_{g_\beta}.
  \end{align*}
  One checks that since $\delta'>1$, the main contribution of the last integral is on the Tian-Yau part $X\setminus U_L$, and is of order $\beta^{n+1-\delta'}$. We therefore get $ |\bar f| \lessapprox \beta^{1-\delta'}  \|f\|_{C^0_{\delta'}}$ and \eqref{constantes} is proved. This ends the proof of the corollary.
\end{proof}

The rest of this section is devoted to the proof of Proposition \ref{prop:bound_L-1}. 

Suppose (\ref{eq:45}) is not true. By Corollary~\ref{coro:semilocal_estimate} and Remark~\ref{gen operator} there exist functions $f_\beta$ such that $\int_X f_\beta d\vol_{g_\beta}=0$ and
\[ \|f_\beta\|_{C^0_\delta} = 1, \qquad \|L_\beta f_\beta\|_{C^\alpha_{\delta+1+\frac1n}} \rightarrow 0, \]
while $\|f_\beta\|_{C^{2,\alpha}_\delta}$ is bounded. Fix $x_\beta\in X$ such that $w_\beta(x_\beta )^\delta|f_\beta(x_\beta )|=1$. We then analyze the various cases depending on the limit of $x_\beta $.

\subsubsection*{First case} The $x_\beta $'s converge to $D$ and $u(x_\beta )<\eta<0$. 

 For now, we restrict the functions $f_\beta$ on $U_L$ and see them as functions on a fixed neighborhood $U=\{\|v\|^2_h < e^{-1}\}$ of the zero section in $L$. In this proof only, we emphasize the dependence of $u$ in $\beta$ and denote it $u_\beta$. We use the diffeomorphism $\Phi_\beta:L\setminus D \to L\setminus D, v\mapsto \|v\|^{\frac 1 \beta -1} \cdot v$, so that $\Phi_\beta^*u_\beta=u: v\mapsto \log \|v\|^2_h$ and $U_\beta:=\Phi_\beta^{-1}(U)= \{\|v\|^2_h < e^{-{\beta}}\}$ is essentially independent of $\beta$. Of course, $\Phi_\beta^*g_\beta$ is still given by \eqref{eq:10}, but there is no more hidden dependence on $\beta$ in that expression. Set $F_\beta=\Phi_\beta^*f_\beta$. Since $\|f_\beta\|_{C^{2,\alpha}_\delta}$ remains bounded, we can extract a limit $F_\beta \rightarrow f$. It turns out that $f$ depends on $u$ only because the norm $\|df_\beta\|_{C^0_{\delta+1}}$ involves a factor $\beta^{-1}$ in the circle direction or $\beta^{-1/2}$ in the divisor $D$ direction. It is possible that $y_\beta:=\Phi_\beta^{-1}(x_\beta)$ satisfies $u(y_\beta)\to -\infty$ but the $C^1$ bound on $F_\beta$ ensures that one can find another point $z_\beta$ with $\eta\ge u(z_\beta)\ge -C$ and $|F_\beta(z_\beta)| \ge \frac 12$ for some $C>0$. 

Therefore we have a non-zero limit $f(u)$ for $u\in (-\infty,0)$ which satisfies
  \begin{gather}
    \sup |w|^{\delta} |f|=1, \label{eq:54} \\
    \frac{f''(u)}{\varphi_1''(u)}+(n-1)\frac{f'(u)}{\varphi_1'(u)}+f(u)=0, \label{eq:53}
  \end{gather}
 where $w(u)=1$ if $u\le -2$ and $w(u)=u$ if $u\ge -1/2$. Note that the limit Laplacian
    \begin{equation}
 \Delta_{\nu_\infty}:=\frac12\big(\frac1{\varphi_1''(u)}\partial_u^2+\frac{n-1}{\varphi_1'(u)}\partial_u\big)\label{eq:66}
\end{equation}
is the Bakry-Emery Laplacian of $(g_\infty,\nu_\infty)$, a general fact in measured Gromov-Hausdorff convergence.

  Moreover
  \[ \frac1{\beta^n} \left|\int_{\tilde u\geq - \beta} f_\beta d\vol_{g_\beta}\right| \leq \beta^{1-\delta} \|f_\beta\|_{C^0_\delta} \]
  so if $\delta<1$ we have that
  \[ \frac1{\beta^n} \int_X f_\beta d\vol_{g_\beta} \rightarrow c \int_{-\infty}^0 f(u) \varphi_1''(u)\varphi_1'(u)^{n-1}du \]
  for some constant $c$, so the limit $f(u)$ satisfies
  \begin{equation}
\int_{-\infty}^0 f(u) \varphi_1''(u)\varphi_1'(u)^{n-1}du =0.\label{eq:55}
\end{equation}

The function $\varphi_1'(u)$ is an obvious solution of (\ref{eq:53}), it corresponds to the dilation vector field in the bundle $L$. It satisfies $\varphi_1'(u) \rightarrow -1$ when $u\rightarrow -\infty$, while the other solution is $f=g\varphi_1'$ with $g=\int \frac{du}{(1-e^{-\varphi_1})}$. By \eqref{t - infty 2}, $f(u)\approx u$ near $-\infty$ (it corresponds to the Green function near $D$). But this is ruled out by (\ref{eq:54}). Therefore we see that up to a constant we must have $f(u)=\varphi_1'(u)$, which gives a contradiction with (\ref{eq:55}).

  \subsubsection*{Second case} The $x_\beta $'s still converge to $D$ but we can extract a limit only in the intermediate region between the normal bundle $L$ and the Tian-Yau metric on $X\setminus D$: this is the case where $u(x_\beta )\rightarrow0$ but $\frac1\beta u(x_\beta )\rightarrow-\infty$. It is similar to the previous one; set   $\ep_\beta:=-u(x_\beta), v:=\frac{u}{\ep_\beta}$ and consider the rescaled functions $h_\beta:=\ep_\beta^{\delta} f_\beta$. It satisfies $h_\beta \le |v|^{-\delta}$ while on the compact sets of $v\in (-\infty,\frac{\beta}{\ep_\beta}]$, the metric $g_\beta$ is asymptotically close to \[ \ep_\beta^{1+\frac 1n}\left(\tfrac1n |v|^{-1+\frac1n} (\tfrac12 dv^2 + 2 \big(\frac{\beta}{\ep_\beta}\big)^2 \eta^2) + \frac{\beta}{\ep_\beta}\cdot |v|^{\frac1n} g_D\right)\]
    by \eqref{eq:11bis}, with $\frac{\beta}{\ep_\beta}\to 0$.

    Geometrically this amounts to saying that we have the following convergence in the measured Gromov-Hausdorff sense: for some constant $C>0$,

    \[ \big(X,\varepsilon_\beta^{-1-\frac1n}g_\beta,x_\beta,\frac{d\vol_{g_\beta}}{\vol_{g_\beta}(B(x_\beta,1))}\big)
      \longrightarrow \big((-\infty,0), \frac1{2n}|v|^{-1+\frac1n}dv^2,-1,C dv\big).\]

    Our Hölder estimates imply similarly to the previous case that $h_\beta$ has a non-zero limit $h$  which is a function of $v$ only. Since $L_\beta=\frac12(\Delta_{g_\beta}+1)=\varepsilon_\beta^{-1-\frac1n}(\frac12 \Delta_{\varepsilon_\beta^{-1-\frac1n}g_\beta}+\varepsilon_\beta^{1+\frac1n})$, the constant term in $L_\beta$ disappears in the limit, so the limit function $h$ is harmonic with respect to the limit Bakry-Emery Laplacian, which is now given by equation (\ref{eq:66}) with $\varphi'(u)=(-u)^{\frac1n}$. Finally this gives the equation
\[ n (-v)^{1-\frac1n} \partial_v^2 h - (n-1) (-v)^{-\frac1n} \partial_v h = 0. \]
Hence $h$ is a linear combination of the harmonic functions $1$ and $|v|^{\frac1n}$, but it also satisfies $h\le |v|^{-\delta}$. We derive a contradiction when $v\rightarrow -\infty$.

\subsubsection*{Third case} The $x_\beta $'s converge in the Tian-Yau part: $u(x_\beta )=O(\beta)$. This is similar: after rescaling, we get a nonzero limit $f$ which is a nonzero harmonic function on the Tian-Yau space $X\setminus D$ with $|f|\leq |\tilde t|^{-\delta}$. This implies $f=0$, which is a contradiction.

\subsection{Resolution of the Kähler-Einstein equation}
\label{sec:resol-kahl-einst}

We first control the quadratic terms of the equation:
\begin{lem}
  For $\delta\geq 0$ and any function $\varphi$ we have
  \begin{equation}
    \label{eq:52}
    \| (\partial \dbar \varphi)^2 \|_{C^\alpha_\delta} \lessapprox \beta^{-\delta} \| \partial \dbar \varphi \|_{C_\delta^\alpha}^2
  \end{equation}
\end{lem}
\begin{proof}
  From Remark \ref{rem:product} we see that for $\delta=0$ we have the estimate. The weight $w_\beta^\delta$ introduces a coefficient $w_\beta^{-\delta}$ in the estimates, which is maximal when the weight $w_\beta$ is minimal, that is on the Tian-Yau part where the value of $w_\beta$ is $\beta$. The lemma follows.
\end{proof}

\begin{proof}[End of the proof of Theorem~\ref{thmB}]
Decompose $P_\beta $ into an affine part and a higher order term part:
\[ P_\beta (\varphi) = P_\beta (0) + L_\beta (\varphi) + Q_\beta (\varphi). \]
It follows from the lemma that if $\|\partial \dbar \varphi\|_{C^\alpha_{\delta+1+\frac1n}}<\varepsilon \beta^{1+\frac1n+\delta}$ for a small enough $\varepsilon>0$, then
\begin{equation}
  \label{eq:57}
  \| Q_\beta (\varphi) - Q_\beta (\psi) \|_{C^\alpha_{\delta+1+\frac1n}} \lessapprox \beta^{-1-\frac1n-\delta} \| \varphi-\psi \|_{C^{2,\alpha}_\delta} \big( \| \varphi \|_{C^{2,\alpha}_\delta} + \| \psi \|_{C^{2,\alpha}_\delta} \big).
\end{equation}

On the other hand from (\ref{eq:22}) we have $|P_\beta(0)|\lessapprox \tilde u^{1+\frac1n}$ for $u\geq -2\beta^\mu$, and $|P_\beta(0)|\lessapprox e^{(\frac12-\varepsilon)\frac u\beta}$ for $u\leq -2 \beta^\mu$. These bounds give the worst control of $\sup |w_\beta^{\delta+1+\frac1n}P_\beta(0)|$ around $u=-\beta^\mu$ and it follows that $\|P_\beta(0)\|_{C^0_{\delta+1+\frac1n}} \lessapprox \beta^{\mu(2(1+\frac1n)+\delta)}$. Adding the seminorm $[\cdot]_\alpha$ does not change the bound and we get
\begin{equation}
  \label{eq:58}
  \| P_\beta (0) \|_{C^\alpha_{\delta+1+\frac1n}} \lessapprox \beta^{\mu(2(1+\frac1n)+\delta)} .
\end{equation}

Fix $\delta$ small enough (the precise bound will be fixed below). Given the bound (\ref{eq:51}), together with (\ref{eq:57}) and (\ref{eq:58}), standard fixed point arguments (see for example \cite[Lemma 1.3]{BiqMin11}) now imply that if
\begin{equation}
  \label{eq:63}
  \beta^{\mu(2(1+\frac1n)+\delta)} \ll \beta^{1+\frac3n+3\delta}
\end{equation}
then the equation $P_\beta (\varphi)=0$ has a unique solution $\varphi(\beta)$ in a ball of radius $\varepsilon \beta^{1+\frac2n+2\delta}$ in $C^{2,\alpha}_\delta$, and this solution actually satisfies 
\begin{equation}
\label{fine control}
\|\varphi(\beta)\|_{C^{2,\alpha}_\delta} \lessapprox \beta^{\mu(2(1+\frac1n)+\delta)-\frac1n-\delta}.
\end{equation}
Observe that for $\mu=1$ the inequality (\ref{eq:63}) is satisfied if $2\delta<1-\frac1n$. Now fix $\delta<\frac12(1-\frac1n)$. We use the flexibility of fixing $\mu\in (0,1)$ as close to $1$ as we want, which geometrically corresponds to gluing more near the Tian-Yau part (then the error is smaller since the Calabi ansatz gives an exact solution up to exponentially small terms). So we fix $\mu<1$ so that (\ref{eq:63}) is satisfied, and these values of $\delta$ and $\mu$ enable us to solve the problem. Moreover for any $\epsilon>0$ we can find values of $(\mu,\delta)$ close to $(1,0)$ so that (\ref{fine control}) gives the following bound on the solution $\varphi(\beta)$:
\begin{equation}
  \label{eq:64}
  \|\varphi(\beta)\|_{C^{2,\alpha}_\delta} \lessapprox \beta^{2+\frac1n-\epsilon}.
\end{equation}

By \cite[Theorem~7.3]{Berndtsson15}, the Kähler-Einstein metric $\wh \omega_\beta:=\omega_\beta+i\partial \dbar \varphi(\beta)$ is the unique Kähler-Einstein metric with cone angle $2\pi \beta$ along $D$, i.e. satisfying $\Ric \wh \omega_\beta=\wh \omega_\beta+(1-\beta)[D]$. Given any compact set $M\Subset X\setminus D$, we have for $\beta\ll 1$: 
\[w_\beta|_M \equiv \beta, \quad \omega_\beta|_M=\beta^{1+\frac 1n} \om_{TY}, \quad \|\beta^{-1-\frac 1n} \partial \dbar \varphi(\beta)\|_{\omega_{TY}}\lessapprox \beta^{1-\epsilon-\delta}\]
where the last estimate follows by \eqref{eq:64}. In particular, $\|\beta^{-1-\frac 1n}\wh \omega_\beta-\om_{TY}\|_{\om_{TY}}=O(\beta^{1-\epsilon-\delta})$ on $M$, and we are done with the first part of Theorem~\ref{thmB}, since higher estimates are obtained as usual by bootstrapping.

At this point of the proof, we are essentially done but for exposition purposes, we collect the remaining statements in Theorem~\ref{thmB} and the remarks following it in:

\begin{lem}
\label{rescaling}
Given $p\in X$ and a non-decreasing family $(\ep_\beta)$ of positive numbers, the pointed Gromov-Hausdorff limits of the rescaled Kähler-Einstein metric $(X,\ep_\beta^{-1}\hat g_\beta, p)$ coincide with that of the model space $(X,\ep_\beta^{-1} g_\beta, p)$ given in section~\ref{sec:geometry}.
\end{lem}

\begin{proof}[Proof of Lemma~\ref{rescaling}]
The estimate \eqref{eq:64} shows that $\|\partial \bar \partial \varphi(\beta)\|_{g_\beta} \lessapprox \beta^{1-\delta-\ep}$, so that, in particular, $|\hat g_\beta - g_\beta|_{g_\beta} \lessapprox \sqrt \beta$, or equivalently
\[|\ep_\beta^{-1}\hat g_\beta - \ep_\beta^{-1} g_\beta|_{\ep_\beta^{-1} g_\beta} \lessapprox \sqrt \beta,\]
and the lemma follows immediately from the definition of pointed Gromov-Hausdorff convergence. 
\end{proof}

The proof of Theorem~\ref{thmB} is now complete. 
\end{proof}
\bibliographystyle{alpha}
\bibliography{biblio}

\end{document}